\documentclass[UTF-8,reqno]{amsart}
\usepackage{enumerate}
\usepackage{mhequ}
\usepackage[margin=1.2in]{geometry}
\usepackage[tt=false]{libertine}
\usepackage{amssymb,url,color, booktabs,nccmath}
\usepackage{mathrsfs}
\usepackage{enumitem}
\usepackage{graphicx}
\usepackage{tikz}
\usetikzlibrary{shapes,snakes}
\usetikzlibrary{calc}
\usetikzlibrary{decorations.shapes}
\usepackage{color}
\usepackage[colorlinks=true]{hyperref}
\hypersetup{
    linkcolor=blue,          % color of internal links
    citecolor=red,        % color of links to bibliography
    filecolor=blue,      % color of file links
    urlcolor=cyan
}

\definecolor{darkergreen}{rgb}{0.0, 0.5, 0.0}

\numberwithin{equation}{section}

\newcommand{\be}{\begin{eqnarray}}
\newcommand{\ee}{\end{eqnarray}}
\newcommand{\ce}{\begin{eqnarray*}}
\newcommand{\de}{\end{eqnarray*}}
\newtheorem{theorem}{Theorem}[section]
\newtheorem{lemma}[theorem]{Lemma}
\newtheorem{remark}[theorem]{Remark}
\newtheorem{definition}[theorem]{Definition}
\newtheorem{proposition}[theorem]{Proposition}
\newtheorem{Examples}[theorem]{Example}
\newtheorem{corollary}[theorem]{Corollary}

\newenvironment{nouppercase}{%
  \renewcommand{\uppercasenonmath}[1]{}}{}

\def\[{{\Big[}}
\def\]{{\Big]}}
\def\<{{\langle}}
\def\>{{\rangle}}
\def\({{\Big(}}
\def\){{\Big)}}

\def\bx{{\mathbf{x}}}

\def\sgn{\mbox{\rm sgn}}

\def\={&\!\!=\!\!&}

\def\1{{\mathbf{1}}}

\def\geq{\geqslant}
\def\leq{\leqslant}
\def\ge{\geqslant}
\def\le{\leqslant}

\def\[{{\Big[}}
\def\]{{\Big]}}
\def\<{{\langle}}
\def\>{{\rangle}}
\def\({{\Big(}}
\def\){{\Big)}}

\def\bx{{\mathbf{x}}}

\def\sgn{\mbox{\rm sgn}}

\def\={&\!\!=\!\!&}
\def\bt{\begin{theorem}}
\def\et{\end{theorem}}
\def\bl{\begin{lemma}}
\def\el{\end{lemma}}
\def\br{\begin{remark}}
\def\er{\end{remark}}
\def\bx{\begin{Examples}}
\def\ex{\end{Examples}}
\def\bd{\begin{definition}}
\def\ed{\end{definition}}
\def\bp{\begin{proposition}}
\def\ep{\end{proposition}}
\def\bc{\begin{corollary}}
\def\ec{\end{corollary}}

\def\geq{\geqslant}
\def\leq{\leqslant}
\def\ge{\geqslant}
\def\le{\leqslant}

\def\<{\langle} \def\>{\rangle}

\allowdisplaybreaks

\tikzset{
        dot/.style={circle,fill=black,inner sep=0pt, outer sep=0.7pt, minimum size=1mm},
        Phi/.style={white!40!red,thick,snake=coil,segment amplitude=0.6pt, segment length=2pt},
         Z/.style={black!40!green,thick,snake=coil,segment amplitude=0.6pt, segment length=2pt},
        C/.style={thick,black!20!blue},
          Cr/.style={thick,black!20!red},
            Cg/.style={thick,black!20!green},
       }

\begin{document}

\title[Dean-Kawasaki Equation with Singular Interactions]{\LARGE Dean-Kawasaki Equation with  Biot-Savart and Keller-Segel Interactions: Existence and Large Deviations}

\author[Xiaohao Ji]{\large Xiaohao Ji}
\address[X. Ji]{Department of Mathematics, Freie Universit\"at Berlin, Germany}
\email{jixh1020@gmail.com}

\author[Yue Sun]{\large Yue Sun}
\address[Y. Sun]{ Academy of Mathematics and Systems Science, Chinese Academy of Sciences, Beijing 100190, China 
}
\email{sunyue183@mails.ucas.ac.cn}

\author[Zhengyan Wu]{\large Zhengyan Wu}
\address[Z. Wu]{Department of Mathematics, Technische Universit\"at M\"unchen, Boltzmannstr. 3, 85748 Garching, Germany}
\email{wuzh@cit.tum.de}

\begin{abstract}
We establish the existence of probabilistically weak, renormalized kinetic solutions to the Dean--Kawasaki equation with singular interaction kernels, including those of Biot--Savart and Keller--Segel type. Under a suitable regularization of the square-root noise coefficient, we further prove a restricted large deviation principle for probabilistically weak solutions to the regularized Dean--Kawasaki equation. The Biot--Savart and Keller--Segel type interactions introduce a scaling criticality within the $L^1$ framework of the Dean--Kawasaki equation and the associated skeleton equation, which gives rise to a significant new challenge. In contrast to [Fehrman, Gess; Invent. Math., 2023], our large deviation analysis relies on a novel exponential tightness argument specifically adapted to the Dean--Kawasaki noise. This approach, combined with a weak-strong uniqueness result for the associated skeleton equation, allows us to partially overcome the criticality induced by the singular interaction kernel.  
\end{abstract}


\date{\today}

\begin{nouppercase}
\maketitle
\end{nouppercase}

\setcounter{tocdepth}{1}
\tableofcontents

\section{Introduction}\label{sec-1}
Fluctuating hydrodynamics introduces conservative SPDE models governed by fluctuation--dissipation principles, leading to specific formulations of highly irregular noise that encode the fluctuations of underlying microscopic particle systems. The Dean--Kawasaki equation is a prototypical example in this framework, where the noise is designed to capture fluctuations of systems of Brownian particles. In the presence of singular interactions, the Dean--Kawasaki equation can be interpreted as a fluctuating mean-field limit of general interacting particle systems described by mean-field SDEs with singular kernels. The associated large deviation rate function characterizes the asymptotic behavior of rare events for these mean-field systems and is also predicted, at a formal level, by MFT \cite{BDGJL}. Establishing a rigorous large deviation principle for the Dean--Kawasaki equation therefore provides a justification of the rate function conjectured by MFT from the perspective of fluctuating hydrodynamics.

In this paper, we establish the existence and the zero--noise large deviation principle for the Dean--Kawasaki equation with interactions of the form
\begin{align}\label{ker-V}
	\mathcal{V}[\rho] := -(\kappa_1 + \kappa_2 \mathbb{J}) \nabla \mathcal{G} \ast \rho,\quad\quad \kappa_1,\kappa_2\in\mathbb{R}, 
\end{align}
where $\mathcal{G}$ denotes the Green function on the torus $\mathbb{T}^2$:
$$
-\Delta\mathcal{G}\ast\rho=\rho-\int_{\mathbb{T}^2}\rho, 
$$
for some suitable test function $\rho$, and $\mathbb{J}$ is the $\frac{\pi}{2}$--rotation matrix given by 
$$
\mathbb{J} = \begin{bmatrix} 0 & -1 \\ 1 & 0 \end{bmatrix}.
$$
Consequently, $\mathcal{V}$ can be expressed as a linear combination of the Keller--Segel and Biot--Savart interaction kernels. Moreover, $-\kappa_1 > 0$ corresponds to the attractive case, whereas $-\kappa_1 < 0$ corresponds to the repulsive case. In particular, refer to \cite[Lemma 2.17]{BrzezniakFlandoliMaurelli2016Euler}, the interaction potential satisfies
\begin{align}\label{eq:Lp-nablaG}
V:=-(\kappa_1 + \kappa_2 \mathbb{J}) \nabla \mathcal{G} \in L^p_{\mathbb{T}^2}, \quad\text{for every } 1 \leq p < 2.
\end{align} 

More specifically, we first prove the existence of probabilistically weak renormalized kinetic solutions to the Dean--Kawasaki equation with such singular interactions, namely
\begin{equation}\label{SPDE-0-intro}
\partial_t \rho^{\varepsilon} = \Delta \rho^{\varepsilon} - \nabla \cdot \left(\rho^{\varepsilon} \mathcal{V}[\rho^{\varepsilon}]\right) - \sqrt{\varepsilon} \nabla \cdot \bigl(\sqrt{\rho^{\varepsilon}} \circ \xi_K\bigr), \quad (t,x) \in (0,T] \times \mathbb{T}^2,
\end{equation}
where $\xi$ denotes a vector--valued space--time white noise, $\xi_K$ is its ultraviolet cutoff, and $\circ$ indicates the Stratonovich integral.

Building on this existence result, we then investigate the zero--noise large deviation principle for probabilistically weak solutions to a regularized version of \eqref{SPDE-0-intro}, given by
\begin{equation}\label{SPDE-1}
\partial_t \rho^{\eta}
= \Delta \rho^{\eta}
- \nabla \cdot \bigl(\rho^{\eta} \mathcal{V}[\rho^{\eta}]\bigr)
- \sqrt{\varepsilon} \nabla \cdot \bigl(s_{\eta}(\rho^{\eta}) \circ \xi_{K}\bigr),
\quad (t,x) \in (0,T] \times \mathbb{T}^2,
\end{equation}
where $K(\varepsilon) \to \infty$ and $\eta(\varepsilon) \to 0^+$ as $\varepsilon \to 0^+$, and $\{s_{\eta}(\cdot): \eta > 0\}$ is a family of smooth approximations of the square--root function on $[0,\infty)$, namely, the mollifications on the square--root and the white noise are gradually removed as $\varepsilon \to 0$. 

\begin{remark}We list the mollification parameters as follows:
\begin{center}
\renewcommand{\arraystretch}{1.25}
\begin{tabular}{|c|c|c|}
\hline
$\eta$ 
& $K$ 
& $\varepsilon$ \\
\hline
Regularizes $\sqrt{\cdot}$ via $s_\eta$
& Mollifies the noise via $\xi_K=P_K\xi$
& Small-noise parameter \\
\hline
\end{tabular}
\end{center}
The solution to \eqref{SPDE-1} is denoted as $\rho^{\eta}$ to remind that the square--root of the equation is mollified. While these parameters are fixed in the proof of existence, in the derivation of the large deviation principle both $\eta$ and $K$ depend on $\varepsilon$ and consequently, $\rho^\eta$ also depends on $\varepsilon$. 

Throughout, we use the convention $s_{0}=\sqrt{\cdot}$, which allows the existence results to be stated uniformly for $\eta>0$ and $\eta=0$. In particular, under this convention we write
\begin{align*}
\rho^{\eta} = \rho^{\varepsilon}, \quad \text{if} \quad \eta \equiv 0.
\end{align*}
\end{remark}

Our analysis builds on the theory of renormalized kinetic solutions developed in \cite{FG24}, as well as on the large deviation and skeleton equation framework introduced in \cite{FG23}. However, new difficulties arise when treating singular interaction kernels that fall outside the Ladyzhenskaya--Prodi--Serrin class (cf.\ \cite[Assumption A1]{WZ24}), for which the weak convergence approach employed in \cite{FG23, WZ24} no longer applies. To overcome this issue, we establish the large deviation principle via a more direct approach. A key challenge lies in proving exponential tightness for the Dean--Kawasaki equation, which requires new ideas. In addition, we develop a weak--strong uniqueness argument within the renormalized kinetic solution framework to handle the criticality induced by the singular interaction kernel.

We next summarize the main results of the paper in the following subsection.

\subsection{Main Results}
We denote $\Psi(\zeta) := \int^\zeta_0\log(\zeta')\mathrm{d}\zeta'$ and define the space of initial conditions of finite entropy as 
\begin{align}\label{eq:def-Ent-space}
    \mathrm{Ent}(\mathbb{T}^2) := \left\{\rho: L^1(\mathbb{T}^2): \rho \geq 0, \int_{\mathbb{T}^2}\Psi(\rho) < \infty\right\}.
\end{align}

We impose a smallness condition on the attractive ($-\kappa_1>0$) Keller--Segel interaction, namely
\begin{align}\label{eq:small-KS}
-\kappa_1 < \frac{4}{C_{\mathrm{GN}}\|\rho_0\|_{L^1_{\mathbb{T}^2}}},
\end{align}
where $C_{\mathrm{GN}}$ denotes the optimal constant in the Gagliardo--Nirenberg inequality (\cite{CD16}) 
$$
\|f\|_{L^4_{\mathbb{T}^2}} \leq C_{\mathrm{GN}} \|\nabla f\|_{L^2_{\mathbb{T}^2}}^{\frac{1}{2}} \|f\|_{L^2_{\mathbb{T}^2}}^{\frac{1}{2}} + C\|f\|_{L^2_{\mathbb{T}^2}},
$$
which holds for all functions $f$ such that $f \in L^2(\mathbb{T}^2)$ and $\nabla f \in L^2(\mathbb{T}^2)$. Under this assumption, we are now in a position to state the main results of the paper.
The first part of the paper is devoted to the probabilistic construction of weak solutions to \eqref{SPDE-1} in the sense of Definitions~\ref{RKS smooth spde} and~\ref{W SPDE-1}.

\begin{theorem}\label{spde WP}
Assume that \eqref{eq:small-KS} holds. Fix $\varepsilon > 0$ and let $\rho_0 \in \mathrm{Ent}(\mathbb{T}^2)$ be a nonnegative initial datum, where $\mathrm{Ent}(\mathbb{T}^2)$ is defined in \eqref{eq:def-Ent-space}. We consider the following two cases.
\begin{enumerate}
\item For the mollified equation, namely $\eta > 0$ in \eqref{SPDE-1}, there exists a probabilistically weak renormalized kinetic solution $\rho^{\eta}$ to \eqref{SPDE-1} with initial data $\rho_0$. Moreover, $\rho^{\eta}$ coincides with a weak solution of \eqref{SPDE-1}.
\item For the limiting equation \eqref{SPDE-0-intro}, corresponding to $\eta = 0$ in \eqref{SPDE-1}, there exists a probabilistically weak renormalized kinetic solution.
\end{enumerate}
\end{theorem}

Having established the existence of probabilistically weak solutions, we next turn to the study of the small-noise limit and establish a restricted large deviation principle.
To this end, we define the rate function $\mathcal{I} : L_{[0,T]}^1 L_{\mathbb{T}^2}^1 \rightarrow [0,\infty]$ by
\begin{align}\label{I0-intro}
 \mathcal{I}(\rho)= \sup_{\varphi\in C^{\infty}([0,T]\times\mathbb{T}^2)}\Big\{\langle \rho,\varphi \rangle_{L^2_{\mathbb{T}^2}}\Big|^T_0 
 & - \langle \rho,(\partial_t + \Delta)\varphi\rangle_{L_{[0,T]}^2L^2_{\mathbb{T}^2}} 
 - \langle\rho \mathcal{V}[\rho],\nabla\varphi\rangle_{L_{[0,T]}^2L^2_{\mathbb{T}^2}}\nonumber\\
& - \frac{1}{2}\langle \rho\nabla \varphi,\nabla\varphi\rangle_{L_{[0,T]}^2L^2_{\mathbb{T}^2}}\Big\},
\quad \text{if } \rho\in \mathcal{E}_{\mathrm{fin}},
\end{align}
and $\mathcal{I}(\rho)=+\infty$ otherwise.
Here, the domain $\mathcal{E}_{\mathrm{fin}}$ of finite entropy dissipation is defined by
\begin{align}\label{Efin-intro}
\mathcal{E}_{\mathrm{fin}}
:=\Big\{\rho\in L_{[0,T]}^{\infty}L^1_{\mathbb{T}^2} :
\|\nabla\sqrt{\rho}\|_{L_{[0,T]}^{2}L^2_{\mathbb{T}^2}}<\infty,
\ \text{and } \rho\geq 0 \ \text{a.e.}\Big\}.
\end{align}

The noise of the equations \eqref{SPDE-0-intro} and \eqref{SPDE-1} is constructed as a standard way in \cite{FG24,FG23,DFG20}, we briefly summarize it as follows. Let $(e_j)_{j\in\mathbb{Z}^2}$ be the Fourier basis of $L^2_{\mathbb{T}^2}$ and let $(B_j)_{j\in\mathbb{Z}^2}$ be a family of i.i.d. two-dimensional Brownian motions. For every $K\in\mathbb{N}$, we define 
\begin{align}\label{noise-def}
\xi:=\sum_{j\in\mathbb{Z}^2}\dot{B}_je_j,\ \text{and }\ \xi_K:=\sum_{|j|\leq K}\dot{B}_je_j.
\end{align}
We denote
\begin{equation}\label{eq:def-NK}
F_{1,k}:=\sum_{\substack{k\in \mathbb{Z}^{d}\\|k| \leq K}}1,\quad \text{and }\quad N_K := \sum_{\substack{k\in \mathbb{Z}^{d}\\|k| \leq K}}|k|^2 \sim K^{d+2},
\end{equation}
and we will use the It\^o formulations of the equations \eqref{SPDE-0-intro} and \eqref{SPDE-1}: 
\begin{equation}\label{SPDE-0-intro-ito}
\partial_t \rho^{\varepsilon} = \Delta \rho^{\varepsilon} - \nabla \cdot \left(\rho^{\varepsilon} \mathcal{V}[\rho^{\varepsilon}]\right) - \sqrt{\varepsilon} \nabla \cdot \bigl(\sqrt{\rho^{\varepsilon}} \xi_K\bigr)+\frac{\varepsilon F_{1,K}}{8}\nabla\cdot\left(\frac{1}{\rho^{\varepsilon}}\nabla\rho^{\varepsilon}\right), \quad (t,x) \in (0,T] \times \mathbb{T}^2,
\end{equation}
and 
\begin{equation}\label{SPDE-1-ito}
\partial_t \rho^{\eta}
= \Delta \rho^{\eta}
- \nabla \cdot \bigl(\rho^{\eta} \mathcal{V}[\rho^{\eta}]\bigr)
- \sqrt{\varepsilon} \nabla \cdot \bigl(s_{\eta}(\rho^{\eta}) \xi_{K}\bigr)+\frac{\varepsilon F_{1,K}}{2}\nabla\cdot\left(s_{\eta}'(\rho^{\eta})^2\nabla\rho^{\eta}\right),
\quad (t,x) \in (0,T] \times \mathbb{T}^2. 
\end{equation}
With these preparations, we are now in a position to state the main large deviation result. 

\begin{theorem}[Restricted LDP]\label{thm:LDP}
Assume that \eqref{eq:small-KS} holds, and that the initial datum satisfies $\rho_0 \in \mathrm{Ent}(\mathbb{T}^2)$. Let $\{\rho^{\eta(\varepsilon)}\}$ be a family of probabilistically weak solutions to \eqref{SPDE-1} with initial data $\rho_0$. Then, under the scaling regime
\begin{align}
\lim_{\varepsilon\rightarrow 0}K(\varepsilon)=\infty, \quad
\lim_{\varepsilon\rightarrow 0}\varepsilon N_{K(\varepsilon)}=0, \quad
\limsup_{\varepsilon\rightarrow 0}\varepsilon N_{K(\varepsilon)}\|s'_{\eta(\varepsilon)}\|_{L^\infty}^2<\infty, \quad
\lim_{\varepsilon\rightarrow0}\eta(\varepsilon) = 0,
\end{align}
the laws $\mu^{\varepsilon}=\mathbb{P}\circ(\rho^{\eta(\varepsilon)})^{-1}$ satisfy a restricted large deviation principle. More precisely, for every open set $G \subset L_{[0,T]}^{1}L^1_{\mathbb{T}^2}$,
\begin{align}\label{RS E}
\liminf_{\varepsilon\rightarrow0}\varepsilon \log \mu^{\varepsilon}(G)
\geq -\inf_{\rho\in G\cap\mathcal{C}_0}\mathcal{I}(\rho),
\end{align}
where the regularity subclass used for the restriction is defined by
\begin{align}\label{S-intro}
\mathcal{C}_0
= \bigcup_{p'>2}
\left\{\rho\in L_{[0,T]}^{\infty}L^1_{\mathbb{T}^2} :
\rho\in L_{[0,T]}^{\infty}W_{\mathbb{T}^2}^{1,p'}\right\}.
\end{align}
Moreover, for any closed set $F \subset L_{[0,T]}^{1}L^1_{\mathbb{T}^2}$,
\begin{align}\label{LDP LB}
\limsup_{\varepsilon\rightarrow0}\varepsilon \log \mu^{\varepsilon}(F)
\leq -\inf_{\rho\in F}\mathcal{I}(\rho).
\end{align}
\end{theorem}
\begin{remark}
The regularity subclass $\mathcal{C}_0$ is the domain on which the large deviation lower and upper bounds can be matched.
This choice is motivated by the weak--strong uniqueness property of the skeleton equation associated with \eqref{SPDE-1}, which plays a crucial role in the analysis of large deviations.
We emphasize that the class $\mathcal{C}_0$ is not optimal, and whether the restricted large deviation principle can be extended to the full space $L_{[0,T]}^{1}L^1_{\mathbb{T}^2}$ remains an open problem within the methodology developed in this paper. 
\end{remark}

\subsection{Related Literature}\label{literature}
Since the works of Kawasaki \cite{K98} and Dean \cite{D96}, the Dean--Kawasaki equation has been used as a fluctuating hydrodynamic description of empirical densities in interacting particle systems. Its unregularized form with space--time white noise is too singular to be interpreted as an ordinary function-valued SPDE. In the martingale formulation, Konarovskyi, Lehmann and von Renesse \cite{KLvR19,KLR20} identified a rigidity phenomenon: the canonical Dean--Kawasaki martingale problem admits only atomic solutions at the particle levels, which rules out nontrivial absolutely continuous solutions in the naive continuum model. Related perspectives include the measure-valued constructions connected with Wasserstein diffusion \cite{KvR19,D22} and second-order Dean--Kawasaki type models \cite{MRZ25}.\\

For conservative SPDEs with spatially regularized or correlated noise, the modern solution theory is based on entropy estimates and kinetic formulations. The kinetic approach for stochastic scalar conservation laws developed in \cite{LPS13,FG16,GS15} and the theory of nonlinear diffusion with conservative gradient noise in \cite{FG19,DG20} provide important antecedents. In their pioneering work \cite[Theorem~1.1]{FG24}, Fehrman and Gess prove existence, uniqueness, and $L^1$-contraction for stochastic kinetic solutions of conservative Dean--Kawasaki type equations under hypotheses that include the square-root coefficient through the locally $\frac{1}{2}$-H\"older noise regime. Their renormalized kinetic formulation is the starting point for our notion of solution, while the present paper has to add a singular nonlocal drift that is not covered by the local-flux term in \cite{FG24}. Subsequent developments of this direction include the SSEP fluctuation equation of Dirr, Fehrman and Gess \cite{DFG}, whole-space conservative SPDEs \cite{FG25}, non-stationary Stratonovich noise \cite{fehrman2025stochastic}, bounded-domain equations \cite{Shyam25}, and the Vlasov--Fokker--Planck--Dean--Kawasaki system \cite{HWZ25}. In \cite{WWZ22}, existence of probabilistically weak renormalized kinetic solutions is provens under Assumption~(A1), namely a Ladyzhenskaya--Prodi--Serrin condition $V\in L^{r}_tL^p_x$ with $\frac{d}{p}+\frac{2}{r}\leq 1$ and $p>d$; pathwise uniqueness in Theorem~1.1 additionally uses Assumption~(A2), an integrability condition on $\nabla\cdot V$. The Biot--Savart and Keller--Segel kernels in the present paper are genuinely outside this framework. \\

On the other hand, large deviations for Dean--Kawasaki type equations are inseparable from the analytic properties of the skeleton equation. The weak convergence method is a standard probabilistic framework for such problems; see Dupuis and Ellis \cite{DE97} and Budhiraja, Dupuis and Maroulas \cite{BDM11}. Fehrman and Gess \cite[Theorem~1]{FG23} establish well-posedness and stability of the skeleton PDE in the energy-critical entropy space, and \cite[Theorem~39]{FG23} identifies the lower semicontinuous envelope of the zero-range rate function; these results underpin their use of the weak convergence method for conservative SPDEs. The third author and Zhang \cite[Theorem~1.2]{WZ24} apply the same philosophy to singular nonlocal interactions satisfying the same Assumptions~(A1)--(A2) as mentioned above. In the present Biot--Savart/Keller--Segel setting, full uniqueness of the skeleton equation is unavailable at the natural entropy level, so the weak convergence framework does not by itself yield a full large deviation principle.

The restriction in our lower bound is motivated by a broader phenomenon in dynamical large deviations without global skeleton uniqueness. Heydecker \cite{Heydecker23} gives a counterexample for Kac's conservative particle system showing that a predicted global lower bound can fail even when an upper bound and a restricted lower bound are available. Gess, Heydecker and the third author \cite{GHW23} develop a restricted large deviation principle for the three-dimensional Landau--Lifshitz--Navier--Stokes equations in which the matching lower bound is tied to the closure of a weak--strong uniqueness class and to the deterministic energy equality. Wang and the third author \cite{WW25} extend this weak--strong-uniqueness strategy to the forced surface quasi-geostrophic equation in negative Sobolev-type spaces. By contrast, the zero-range result of Gess and Heydecker \cite{GH23} is a positive case where entropy dissipation and skeleton stability lead to a full large deviation principle.

Martini and Mayorcas study an additive-noise approximation to Keller--Segel--Dean--Kawasaki dynamics, which is complementary to the multiplicative conservative model considered here. Their local theory \cite[Theorem~1.2]{AA25} treats a paracontrolled, renormalized solution to a two-dimensional parabolic-elliptic Keller--Segel equation driven by additive noise $\nabla\cdot(\sigma\xi)$; the renormalization is part of the singular SPDE analysis and differs from the conservative square-root fluctuation structure. Their small-noise paper \cite{AA24} obtains law-of-large-numbers, central-limit, and large-deviation results in distribution spaces.\\

The deterministic and particle-system literature on singular interactions is also relevant to the treatment of the drift. Classical mean-field limits for smooth kernels go back to McKean \cite{M67}. Jabin and Wang \cite{JW18} introduced quantitative propagation-of-chaos estimates for kernels controlled in negative Sobolev norms, Bresch, Jabin and Wang \cite{BJW19} applied related modulated-energy ideas to singular kernels including the Patlak--Keller--Segel model, Duerinckx \cite{D16} treated mean-field limits for Riesz interaction gradient flows, and Serfaty \cite{S20} developed a modulated-energy approach for Coulomb-type flows. These works guide the handling of Biot--Savart and Keller--Segel singularities at the deterministic level, but they do not address the conservative square-root noise or the entropy-level skeleton compactness needed here. \\

Several adjacent developments clarify what the present paper does not attempt to optimize. Weak-error estimates for regularized Dean--Kawasaki approximations were proved in \cite{DKP24,DJP25}. Fluctuation and central-limit results for conservative SPDEs and related particle systems include \cite{WZZ21,GWZ24,DFG,CF23}, while central-limit and moderate-deviation questions for stochastic scalar conservation laws were studied by the third author and Zhang \cite{WZ22}. Numerical, structure-preserving, and long-time questions are pursued in \cite{CF23arma,CFIR26,fehrman2022ergodicity,PW25}.

\subsection{Key ideas and technical comments}\label{sec-1-3}
We now explain the proof strategy. The argument has two main components: exponential tightness for the laws of \eqref{SPDE-1}, and direct large-deviation bounds in which the lower bound is restricted to the weak--strong uniqueness class of the skeleton equation.

{\bf Exponential tightness.}
We first establish exponential tightness for weak solutions of \eqref{SPDE-1} in $L^1_{[0,T]\times \mathbb{T}^2}$, equipped with the strong topology. In the spirit of the Aubin-Lions compactness criterion, this requires two key exponential estimates: the exponential entropy dissipation estimate and the exponential time-regularity estimate. 

For the exponential entropy dissipation estimates, we work within the framework of the pathwise entropy dissipation inequality \eqref{eq-2.5}. We then exploit the properties of exponential supermartingales to derive the desired bounds. A key observation in this argument is that the quadratic variation of the martingale term can be absorbed by the Laplacian term, which hinges on the fluctuation-dissipation structure inherent in the Dean-Kawasaki dynamics. As a result, we obtain
\begin{align*}
\limsup_{R\rightarrow \infty}\lim_{\varepsilon\rightarrow 0}\varepsilon \log\mathbb{P}^{\varepsilon}\Big(\|\nabla\sqrt{\rho^{\varepsilon}}\|_{L^2_{[0,T]}L^2_{\mathbb{T}^2}}^2>R\Big)=-\infty. 	
\end{align*}
In contrast with \cite{FG23,FG24}, this provides a new estimate for the Dean--Kawasaki equation. The pathwise entropy inequality \eqref{eq-2.5} restricts the solutions to a specific subclass, in comparison with \cite[Definition 3.2]{FG24}. Together with the preservation of the $L^1_{\mathbb{T}^2}$-mass, this yields an exponential estimate in $L^2_{[0,T]}W^{1,1}_{\mathbb{T}^2}$. 

To apply the Aubin-Lions compactness criterion, it is natural to consider the quantity 
\begin{align*}
\exp\left(\frac{1}{\varepsilon}\|\rho^{\varepsilon}\|_{W^{\alpha,1}_{[0,T]}H^{-\gamma}_{\mathbb{T}^2}}\right)
\end{align*}
for some $\alpha\in(0,1/2)$ and $\gamma>2$. However, this introduces significant challenges. In contrast to \cite{GHW23}, where Gess, Heydecker and the third author established similar exponential time-regularity estimates, their approach crucially relies on the additive structure of the noise and Gaussian properties such as Fernique-type bounds. The multiplicative noise in Dean-Kawasaki dynamics presents additional difficulties: the technique in \cite[Lemma 4.6]{GHW23} no longer applies, and the $W^{\alpha,1}_{[0,T]}H^{-\gamma}_{\mathbb{T}^2}$-norm breaks the martingale structure of the noise. To overcome this, we employ the Taylor expansion of the exponential function and estimate the $p$-moment bound of the $W^{\alpha,1}_{[0,T]}H^{-\gamma}_{\mathbb{T}^2}$-norm for every $p \geq 1$.

{\bf Large deviations by hand.}
The upper bound is established using a standard exponential martingale method, where the exponential tightness plays a crucial role in the analysis. For the lower bound, we adopt the entropy method (see Lemma \ref{entropymethod} for details). Compared to the approach in \cite{FG23}, where Fehrman and Gess employed the classical weak convergence method, our use of the skeleton equation \eqref{ske} follows a logically different route. In the weak convergence framework, it is essential to prove the existence of solutions to the skeleton equation. In contrast, the entropy method involves explicitly constructing a control $g$ such that any given trajectory $\rho \in L^1_{[0,T]}L^1_{\mathbb{T}^2}$ satisfying $\mathcal{I}(\rho) < \infty$ is a weak solution to \eqref{ske} with respect to the control $g$. This allows us to restrict the analysis of the lower bound to trajectories lying within the regularity class $\mathcal{C}_0$.

More specifically, the verification of the entropy method requires the construction of a sequence of probability measures concentrating on an arbitrarily chosen trajectory. This necessitates a compactness argument and a passage to the limit for the skeleton equation \eqref{ske}. However, due to the absence of pathwise uniqueness, the compactness approach alone does not guarantee that the sequence of measures converges to the desired trajectory. To circumvent this difficulty, we verify the entropy method by selecting a trajectory within the class $\mathcal{C}_0$, thereby enabling the use of the weak-strong uniqueness. This is the key reason why we must restrict the lower bound to the class $\mathcal{C}_0$.

\subsection{Criticality of the skeleton equation}\label{sec:supercri}
In the study of large deviations for SPDEs, the associated skeleton equation often plays a central role. For equation \eqref{SPDE-1}, the corresponding skeleton equation is given by
\begin{equation}\label{ske}
\partial_t \rho = \Delta \rho - \nabla \cdot (\rho \mathcal{V}[\rho]) - \nabla \cdot \left( \sqrt{\rho}\, g \right),
\end{equation}
where the choice of control $g \in L^2_{[0,T]}L^2_{\mathbb{T}^2}(\mathbb{R}^2)$ reflects the Cameron--Martin space associated with space-time white noise. %Specifically, a straightforward mathematical analysis shows that both the Biot-Savart and Keller-Segel kernels satisfy $L^p_{\mathbb{T}^2}(\mathbb{R}^2)$-integrability for every $p \in [1,2)$, and we seek to exploit this property. 
We first present a scaling argument to analyze general $L^p$-integrable kernels. For convenience, we extend equation \eqref{ske} to $\mathbb{R}^d$ and consider a general kernel $V \in L^p_{\mathbb{R}^d}(\mathbb{R}^d)$ for some $p \in [1, d)$. For $(x,t) \in \mathbb{R}^d \times [0,T]$, define the rescaled function $\tilde{\rho}(x,t) = \lambda \rho(\iota x, \tau t)$, which formally satisfies
\begin{equation*}
\partial_t \tilde{\rho} = \left( \frac{\tau}{\iota^2} \right) \Delta \tilde{\rho} - \nabla \cdot \left( \tilde{\rho} \, \tilde{V} \ast \tilde{\rho} \right) - \nabla \cdot \left( \sqrt{\tilde{\rho}} \, \tilde{g} \right) \quad \text{in } \mathbb{R}^d \times (0,T), \quad \tilde{\rho}(\cdot, 0) = \lambda \rho_0(\iota \cdot),
\end{equation*}
where $\tilde{V}$ and $\tilde{g}$ are defined as
\begin{align*}
\tilde{V}(x) = \left( \frac{\tau \iota^{d-1}}{\lambda} \right) V(\iota x), \quad \tilde{g}(x,t) = \left( \frac{\tau \lambda^{1/2}}{\iota} \right) g(\iota x, \tau t).
\end{align*}
Following a similar argument to that in \cite{FG23}, the energy-critical space for \eqref{ske} is $L^1_{[0,T]}L^1_{\mathbb{R}^d}$. Preservation of the $L^1_{\mathbb{R}^d}$-norm of the initial data $\tilde{\rho}(0)$ and the $L^2_{[0,T]}L^2_{\mathbb{R}^d}(\mathbb{R}^d)$-norm of $\tilde{g}$ implies the constraints
$$
\frac{\tau}{\iota^2} = 1 \quad \text{and} \quad \lambda = \iota^d.
$$
Consequently, if we consider a generic interaction kernel, the $L^p$-norm of the rescaled kernel satisfies
$$
\|\tilde{V}\|_{L^p} = \iota^{1 - \frac{d}{p}} \|V\|_{L^p},
$$
from which it follows that $p = d$ is the critical index; while for the Coulomb interaction that we consider, $\tilde{V} = \iota V(\iota \cdot) = V$ and the equation is critical. Therefore, relying solely on the $L^p$-integrability ($p\in[1,2)$) of the Biot-Savart and Keller-Segel kernels is quite limited, as it leads to a supercritical regime for the associated PDE.

Despite this limitation, we are still able to prove a weak--strong uniqueness result. Precisely, as long as a renormalized kinetic solution lies in the space $\mathcal{C}_0$, as defined in \eqref{S-intro}, then all renormalized kinetic solutions must coincide with this one.

\subsection{Structure of the paper} 
In Section \ref{sec-2}, we introduce notions of solutions to the SPDEs. In Section \ref{sec:existence}, we prove Theorem \ref{spde WP}, establishing existence of renormalized kinetic solutions to \eqref{SPDE-0-intro} via compactness, and in Section \ref{sec-4} we derive exponential estimates showing exponential tightness of solution laws to \eqref{SPDE-1} in $L^1_{[0,T]\times\mathbb{T}^2}$. 

Section \ref{sec:LB-LDP} is devoted to the lower bound in the large deviation principle (Theorem \ref{thm:LDP}), using the entropy method and relying on weak-strong uniqueness for the skeleton equation (Proposition \ref{L1uniq}, proved in Section \ref{sec-ws!}). The entropy method requires constructing probability measures concentrating on a chosen trajectory. Proposition \ref{C1} shows compactness and passage to the limit for the skeleton equation \eqref{ske} hold. Due to the supercriticality as shown in Section \ref{sec:supercri}, we verify the entropy method on the more regular class $\mathcal{C}_0$ defined in \eqref{S-intro} where the uniqueness of the skeleton equation holds. Section \ref{sec:LDP-UB} completes the proof of Theorem \ref{thm:LDP} by establishing the upper bound via the exponential martingale method and min--max principle.

\subsection{Use of notations}
Let $\|\cdot\|_{L^p}$ denote the norm in the Lebesgue space $L^p_{\mathbb{T}^2}$ for $p \in [1,\infty]$, $C^{\infty}(\mathbb{T}^2 \times (0,\infty))$ denote the space of infinitely differentiable functions on $\mathbb{T}^2 \times (0,\infty)$ and $C^{\infty}_c(\mathbb{T}^2 \times (0,\infty))$ be the subspace of such functions with compact support. For a non-negative integer $k$ and $p \in [1,\infty]$, we denote by $W^{k,p}_{\mathbb{T}^2}$ the usual Sobolev space on $\mathbb{T}^2$ and we write $H^{k}_{\mathbb{T}^2} = W^{k,2}_{\mathbb{T}^2}$,
and denote by $H^{-k}_{\mathbb{T}^2}$ the topological dual of $H^k_{\mathbb{T}^2}$. Furthermore, $C^k_{\mathrm{loc}}(\mathbb{R})$ denotes the space of $k$-times differentiable functions endowed with the topology of local convergence.

The notation $a \lesssim b$ for every $a,b \in \mathbb{R}$ means that $a \leq C\, b$ for some constant $C > 0$ independent of the relevant parameters. We use the letter $C$ to denote a generic constant that can be explicitly computed in terms of known quantities; its value may change from line to line.

\section{Preliminaries}\label{sec-2}
%According to \cite[Lemma 3.3]{WWZ22}, it holds $\nabla\rho=2\sqrt{\rho}\nabla\sqrt{\rho}$ as distributions.
We first provide definitions of solutions. 
\begin{definition}(Renormalized kinetic solution)\label{RKS smooth spde}
	Let the interaction $\mathcal{V}[\cdot]$ be defined by \eqref{ker-V} and $\rho_0\in L^{1}_{\mathbb{T}^2}$. For any $\varepsilon \in (0,1)$ and $\eta \in[0,1)$, a renormalized kinetic solution $\rho$ of \eqref{SPDE-1} with initial datum $\rho_0$ is a nonnegative, almost surely continuous $L^1_{\mathbb{T}^2}$-valued $\mathcal{F}_t$-predictable process $\rho \in L^1_{\Omega \times[0, T]}L^1_{\mathbb{T}^2}$ that satisfies the following properties.

	\begin{enumerate}
		\item Conservation of mass: almost surely for every $t\in[0,T]$,
		\begin{equation}\label{eq-2.4}
		\|\rho(\cdot,t)\|_{L^{1}_{\mathbb{T}^{2}}}=\left\|\rho_0\right\|_{L^{1}_{\mathbb{T}^{2}}}.
		\end{equation}
		%\item Local regularity: for every $K\in\mathbb{N}$,
		%\begin{equation}
		%[(\rho\wedge K)\vee1/K]\in L^{2}\left(\Omega;L^{2}\left([0,T];H^{1}\left(\mathbb{T}^{d}\right)\right)\right).\notag
		%\end{equation}
		\item Pathwise entropy dissipation estimates: almost surely, 
		\begin{align}\label{eq-2.5}
\int_0^t\int_{\mathbb{T}^2}
|\nabla\sqrt{\rho}|^2\mathrm{d}x\mathrm{d}s
\lesssim&\int_{\mathbb{T}^2}\Psi(\rho_0)\mathrm{d}x+\sqrt{\varepsilon}\int_0^t\int_{\mathbb{T}^2}
2\nabla\sqrt{\rho}\cdot
\xi_{K}
+\varepsilon N_KT+1,
\end{align}
for all $t \in [0,T]$.  
%		\begin{equation}
%		\mathbb{E}\int_{0}^{T}\int_{\mathbb{T}^{2}}\left|\nabla \sqrt\rho\right|^{2}\le c\left(T, d,\|\hat{\rho}\|_{L^1_{\mathbb{T}^2}}\right).
%		\end{equation}
		%\item We have
	%	\begin{align}
		%	&\nabla \sqrt\rho\in L^2\left(\Omega\times[0,T];L^2(\mathbb{T}^2)\right), \notag\\
		%	&\nabla\sqrt{\rho} V\ast\rho\in L^1\left(\Omega\times[0,T];L^1(\mathbb{T}^2)\right), \notag\\
		%	&V\ast(\nabla\rho)\in L^1\left(\Omega\times[0,T];L^1(\mathbb{T}^2)\right).
	%	\end{align}
	
		Furthermore, there exists a nonnegative kinetic measure $q$ (\cite[Definition 3.1]{FG24}) satisfying the following properties. 
		\item Regularity: almost surely
		\begin{align}\label{eq-2.66}
		4\delta_{0}(\xi-\rho)\xi|\nabla \sqrt{\rho}|^{2}\le q\quad {\rm{on}}\ \mathbb{T}^2\times(0,\infty)\times[0,T].
		\end{align}
		\item Vanishing at infinity: we have
		\begin{align}\label{eq-2.7}
		\liminf_{M\to \infty}\mathbb{E}\Big[q(\mathbb{T}^2\times [0,T]\times [M,M+1])\Big]=0.
		\end{align}
		%\item Decay at 0: we have
%		\begin{align}\label{eq-2.8}
%		\lim_{M\to \infty}M\mathbb{E}\Big[q(\mathbb{T}^2\times [0,T]\times [1/M,2/M])\Big]=0.
%		\end{align}
		\item The equation: for every $\psi\in\mathrm{C}_{c}^{\infty}\left(\mathbb{T}^{2}\times(0,\infty)\right)$, almost surely for almost every $t\in[0,T]$,
		\begin{equation}
 \begin{aligned}
&\int_{\mathbb{R}}\int_{\mathbb{T}^2} \chi(x, \xi, \cdot) \psi(x, \xi) \Big|^t_0= -\int_{0}^{t} \int_{\mathbb{T}^2}\nabla\rho \cdot\nabla_x\psi(x,\rho)
-\int_{0}^{t} \int_{\mathbb{R}} \int_{\mathbb{T}^2} \partial_{\xi} \psi(x, \xi) \mathrm{d}m\\
&
-\int_{0}^{t} \int_{\mathbb{T}^2} \nabla\cdot \left(\rho \mathcal{V}[\rho]\right)\psi(x, \rho)
+\sqrt{\varepsilon}\int_{0}^{t} \int_{\mathbb{T}^2} s_{\eta}(\rho)\xi_{K}\cdot \nabla \psi(x, \rho)\\
&-\frac{\varepsilon}{2}\int^t_0\int_{\mathbb{T}^2}F_{1,k}s_{\eta}'(\rho)^2\nabla\rho\cdot(\nabla\psi)(x,\rho)+\frac{\varepsilon N_K}{2}\int_{0}^{t} \int_{\mathbb{T}^2}[s_{\eta}(\rho)]^2 \partial_{\xi} \psi(x, \rho),
\end{aligned}
			\end{equation}
				where $\chi$ is the kinetic function defined by $\chi(x,\xi,t)=\mathbf{1}_{\{0<\xi<\rho(x,t)\}}$ and $N_K$ is defined in \eqref{eq:def-NK}.
	\end{enumerate}	
\end{definition}

\begin{definition}\label{W SPDE-1}
Let $\mathcal{V}[\cdot]$ and $\rho_0$ be defined as in Definition \ref{RKS smooth spde}, for any $\varepsilon,\eta\in(0,1)$, a weak solution  $\rho$ of \eqref{SPDE-1} with initial data $\rho(\cdot, 0)=\rho_0$ is a nonnegative, $\mathcal{F}_{t}$-predictable process that satisfies the following properties.
\begin{enumerate}
		\item Conservation of mass: almost surely for every $t\in[0,T]$,
		\begin{equation}
		\|\rho(\cdot,t)\|_{L^{1}_{\mathbb{T}^{2}}}=\left\|\rho_0\right\|_{L^{1}_{\mathbb{T}^{2}}}.
		\end{equation}
		%\item Local regularity: for every $K\in\mathbb{N}$,
		%\begin{equation}
		%[(\rho\wedge K)\vee1/K]\in L^{2}\left(\Omega;L^{2}\left([0,T];H^{1}\left(\mathbb{T}^{d}\right)\right)\right).\notag
		%\end{equation}
		\item Pathwise entropy dissipation estimates: almost surely, 
		\begin{align}\label{eq-2.5-2}
 \int_0^t\int_{\mathbb{T}^2}
|\nabla\sqrt{\rho}|^2\mathrm{d}x\mathrm{d}s
\lesssim&\int_{\mathbb{T}^2}\Psi(\rho_0)\mathrm{d}x+\sqrt{\varepsilon}\int_0^t\int_{\mathbb{T}^2}
\frac{s_{\eta}(\rho)}{\rho}\nabla\rho\cdot
\xi_{K}
+\varepsilon N_KT+1,
\end{align}
for all $t \in [0,T]$.  
        
\item The equation: 
for every $\psi \in C^{\infty}_{\mathbb{T}^2}$, almost surely for every $t \in[0, T]$,
\begin{align}\label{eq:kin_DK_reg}
\int_{\mathbb{T}^2} \rho(x, t) \psi(x) & =\int_{\mathbb{T}^2} \rho_0 \psi-\int_{0}^{t} \int_{\mathbb{T}^2} \nabla \rho \cdot \nabla \varphi+\int_{0}^{t} \int_{\mathbb{T}^2}(\rho \mathcal{V}[\rho]) \cdot \nabla \psi \nonumber\\
&+\sqrt{\varepsilon}\int_{0}^{t} \int_{\mathbb{T}^2} s_{\eta}(\rho) \nabla \psi \cdot \xi_{K}-\frac{\varepsilon F_{1,K}}{2} \int_{0}^{t} \int_{\mathbb{T}^2} \left[s_{\eta}^{\prime}(\rho)\right]^{2} \nabla \rho \cdot \nabla \psi. 
\end{align}
\end{enumerate}	
\end{definition}
We emphasize that the renormalized kinetic solution is well-defined even for $\eta = 0$, while when we use the usual notion of weak solutions, we can only consider $\eta > 0$, namely mollified square--root function. The correction term in \eqref{W SPDE-1} may not be well-defined when $\eta = 0$, as 
\begin{align*}
[s_{\eta}^{\prime}(\rho)]^{2} \nabla \rho = \nabla \log(\rho),
\end{align*}
which may not even be a well--defined distribution, even if we assume $\rho \in L^2_{[0,T]}W^{1,1}_{\mathbb{T}^2}$. %When $\eta>0$, the above two definitions are equivalent, related discussion can be found in \cite[Proposition 5.17]{FG24}. 

In the following, we show that the interaction kernel terms are well-defined in Definition \ref{RKS smooth spde} and \ref{W SPDE-1}. Indeed, since 
\begin{align}\label{L2-ES}
\|\rho\|_{L^p_{\mathbb{T}^2}} = \|\sqrt{\rho}\|^2_{L_{\mathbb{T}^2}^{2p}} \lesssim \|\nabla\sqrt{\rho}\|^{2 -2/p}_{L_{\mathbb{T}^2}^{2}}\|\rho\|^{1/p}_{L^1_{\mathbb{T}^2}}+\|\rho\|_{L^1_{\mathbb{T}^2}},
\end{align}
for any $1\leq p<\infty$ and $\rho \geq 0$, by Young's inequality $$\|\mathcal{V}[\rho]\|_{L^{2}_{\mathbb{T}^2}} \lesssim \|\nabla\mathcal{G}\|_{L^1_{\mathbb{T}^2}}\|\rho\|_{L^2_{\mathbb{T}^2}}\lesssim \|\nabla\mathcal{G}\|_{L^1_{\mathbb{T}^2}}\|\nabla\sqrt{\rho}\|_{L_{\mathbb{T}^2}^{2}}\|\rho\|^{1/2}_{L^1_{\mathbb{T}^2}}+\|\nabla\mathcal{G}\|_{L^1_{\mathbb{T}^2}}\|\rho\|_{L^1_{\mathbb{T}^2}}.$$ 
By H\"older's inequality, 
\begin{align}\label{interpolation}
\|\rho\mathcal{V}[\rho]\|_{L^1_{\mathbb{T}^2}}\lesssim \|\rho\|_{L^2_{\mathbb{T}^2}}\|\mathcal{V}[\rho]\|_{L^2_{\mathbb{T}^2}}\lesssim \|\nabla\mathcal{G}\|_{L^1_{\mathbb{T}^2}}\|\nabla\sqrt{\rho}\|_{L_{\mathbb{T}^2}^{2}}^2\|\rho\|_{L^1_{\mathbb{T}^2}}+\|\nabla\mathcal{G}\|_{L^1_{\mathbb{T}^2}}\|\rho\|_{L^1_{\mathbb{T}^2}}^2.
\end{align}

Furthermore, 
\begin{align}\label{eq:interaction is L1}
 \left\|\nabla\cdot \left(\rho \mathcal{V}[\rho]\right)\psi(x,\rho)\right\|_{L^1_{\mathbb{T}^2}} & \lesssim \left\|\sqrt{\rho} \nabla\sqrt{\rho} \cdot \mathcal{V}[\rho]\psi(x,\rho)\right\|_{L^1_{\mathbb{T}^2}} + \|\rho^2\psi(x,\rho)\|^2_{L^1_{\mathbb{T}^2}} \nonumber \\
 & \leq C(\psi) \|\nabla \sqrt{\rho}\|^{2}_{L^2_{\mathbb{T}^2}}\|\rho\|^{1/2}_{L^1_{\mathbb{T}^2}}+C(\psi)\|\nabla \sqrt{\rho}\|_{L^2_{\mathbb{T}^2}}\|\nabla\mathcal{G}\|_{L^1_{\mathbb{T}^2}}\|\rho\|_{L^1_{\mathbb{T}^2}} + C(\psi)\notag\\
 & \leq C(\psi) \|\nabla \sqrt{\rho}\|^{2}_{L^2_{\mathbb{T}^2}}\|\rho\|^{1/2}_{L^1_{\mathbb{T}^2}}+C(\psi,\|\nabla\mathcal{G}\|_{L^1_{\mathbb{T}^2}},\|\rho\|_{L^1_{\mathbb{T}^2}})
\end{align}
Therefore the integrals in Definition \ref{RKS smooth spde} and \ref{W SPDE-1} are well-defined. 

\section{The existence}\label{sec:existence}
In this section, we show the existence of a renormalized kinetic solution of the Dean--Kawasaki equation \eqref{SPDE-0-intro} with the full square--root, which we recall to be 
\begin{align*}
\partial_t\rho^{\varepsilon}=\Delta\rho^{\varepsilon}-\nabla\cdot\left(\rho^{\varepsilon}\mathcal{V}[\rho^{\varepsilon}]\right) - \sqrt{\varepsilon}\nabla\cdot(\sqrt{\rho^{\varepsilon}}\circ\xi_{K}), 	\ \ (t,x)\in(0,T]\times\mathbb{T}^2, 
\end{align*}
namely, we prove the following theorem of existence: 
\begin{theorem}\label{spde WP-existence}
  For every fixed $\varepsilon > 0$ and $K \in \mathbb{N}$, we let 
 \begin{align}\label{eq:def-Ent}
 \rho_0 \in \mathrm{Ent}(\mathbb{T}^2) := \left\{\rho \in L^1_{\mathbb{T}^2}: \int_{\mathbb{T}^2} \Psi(\rho) < \infty \right\}.
\end{align} 
Then there exists a probabilistically weak nonnegative renormalized kinetic solution $\rho^{\varepsilon}$ of \eqref{SPDE-0-intro} in the sense of Definition \ref{RKS smooth spde} with initial data $\rho_0$.
\end{theorem}

We will therefore fix $\varepsilon > 0$ and $K\in\mathbb{N}$ in this section. In order to show the existence of the probabilistically weak renormalized kinetic solution to \eqref{SPDE-1}, we introduce a family of stochastic PDEs with regularized coefficients as approximations. The mollification of the square--root is chosen as follows: Let $s_{\eta}:[0, \infty) \rightarrow[0, \infty)$ be a family of smooth non-decreasing functions such that $s_{\eta}(0)=0$, 
\begin{align}
0 \leqslant s_{\eta}(\zeta) \lesssim \sqrt{\zeta}, 0 \leq s'_{\eta}(\zeta) \lesssim \frac{1}{\sqrt{\zeta}} \text {, and } s_{\eta} \to \sqrt{\cdot} \ \text{uniformly on compacts}.
\end{align}
We also regularize the singular interaction term, that is, we let $\kappa_\gamma( x):=\frac{1}{\gamma^2} \kappa ( \frac{x}{\gamma})$ be the standard convolution kernel on $ \mathbb{T}^2$ and 
\begin{align}
\mathcal{V}_{\gamma}[\rho] := \kappa_{\gamma}\ast\mathcal{V}[\rho] =  -(\kappa_1 + \kappa_2 \mathbb{J}) \nabla (\kappa_{\gamma}\ast\mathcal{G}) \ast \rho,
\end{align}
and consider the regularized equation with $V_{\gamma}[\cdot]$ and $s_{\eta}(\cdot)$, i.e. 
\begin{align}\label{smooth spde}
\partial_t\rho^{\gamma, \eta}=\Delta\rho^{\gamma, \eta} -\nabla\cdot(\rho^{\gamma, \eta}\mathcal{V}_{{\gamma}}[\rho^{\gamma, \eta}])
- \sqrt{\varepsilon} \nabla\cdot(s_{\eta}(\rho^{\gamma, \eta})\xi_{K})
+\frac{\varepsilon F_{1,K}}{2}\nabla\cdot([s'_{\eta}(\rho^{\gamma, \eta})]^2\nabla\rho^{\gamma, \eta}).
\end{align}
We omit the dependence of the solution of \eqref{smooth spde} on $\varepsilon$ and $\eta$ in the notation and recall \cite[Theorem 5.6 and Proposition 5.9]{WWZ22}, from which the well--posedness of \eqref{smooth spde} immediately follows.
\begin{proposition}\label{entropy-mollified-DK-existence}
 For any fixed $\varepsilon,\eta,\gamma\in(0,1)$ and $0 \leq \rho_0 \in \mathrm{Ent}(\mathbb{T}^2)$, there exists a unique probabilistically strong nonnegative solution $\rho^{\gamma, \eta}$ to \eqref{smooth spde} in the sense of Definition \ref{W SPDE-1} with initial data $\rho_0$, which is also a renormalized kinetic solutions of \eqref{smooth spde} in the sense of Definition \ref{RKS smooth spde}.\end{proposition}
 
\subsection{Uniform estimates}
We first establish a uniform entropy dissipation estimate. While this type of estimate is standard in the context of Dean--Kawasaki SPDEs, special attention is required here for the control of the kernel term. 
\begin{proposition}\label{entropy estimate}
 Let the interaction $\mathcal{V}[\cdot]$ be defined by \eqref{ker-V} under the condition that \eqref{eq:small-KS} holds. Given $\rho_0\in \mathrm{Ent}(\mathbb{T}^2)$ and the solution $\rho^{\gamma, \eta}$ as in Proposition \ref{entropy-mollified-DK-existence}, we have
\begin{align}\label{entropy0}
\mathbb{E}\left(\int_{\mathbb{T}^2}\Psi(\rho^{\gamma, \eta}(T))\right) + \sigma\mathbb{E}\left(\int_0^T\|\nabla\sqrt{\rho^{\gamma, \eta}}\|_{L^2_{\mathbb{T}^2}}^2\mathrm{d}s\right)\lesssim \int_{\mathbb{T}^2}\Psi(\rho_0) + \varepsilon N_KT, 
\end{align}
for some $\sigma > 0$ which is independent of $\gamma,\eta,\varepsilon$ and $K$.
\end{proposition}
\begin{proof}
For any fixed $\delta\in (0,1)$, we define $\Psi_{\delta} : [0,\infty) \rightarrow \mathbb{R} $ be the unique smooth function satisfying $\Psi_{\delta}(0) =
0 $ and $\Psi'_{\delta}(\xi) = \log(\xi + \delta)$. It follows from the It\^{o}'s formula and the nonnegativity of $\rho^{\varepsilon}$ that
\begin{align}\label{PW ET q1}
&\int_{\mathbb{T}^2}\Psi_{\delta}(\rho^{\gamma, \eta}) \Big|^T_0 + 4\int_0^T\int_{\mathbb{T}^2}\frac{\rho^{\gamma, \eta}}{\rho^{\gamma, \eta}+\delta}
|\nabla\sqrt{\rho^{\gamma, \eta}}|^2
= \int_0^T\int_{\mathbb{T}^2}\frac{\rho^{\gamma,\eta}}{\rho^{\gamma, \eta}+\delta}\mathcal{V}_{{\gamma}}[\rho^{\gamma, \eta}]\cdot\nabla\rho^{\gamma, \eta}
\nonumber\\
& + 2\sqrt{\varepsilon}\int_0^T \int_{\mathbb{T}^2}
\frac{s_{\eta}(\rho^{\gamma, \eta})\sqrt{\rho^{\gamma, \eta}}}{\rho^{\gamma, \eta}+\delta}\nabla\sqrt{\rho^{\gamma, \eta}} \cdot \xi_{K} + \frac{\varepsilon N_K}{2}\int_{0}^{T}\int_{\mathbb{T}^2}\frac{ [s_{\eta}(\rho^{\gamma, \eta})]^2}{\rho^{\gamma, \eta}+\delta}.
\end{align}
For the interaction term, we have in particular
\begin{align*}
\int_0^T\int_{\mathbb{T}^2}\frac{\rho^{\gamma, \eta}}{\rho^{\gamma, \eta}+\delta}\mathcal{V}_{{\gamma}}[\rho^{\gamma, \eta}]\cdot\nabla\rho^{\gamma, \eta} = -\kappa_1 \int_0^T \int_{\mathbb{T}^2}\frac{\rho^{\gamma, \eta}}{\rho^{\gamma, \eta}+\delta} \nabla (\kappa_{\gamma}\ast\mathcal{G}) \ast \rho^{\gamma, \eta}\cdot\nabla\rho^{\gamma, \eta},
\end{align*}
where the term with Biot--Savart kernel vanishes, since $\mathbb{J} \nabla (\kappa_{\gamma}\ast\mathcal{G} \ast \rho^{\gamma, \eta})$ is divergence--free. On the other hand, as $\rho^{\gamma, \eta}$ preserves the $L^1_{\mathbb{T}^2}$-mass and $\nabla \sqrt{\rho^{\gamma, \eta}} \in L^2_{\Omega}L^2_{[0,T] \times \mathbb{T}^2}$, and we have 
\begin{align*}
\left|\frac{\rho^{\gamma, \eta}}{\rho^{\gamma, \eta}+\delta} (\nabla\kappa_{\gamma}\ast\mathcal{G}) \ast \rho^{\gamma, \eta}\cdot\nabla \rho^{\gamma, \eta}\right| \lesssim \|\kappa_{\gamma}\|_{L^{\infty}_{\mathbb{T}^2}}\|\rho^{\gamma, \eta}\|_{L^{\infty}_{[0,T]}L^1_{\mathbb{T}^2}} |\sqrt{\rho^{\gamma, \eta}} \nabla \sqrt{\rho^{\gamma, \eta}}|,
\end{align*}
it follows from the dominated convergence theorem and the integration by parts formula that  
\begin{align}\label{ker-es-1}
\lim_{\delta \to 0} \mathbb{E}\int_0^T\int_{\mathbb{T}^2} - \frac{\rho^{\gamma, \eta}}{\rho^{\gamma, \eta}+\delta}\nabla (\kappa_{\gamma}\ast\mathcal{G}) \ast \rho^{\gamma, \eta}\cdot\nabla\rho^{\gamma, \eta} =  -\mathbb{E} \int_0^T\int_{\mathbb{T}^2} (\kappa_{\gamma} \ast \rho^{\gamma, \eta}) \rho^{\gamma, \eta}+\|\rho_0\|_{L^1_{\mathbb{T}^2}}^2. 
\end{align}
In case of repulsive kernels ($\kappa_1 > 0$), we can therefore neglect the contribution from the first term of \eqref{ker-es-1}; otherwise we bound
\begin{align}\label{ker-es-2}
\sup_{\gamma > 0}\left|\int_0^T\int_{\mathbb{T}^2} (\kappa_{\gamma} \ast \rho^{\gamma, \eta}) \rho^{\gamma, \eta} \right| \leq \|\rho^{\gamma, \eta}\|^2_{L^2_{[0,T]}L^2_{\mathbb{T}^2}} \leq C_{\mathrm{GN}}\|\nabla\sqrt{\rho^{\gamma, \eta}}\|_{L^2_{[0,T]}L^2_{\mathbb{T}^2}}^2\|\rho_0\|_{L^1_{\mathbb{T}^2}}+C(\rho_0),
\end{align}
where the last inequality is due to Sobolev embedding and $L^1$ conservation. By taking expectation, passing the limit $\delta \to 0$ in \eqref{PW ET q1} and using the condition \eqref{eq:small-KS}, we arrive at the desired bound \eqref{entropy0}.
\end{proof}

We remark that the solution $\rho^{\gamma, \eta}$ to \eqref{smooth spde} depends on the parameters $\gamma$ (mollification of interaction), $\eta$ (mollification of the square--root) and $\varepsilon$. For the existence of solution of \eqref{SPDE-0-intro}, we now fix $\varepsilon$ and send $\gamma, \eta \to 0$ in this section. With the entropy estimate at hand, we aim at proving the tightness of the approximating solutions 
\begin{align*}
\{\rho^{\gamma, \eta}\}_{\gamma,\eta\in(0,1)} \subset L^1_{[0, T]} L^1_{\mathbb{T}^2}
\end{align*}
using the Aubins-Lions-Simon compactness criterion (\cite[Corollary 5]{Sim87}). As already noticed in \cite{FG24}, it is not apparent at all how to derive temporal regularity of the solution due to the It\^o--Stratonovich correction term, for exactly the same reason that the weak solution of \eqref{SPDE-0-intro} is not well--defined. We therefore follow the idea of \cite{FG24} to introduce a smooth truncation function (at zero) $h_\delta$ as follows: For every $\delta\in(0,1)$, let $\psi_{\delta}\in C^{\infty}([0,\infty))$ be a smooth nondecreasing function satisfying 
\begin{align}
0 \leq \psi_{\delta} \leq 1, \quad \psi_{\delta}(\xi)=1, \ \text{if} \ \xi\ge\delta, \quad \psi_{\delta}(\xi)=0, \ \text{if} \ \xi\le \frac{\delta}{2}; \quad h_{\delta}(\xi) := \psi_{\delta}(\xi) \xi.
\end{align} 
and $\left|\psi_{\delta}'(\xi)\right|\lesssim \delta^{-1}$. It follows directly from the construction that $h_{\delta}'(\xi)=\psi_{\delta}'(\xi)\xi+\psi_{\delta}(\xi)$ is uniformly bounded in $\delta > 0$, and since 
\begin{align*}
\nabla \left[h_{\delta}(\rho^{\gamma, \eta})\right] = 2h'_{\delta}(\rho^{\gamma, \eta}) \sqrt{\rho^{\gamma, \eta}} \nabla \sqrt{\rho^{\gamma, \eta}},
\end{align*}
the entropy estimate \eqref{entropy0} and $L^1$ conservation directly lead to the apriori bound
\begin{align}\label{eq:entropy-truncated}
\sup_{\eta > 0, \gamma > 0}\mathbb{E}\left(\|h_{\delta}(\rho^{\gamma, \eta})\|^2_{L^{2}_{[0,T]}W^{1,1}_{\mathbb{T}^2}}\right) < \infty.
\end{align}
On the other hand, as the singularity of the correction term at zero is truncated, we have the following the temporal regularity estimates of $h_{\delta}(\rho^{\gamma, \eta})$, which is therefore $\delta$--dependent. 
\begin{lemma}\label{spde E T2}Under the same setting as in Proposition \ref{entropy estimate}, for every $\delta\in(0,1)$ fixed, $\beta > 2$ and $0 < \alpha <1/2$, we have
\begin{align}	
\sup_{\eta >0, \gamma > 0}\mathbb{E}\left(\|h_{\delta}(\rho^{\gamma, \eta})\|_{W^{\alpha,1}_{[0,T]}H^{-\beta}_{\mathbb{T}^2}}\right) < \infty
\end{align}
\end{lemma}
\begin{proof}
By It\^o's formula, it holds in the sense of distribution that
\begin{align*}
h_{\delta}(\rho^{\gamma, \eta}) \Big|_0^t
& = \int_{0}^{t} \nabla \cdot \big(h'_{\delta}(\rho^{\gamma, \eta}) \nabla \rho^{\gamma, \eta}\big)
 - \int_{0}^{t} h''_{\delta}(\rho^{\gamma, \eta}) \, |\nabla \rho^{\gamma, \eta}|^{2} \\
& + \frac{\varepsilon F_{1,K}}{2}\int_{0}^{t} \nabla \cdot \big(h'_{\delta}(\rho^{\gamma, \eta}) [s'_{\eta}(\rho^{\gamma, \eta})]^{2} \nabla \rho^{\gamma, \eta}\big)-\frac{\varepsilon F_{1,K}}{2}\int^t_0h_{\delta}''(\rho^{\gamma,\eta})[s_{\eta}'(\rho^{\gamma,\eta})]^2|\nabla\rho^{\gamma,\eta}|^2\\
& + \frac{\varepsilon N_K}{2} \int_{0}^{t} h''_{\delta}(\rho^{\gamma, \eta}) s_{\eta}^{2}(\rho^{\gamma, \eta}) + \sqrt{\varepsilon}\int_{0}^{t} h'_{\delta}(\rho^{\gamma, \eta}) \nabla \cdot \left[ s_{\eta}(\rho^{\gamma, \eta})  \rho^{\gamma, \eta} \xi_K\right]\\
& -\int_{0}^{t} h'_{\delta}(\rho^{\gamma, \eta})\nabla\cdot(\rho^{\gamma, \eta}\mathcal{V}_{{\gamma}}[\rho^{\gamma, \eta}]). 
\end{align*}
Following the proof of Proposition 5.13 in \cite{WWZ22}, it suffices to estimate the interaction term, which we rewrite as 
\begin{align*}
-\int_{0}^{t}\nabla \cdot \left(h_{\delta}'(\rho^{\gamma, \eta}) \rho^{\gamma, \eta} \mathcal{V}_{\gamma}[\rho^{\gamma, \eta}]\right)+\int_{0}^{t} h_{\delta}''(\rho^{\gamma, \eta}) \nabla \rho^{\gamma, \eta}\cdot \left(\rho^{\gamma, \eta} \mathcal{V}_{\gamma}[\rho^{\gamma, \eta}]\right) =: (\mathrm{I}) + (\mathrm{II}).
\end{align*}
Since $\beta > 2$ and $\|h'_{\delta}\|_{L^{\infty}} < \infty$, we have by the same bound as in \eqref{eq:interaction is L1} that
\begin{align*}
\sup_{\eta, \gamma > 0}\left\|(\mathrm{I})\right\|_{W_{[0,T]}^{1,1}H^{-\beta}_{\mathbb{T}^2}} \lesssim \sup_{\eta, \gamma > 0}\left\| \mathcal{V}_{\gamma}[\rho^{\gamma, \eta}]\rho^{\gamma, \eta}\right\|_{L^1_{[0,T]}L_{\mathbb{T}^2}^1} \lesssim \sup_{\eta, \gamma > 0}\left\|\nabla\sqrt{\rho^{\gamma, \eta}}\right\|_{L^2_{[0,T]}L_{\mathbb{T}^2}^{2}}\|\rho^{\gamma, \eta}_0\|^{1/2}_{L^1_{\mathbb{T}^2}}+C(\rho_0),
\end{align*}
while since $\zeta \mapsto \zeta^{3/2} h''_{\delta}(\zeta)$ is bounded, 
\begin{align*}
\sup_{\eta, \gamma > 0}\left\|(\mathrm{II})\right\|_{W_{[0,T]}^{1,1}H^{-\beta}_{\mathbb{T}^2}} \lesssim \sup_{\eta, \gamma > 0}\left\| \mathcal{V}_{\gamma}[\rho^{\gamma, \eta}]\nabla \sqrt{\rho^{\gamma, \eta}}\right\|_{L^1_{[0,T]}L_{\mathbb{T}^2}^1} \lesssim \|\nabla \sqrt{\rho^{\gamma, \eta}}\|^{2}_{L^2_{[0,T]}L^2_{\mathbb{T}^2}}\|\rho^{\gamma, \eta}_0\|^{1/2}_{L^1_{\mathbb{T}^2}}+C(\rho_0).
\end{align*}
\end{proof}

Applying Aubin--Lions-Simon with \eqref{eq:entropy-truncated} and Lemma \ref{spde E T2}, with the help of the equivalence characterization of the $L^1_{[0,T]}L^1_{\mathbb{T}^2}$-convergence \cite[Definition 5.19 and Lemma 5.20]{FG24}, we can now derive the desired tightness result in $L^1_{[0,T]}L^1_{\mathbb{T}^2}$. 

In the following, we establish tightness for the martingale term. This argument is analogous to \cite[Proposition 5.23]{FG24}, and therefore we omit the proof. 

\begin{proposition}\label{T rho eta}
Under the same setting as in Proposition \ref{entropy estimate} and for every $\varepsilon \in (0,1)$ fixed, the laws of $\left\{\rho^{\gamma, \eta}\right\}_{\eta ,\gamma \in(0,1)}$ are tight on $L^1_{[0,T]}L^1_{\mathbb{T}^2}$ with the strong topology. Furthermore, for any $\alpha \in(0,1 / 2)$ and $\psi \in \mathrm{C}_c^{\infty}(\mathbb{T}^2\times(0, \infty))$, the family of laws of the martingales $(M^{\gamma,\eta})_{\eta,\gamma\in(0,1)}$ and $(M^{\gamma,\eta}_{\text{entropy}})_{\eta,\gamma\in(0,1)}$ defined by
 \begin{align}\label{D M}
 M_t^{\gamma,\eta}:=\sqrt{\varepsilon}\int_0^t \int_{\mathbb{T}^2} \psi\left(x, \rho^{\gamma, \eta}\right) \nabla \cdot\left[s_{\eta}\left(\rho^{\gamma, \eta}\right) \xi_{K}\right],  
\end{align}
and 
\begin{align}
M^{\gamma,\eta}_{\text{entropy},t}:=\sqrt{\varepsilon}\int_0^t\int_{\mathbb{T}^2}
\frac{s_{\eta}(\rho^{\gamma, \eta})}
{\sqrt{\rho^{\gamma, \eta}}}\nabla\sqrt{\rho^{\gamma, \eta}}\cdot\xi_{K}	
\end{align}

is tight on $\mathcal{C}_{[0, T]}\mathbb{R}$. 
\end{proposition}
\begin{remark}
	We remark that $M^{\gamma,\eta}_{\text{entropy}}$ is a new term compared to \cite[Proposition~5.23]{FG24}. The tightness of $M^{\gamma,\eta}_{\text{entropy}}$ will be utilized for obtaining solutions that satisfy the pathwise entropy inequality \eqref{eq-2.5}. 
\end{remark}

We can now prove the main theorem of this chapter. 

\begin{proof}[Proof of Theorem \ref{spde WP}]
With the help of aforementioned tightness results, we know that 
$$
\left(\rho^{\gamma,\eta},\nabla\sqrt{\rho^{\gamma,\eta}},m^{\gamma,\eta},M^{\gamma,\eta},M^{\gamma,\eta}_{\text{entropy}},\left(\int^{\cdot}_0\langle e_j,\xi\rangle\right)_{j\in\mathbb{Z}^2}\right)_{\gamma,\eta}
$$
is tight on 
$$
L^1_{[0,T]}L^1_{\mathbb{T}^2}\times(L^2_{[0,T]}L^2_{\mathbb{T}^2},w)\times\mathcal{M}_{[0,T]\times\mathbb{T}^2}\times\mathcal{C}_{[0,T]}\mathbb{R}\times\mathcal{C}_{[0,T]}\mathbb{R}\times(\mathcal{C}_{[0,T]}\mathbb{R}^2)^{\mathbb{Z}^2},
$$
where $m^{\gamma,\eta}$ denotes the kinetic measure of $\rho^{\gamma,\eta}$, $\mathcal{M}_{[0,T]\times\mathbb{T}^2}$ denotes the space of Borel measures, endowed with weak topology. 

Applying the Jakubowski-Skorokhod representation theorem \cite{Jak97} with Proposition \ref{entropy estimate} and Proposition \ref{T rho eta}, we can find sequences $\{\gamma_k\}_{k \in \mathbb{N}}$ and $\{\eta_k\}_{k \in \mathbb{N}}$ both converging to zero, together with a new probability space, still denoted by $(\Omega, (\mathcal{F}_t)_{t \in [0,T]}, \mathbb{P})$. On this probability space, there exists a family of adapted space--time white noise $\{\xi^k\}_{k \in \mathbb{N}} \cup \{\xi\}$, such that 
\begin{align}
\int_0^{\cdot} \langle e_j, \xi^k \rangle \to \int_0^{\cdot} \langle e_j, \xi \rangle, 
\end{align}
strongly in $C_{[0,T]}\mathbb{R}^2$ a.s. as $k \to \infty$ for any $j \in \mathbb{Z}^2$. Furthermore, due to the tightness result and the entropy dissipation estimates, there exists a family of processes $\{\rho^{k}\}_{k \in \mathbb{N}} \cup \{\rho\}$ in $L^{\infty}_{[0,T]}L^1_{\mathbb{T}^2}$, share the same law with the original variables, such that
\begin{align}\label{eq:strong-L1-convergence}
\rho^{k} \to \rho,\quad \text{strongly in} \ L^1_{[0,T]}L^1_{\mathbb{T}^2}, \quad \rho^{k} \to \rho, \quad \text{a.e.},
\end{align}
and
\begin{align}\label{eq:weak-grad-square-convergence}
\nabla\sqrt{\rho^{k}} \rightharpoonup \nabla\sqrt{\rho}, \quad \text{weakly in} \ L^2_{[0,T]}L^2_{\mathbb{T}^2}.
\end{align}
Furthermore, there exists a family of process $\{M^k,M^k_{\text{entropy}}\}_{k\in\mathbb{N}}\cup\{M,M_{\text{entropy}}\}$ in $\mathcal{C}_{[0,T]}\mathbb{R}\times\mathcal{C}_{[0,T]}\mathbb{R}$, share the same law with the original variables, such that
\begin{align}\label{eq:martingale-conv-1}
M^k \to M,\quad \text{in} \ \mathcal{C}_{[0,T]}\mathbb{R},
\end{align}
and
\begin{align}\label{eq:martingale-conv-2}
M^k_{\text{entropy}} \to M_{\text{entropy}}, \quad \text{in} \ \mathcal{C}_{[0,T]}\mathbb{R}.
\end{align}
Additionally, similar to \cite[Theorem 5.25]{FG24}, there exists a family of defect measures $\{m^{k}\}_{k \in \mathbb{N}}$ such that $m^{k} \geq \delta_0(\xi - \rho^{k}) |\nabla \rho^{k}|^2$, a.s. and a.s. for every $\psi \in C_c^{\infty}(\mathbb{T}^2 \times \mathbb{R}_+)$ and $t \in [0,T]$, 
\begin{align}\label{WKS-k}
&\sqrt{\varepsilon}\int_0^t \int_{\mathbb{T}^2} \psi\left(x, \rho^k\right) \nabla \cdot\left(s_{\eta_k}(\rho^k) \cdot \xi^k_{K}\right)  =-\int_{\mathbb{R}}\int_{\mathbb{T}^2} \chi^k(x, \xi, \cdot) \psi(x, \xi) \Big|^t_0 \nonumber\\
&-\int_{0}^{t} \int_{\mathbb{T}^2}\nabla\rho^k\cdot\nabla_x\psi(x,\rho^k)
-\int_{0}^{t} \int_{\mathbb{T}^2} \nabla\cdot\left(\rho^k \mathcal{V}_{\gamma_k}[\rho^k])\psi(x, \rho^k\right)\nonumber\\
&-\int_{0}^{t} \int_{\mathbb{R}} \int_{\mathbb{T}^2} \partial_{\xi} \psi(x, \xi) \mathrm{d}m^k-\frac{\varepsilon F_{1,K}}{2}\int^t_0\int_{\mathbb{T}^2}s_{\eta_k}'(\rho^k)^2\nabla\rho^k\cdot(\nabla\psi)(x,\rho^k)
+\frac{\varepsilon N_K}{2}\int_{0}^{t} \int_{\mathbb{T}^2}s_{\eta_k}(\rho^k)^2 \partial_{\xi} \psi(x, \rho^k),
\end{align}
where $\chi^k$ denotes the renormalized kinetic function of $\rho^k$, $\xi^k_K := P_K \xi^k$ denotes colored space--time white noises. 

\textbf{Step 1: Recover the equation.}
We now aim to show that the pair $(\rho, m)$ is a kinetic solution to \eqref{cspde} in the sense of Definition \ref{SPDE-0-intro} with respect to $\xi$. In light of the analysis in \cite[Theorem 5.5]{FG24} and \cite[Theorem 6.3]{WWZ22}, it remains to justify the passage to the limit in the kernel terms, namely we claim that almost surely for every $t \in [0,T]$, 
\begin{align}\label{eq:claim-convergence-kernel-kinetic}
\lim_{k \to \infty}\int_{0}^{t} \int_{\mathbb{T}^2} \psi(x, \rho^k) \nabla \cdot\left(\rho^k \mathcal{V}_{\gamma_k}[\rho^k]\right) = \int_{0}^{t} \int_{\mathbb{T}^2} \psi(x, \rho) \nabla \cdot\left(\rho \mathcal{V}[\rho]\right),
\end{align}
which follows from the following two limits: almost surely for every $t \in [0,T]$, 
\begin{align}\label{eq:conv-kinetic-interaction1}
\lim_{k \to \infty}\int_{0}^{t} \int_{\mathbb{T}^2} (\nabla_x\psi)(x, \rho^k)  \cdot \mathcal{V}_{\gamma_k}[\rho^k] \rho^k  = \int_{0}^{t} \int_{\mathbb{T}^2} (\nabla_x\psi)(x, \rho)  \cdot \mathcal{V}[\rho] \rho,
\end{align}
where $\nabla_x \psi$ is smooth and compactly supported in the velocity variable so that $(x, \zeta) \mapsto \zeta \nabla_x \psi(x,\zeta)$ is smooth and bounded; and almost surely for every $t \in [0,T]$, 
\begin{align}\label{eq:conv-kinetic-interaction2}
\lim_{k \to \infty}\int_{0}^{t} \int_{\mathbb{T}^2} \tilde{\psi}(x, \rho^k)  \left(\kappa_{\gamma_k} \ast \rho^k-\|\rho_0\|_{L^1_{\mathbb{T}^2}}\right)  = \int_{0}^{t} \int_{\mathbb{T}^2}  \tilde{\psi}(x, \rho)  \left(\rho-\|\rho_0\|_{L^1_{\mathbb{T}^2}}\right),
\end{align}
for $\tilde{\psi}(\cdot, \xi):= \int_0^{\xi} \partial_{\zeta}\psi(\cdot,\zeta) \zeta\mathrm{d}\zeta$, since due to integration by parts and chain rule, 
\begin{align*}
\int_{\mathbb{T}^2} \psi(x, \rho) \nabla \cdot\left(\rho \mathcal{V}[\rho]\right) = -\int_{\mathbb{T}^2} (\nabla_x\psi)(x, \rho) \cdot \mathcal{V}[\rho] \rho + \int_{\mathbb{T}^2} \tilde{\psi}(x, \rho) \nabla\cdot \mathcal{V}[\rho].
\end{align*}
Recall that $\rho^k \to \rho$ in $L^1([0,T]\times\mathbb{T}^2)$ strongly according to \eqref{eq:strong-L1-convergence}, and it follows from the definition of $\mathcal{V}[\cdot]$ that 
\begin{align}\label{eq:cor-conv-L1L1}
\lim_{k \to \infty}\left( \left\|\mathcal{V}_{\gamma_k}[\rho^k - \rho]\right\|_{L^1_{[0,T]}L^1_{\mathbb{T}^2}}+\left\|\mathcal{V}_{\gamma_k}[\rho]-\mathcal{V}[\rho]\right\|_{L^1_{[0,T]}L^1_{\mathbb{T}^2}}\right) = 0, \quad \lim_{k \to \infty} \left\|\kappa_{\gamma_k} \ast \rho^k - \rho\right\|_{L^1_{[0,T]}L^1_{\mathbb{T}^2}} = 0. 
\end{align}
By the decomposition $|\mathcal{V}_{\gamma_k}[\rho^k] - \mathcal{V}[\rho]| \leq |\kappa_{\gamma_k} \ast \mathcal{V}[\rho] - \mathcal{V}[\rho]| +|\mathcal{V}_{\gamma_k}[\rho^k - \rho]|$, and since $(x, \zeta) \mapsto \zeta \nabla_x\psi(x, \zeta)$ is bounded, we have almost surely, 
\begin{align*}
\lim_{k \to \infty}\Big|\int_{0}^{t} \int_{\mathbb{T}^2} \nabla_x\psi(x, \rho^k)  \cdot \mathcal{V}_{\gamma_k}[\rho^k] \rho^k  - \int_{0}^{t} \int_{\mathbb{T}^2} \nabla_x\psi(x, \rho^k)  \cdot \mathcal{V}[\rho] \rho^k\Big| = 0, 
\end{align*}
for all $t\in[0,T]$. On the other side, as we have the point-wise convergence of $\nabla_x\psi(x, \rho^k) \rho^k$ to $\nabla_x\psi(x, \rho) \rho$ as $k \to \infty$, and $\left|\left[\nabla_x\psi(x, \rho^k) \rho^k - \nabla_x\psi(x, \rho) \rho\right] \cdot \mathcal{V}[\rho]\right| \lesssim \mathcal{V}[\rho] \in L^1_{[0,T]}L^1_{\mathbb{T}^2}$
we have by dominated convergence that
\begin{align*}
\lim_{k \to \infty} \int_{0}^{t} \int_{\mathbb{T}^2} \nabla_x\psi(x, \rho^k)  \cdot \mathcal{V}[\rho] \rho^k  = \int_{0}^{t} \int_{\mathbb{T}^2} \nabla_x\psi(x, \rho)  \cdot \mathcal{V}[\rho] \rho.
\end{align*}
The two limits above lead to \eqref{eq:conv-kinetic-interaction1}, while \eqref{eq:conv-kinetic-interaction2} follows from the second limit in \eqref{eq:cor-conv-L1L1} and the identical argument as above. For the remaining terms, we refer to \cite[Theorem 5.25]{FG24}. The passage to the limit for these terms holds almost surely for almost every $t\in[0,T]$. Consequently, $\bar{\rho}$ satisfies the kinetic formulation almost surely for almost every $t\in[0,T]$.

\textbf{Step 2: The path-wise entropy estimate.}
With the help of the entropy dissipation estimate \eqref{PW ET q1}, and using the identity in law, we obtain
\begin{align*}
&\int_{\mathbb{T}^2}\Psi_{\delta}(\rho^{k}) \Big|^T_0 
+ 4\int_0^T\int_{\mathbb{T}^2}\frac{\rho^{k}}{\rho^{k}+\delta}
|\nabla\sqrt{\rho^{k}}|^2-\int^T_0\int_{\mathbb{T}^2}\frac{\rho^k}{\rho^k+\delta}\mathcal{V}_{\gamma_k}[\rho^k]\cdot\nabla\rho^k\\
\lesssim\;& 2\sqrt{\varepsilon}\int_0^T \int_{\mathbb{T}^2}
\frac{s_{\eta_k}(\rho^{k})\sqrt{\rho^{k}}}{\rho^{k}+\delta}
\nabla\sqrt{\rho^{k}} \cdot \xi^k_{K}
+ \frac{\varepsilon N_K}{2}\int_{0}^{T}\int_{\mathbb{T}^2}
\frac{ [s_{\eta}(\rho^{k})]^2}{\rho^{k}+\delta}.
\end{align*}

By the choice of $s_{\eta_k}(\cdot)$, we adopt the convention that
$\frac{s_{\eta_k}(\zeta)\sqrt{\zeta}}{\zeta}=0$ when $\zeta=0$.
For the last term, using the assumption $s_{\eta_k}(\cdot)\leq c\sqrt{\cdot}$, we obtain
\begin{align*}
\frac{\varepsilon N_K}{2}\int^T_0\int_{\mathbb{T}^2}
\frac{[s_{\eta_k}(\rho^k)]^2}{\rho^k+\delta}
\lesssim \varepsilon N_K T.
\end{align*}

For the martingale term, by the assumption on $s_{\eta}(\cdot)$, it follows that
$$
\left|\frac{s_{\eta_k}(\rho^{k})\sqrt{\rho^k}}{\rho^{k}+\delta}
-\frac{s_{\eta_k}(\rho^{k})\sqrt{\rho^{k}}}{\rho^{k}}\right|
\lesssim 1.
$$
Hence, by the dominated convergence theorem, we obtain
\begin{align*}
&\mathbb{E}\Bigg|\int_0^T \int_{\mathbb{T}^2}
\frac{s_{\eta_k}(\rho^{k})\sqrt{\rho^{k}}}{\rho^{k}+\delta}
\nabla\sqrt{\rho^{k}} \cdot \xi^k_{K}
-\int_0^T \int_{\mathbb{T}^2}
\frac{s_{\eta_k}(\rho^{k})\sqrt{\rho^{k}}}{\rho^{k}}
\nabla\sqrt{\rho^{k}} \cdot \xi^k_{K}\Bigg|^2\\
\leq\;& \mathbb{E}\sum_{j=1}^K\int^T_0\Bigg(
\int_{\mathbb{T}^2}\left(\frac{s_{\eta_k}(\rho^{k})\sqrt{\rho^{k}}}{\rho^{k}+\delta}
-\frac{s_{\eta_k}(\rho^{k})\sqrt{\rho^{k}}}{\rho^{k}}\right)
\nabla\sqrt{\rho^{k}}\,e_j\,\mathrm{d}x\Bigg)^2\mathrm{d}s\\
\leq\;& \mathbb{E}\left(\int^T_0
\|\nabla\sqrt{\rho^{k}}\|_{L^2(\mathbb{T}^2)}^2\,\mathrm{d}s\right)
\left(\mathbb{E}\int^T_0\sum_{j=1}^K
\left\|\frac{s_{\eta_k}(\rho^{k})\sqrt{\rho^{k}}}{\rho^{k}+\delta}
-\frac{s_{\eta_k}(\rho^{k})\sqrt{\rho^{k}}}{\rho^{k}}\right\|_{L^2(\mathbb{T}^2)}^2
\mathrm{d}s\right)
\to 0,
\end{align*}
as $\delta\to 0$.

Consequently, along a subsequence as $\delta\to 0$, together with the kernel estimates \eqref{ker-es-1} and \eqref{ker-es-2}, we obtain almost surely
\begin{align}\label{pathwise-entropy-0}
\int_0^T\int_{\mathbb{T}^2}
|\nabla\sqrt{\rho^{k}}|^2
\lesssim\; \int_{\mathbb{T}^2}\Psi(\rho_0)\,\mathrm{d}x
+ 2\sqrt{\varepsilon}\int_0^T \int_{\mathbb{T}^2}
\frac{s_{\eta_k}(\rho^{k})}{\sqrt{\rho^{k}}}
\nabla\sqrt{\rho^{k}} \cdot \xi_{K}
+ \varepsilon N_K T + 1.
\end{align}

For any $t\in[0,T]$, $\varepsilon\in(0,1)$, and $k\in\mathbb{N}$, similarly to \cite[(5.29)--(5.33)]{FG24}, a characterization of the martingale $M^k_{\text{entropy}}$ shows that 
$$
M^{k}_{\mathrm{entropy}}(t)
= \sqrt{\varepsilon}\int_0^t\left\langle
\frac{s_{\eta_k}(\rho^k)}{\sqrt{\rho^k}}
\nabla\sqrt{\rho^k}, \xi^k_{K}(s)\right\rangle.
$$

In the following, we show that, along a subsequence as $k\to\infty$, almost surely for every $t\in[0,T]$,
\begin{align}\label{martingale-conv}
M^k_{\mathrm{entropy}}(t)\to
M_{\mathrm{entropy}}(t)
= \sqrt{\varepsilon}\int_0^t\left\langle
\nabla\sqrt{\rho}, \xi_{K}(s)\right\rangle.
\end{align}%It follows from the It\^o isometry and \eqref{entropy0} that
% \begin{align*}
% \mathbb{E}\int^T_0[M^{k}(t)]^2\mathrm{d}t\leq & C(K)T\mathbb{E}\int_{0}^{T}\|\nabla\sqrt{\rho^k}\|_{L^2(\mathbb{T}^d)}^2\mathrm{d}s\leq C(K,T).
% \end{align*}
%This implies that there exists $M\in L^2_{\Omega\times[0,T]}$ such that for every $f\in L^2_{\Omega\times[0,T]}$, 
%\begin{equation}\label{KSC M}
%  \lim_{k\rightarrow\infty}\mathbb{E}\left(
%  \int^T_0f(t) M^{k}(t)\mathrm{d}t\right)=\mathbb{E}\left(\int^T_0 f(t) M(t)\mathrm{d}t\right).
%\end{equation}
In the following, we characterize $M_{\mathrm{entropy}}$ as a martingale with respect to the filtration 
$\mathcal{F}_t := \sigma(\rho|_{[0,t]}, \xi|_{[0,t]})$, and as a stochastic integral with respect to $\xi_K$.

\medskip
\noindent\textbf{Martingale property of $M_{\mathrm{entropy}}$.}
For every $0 \leq s \leq t \leq T$ and $A \in \mathcal{F}_s$, since $(M^k_{\mathrm{entropy},t})_{k \in \mathbb{N}}$ is uniformly integrable for each $t \in [0,T]$, it follows from the martingale property of $M^k_{\mathrm{entropy}}$ and its almost sure convergence in $\mathcal{C}([0,T]; \mathbb{R})$ that 
\begin{equation}
\mathbb{E}\left( \mathbf{1}_{A} (M_{\mathrm{entropy},t}-M_{\mathrm{entropy},s})\right)
= \lim_{k\to \infty}\mathbb{E}\left(\mathbf{1}_{A} (M^k_{\mathrm{entropy},t}-M^k_{\mathrm{entropy},s})\right)
= 0.
\end{equation}
This implies that $\{M_{\mathrm{entropy},t}\}_{t\in[0,T]}$ is a martingale with respect to $\{\mathcal{F}_t\}_{t\in[0,T]}$.

\medskip
\noindent\textbf{Representation as a stochastic integral.}
For every $0 \leq s \leq t \leq T$, $A \in \mathcal{F}_s$, and $|j|\leq K$, it follows from the covariance structure of $\xi_K$ that
\begin{equation}
\mathbb{E}\Bigg[\mathbf{1}_A\left(
M^k_{\mathrm{entropy},t} B^k_{j,t}
- M^k_{\mathrm{entropy},s} B^k_{j,s}
- \sqrt{\varepsilon}\int_s^t \left\langle
\frac{s_{\eta_k}(\rho^k)}{\sqrt{\rho^k}}\nabla\sqrt{\rho^k},
e_j \right\rangle
\right)\Bigg] = 0.
\end{equation}

By the integration by parts formula and the dominated convergence theorem, we have
\begin{align*}
&\mathbb{E}\left|\mathbf{1}_A \sqrt{\varepsilon}
\int_s^t \left\langle
\frac{s_{\eta_k}(\rho^k)}{\sqrt{\rho^k}}\nabla\sqrt{\rho^k}
- \nabla\sqrt{\rho^k}, e_j
\right\rangle \right| \\
=\;& \mathbb{E}\left|\mathbf{1}_A \sqrt{\varepsilon}
\int_s^t \left\langle
\left(\int_0^{\rho^k}\frac{s_{\eta_k}(\zeta)}{2\zeta}\,\mathrm{d}\zeta
- \sqrt{\rho^k}\right),
\nabla e_j
\right\rangle \right|
\to 0,
\end{align*}
as $k \to \infty$.

Passing to the limit $k \to \infty$, we obtain that for every $A \in \mathcal{F}_s$,
\begin{equation}\label{KSC MI 1}
\mathbb{E}\Bigg[\mathbf{1}_A\Big(
M_{\mathrm{entropy},t} B_{j,t}
- M_{\mathrm{entropy},s} B_{j,s}
- \sqrt{\varepsilon}\int_s^t \left\langle
\nabla\sqrt{\rho}, e_j
\right\rangle
\Big)\Bigg] = 0.
\end{equation}
This implies that
$$
M_{\mathrm{entropy},t} B_{j,t}
- \sqrt{\varepsilon}\int_0^t \left\langle
\nabla\sqrt{\rho}, e_j
\right\rangle
$$
is a $\mathcal{F}_t$-martingale.

By a similar argument, for every $0 \leq s \leq t \leq T$ and $A \in \mathcal{F}_s$,
\begin{equation}\label{KSC MI 2}
\mathbb{E}\Bigg[\mathbf{1}_A\Big(
[M_{\mathrm{entropy},t}]^2
- [M_{\mathrm{entropy},s}]^2
- \varepsilon\int_s^t
\left\|P_K(\nabla\sqrt{\rho})\right\|_{L^2_{\mathbb{T}^2}}^2
\Big)\Bigg] = 0,
\end{equation}
which implies that
$$
[M_{\mathrm{entropy},t}]^2
- \varepsilon\int_0^t
\left\|P_K(\nabla\sqrt{\rho})\right\|_{L^2_{\mathbb{T}^2}}^2
$$
is a $\mathcal{F}_t$-martingale.

Combining \eqref{KSC MI 1}, \eqref{KSC MI 2}, and the covariance structure of Brownian motions, we deduce that
\begin{equation}
\mathbb{E}\Bigg[\Big(
M_{\mathrm{entropy},t}
- \sqrt{\varepsilon}\int_0^t \langle
\nabla\sqrt{\rho}, \xi_K
\rangle
\Big)^2\Bigg] = 0,
\end{equation}
which implies that, almost surely, for every $t \in [0,T]$,
\begin{equation}\label{KSC M IR 2}
M_{\mathrm{entropy},t}
= \sqrt{\varepsilon}\int_0^t
\langle \nabla\sqrt{\rho}, \xi_K \rangle.
\end{equation}

Therefore, passing to the limits of \eqref{pathwise-entropy-0}, then we conclude the pathwise entropy dissipation estimate \eqref{eq-2.5} for $\rho$. This completes the proof.

\end{proof}

\section{Exponential estimates}\label{sec-4}
In the following sections, we investigate the large deviation properties of \eqref{SPDE-1}. By following the approach developed in the previous section, one can establish the existence of probabilistic weak renormalized kinetic solutions to \eqref{SPDE-1}. Moreover, these solutions also qualify as weak solutions in the sense of Definition \ref{W SPDE-1}. In this section, we derive exponential estimates for the solutions to \eqref{SPDE-1} under suitable scalings of $(\varepsilon, K, \eta)$. Recall that the equation can be formulated in It\^o integral form 
\begin{equation}\label{cspde}
\partial_t\rho^{\eta}=\Delta\rho^{\eta}-\nabla\cdot\left(\rho^{\eta}\mathcal{V}[\rho^{\eta}]\right)
-\sqrt{\varepsilon}\nabla\cdot\left(s_{\eta}(\rho^{\eta})\cdot\xi_{K}\right)
+\frac{\varepsilon F_{1,K}}{2}\nabla\cdot \left([s'_{\eta}(\rho^{\eta})]^2\nabla\rho^{\eta}\right).
\end{equation}

\begin{proposition}\label{exponential-tight}
Let $\mathcal{V}[\cdot]$ be defined by \eqref{ker-V} such that \eqref{eq:small-KS} holds and consider the scaling
\begin{equation}\label{scale1}
\begin{gathered}
K(\varepsilon)\to\infty,\qquad \varepsilon N_{K(\varepsilon)}\rightarrow 0,\qquad
\varepsilon N_{K(\varepsilon)}\|s'_{\eta(\varepsilon)}\|_{L^\infty}^2\lesssim 1,\qquad \eta(\varepsilon) \to 0.
\end{gathered}
 \end{equation}
 as $\varepsilon\rightarrow 0$. For every probabilistically weak renormalized kinetic solution $\rho^{\eta}$ to \eqref{cspde}, in the sense of Definition \ref{RKS smooth spde} with initial data $\rho_0$ of finite entropy, we let $\mu^{\varepsilon}$ be the law of $\rho^{\eta} = \rho^{\eta(\varepsilon)}$. Then there exists a sequence of compact sets $\{\mathcal{K}_M\}_{M\in\mathbb{N}}$ in $L^1_{[0,T]}L^1_{\mathbb{T}^2}$, such that
\begin{equation}\label{ec2-2}
  \lim_{M\to \infty}\limsup_{\varepsilon\to 0}\varepsilon \log\mu^{\varepsilon}(\mathcal{K}^c_M) = -\infty.
\end{equation}
\end{proposition}

\subsection{Exponential entropy dissipation estimates}

In the following, we show that the exponential moment of the Fisher information is finite, as a consequence of the pathwise entropy estimate. 

\begin{lemma}\label{exponential-entropy}
Under the same setting as in Proposition \ref{exponential-tight}, there exist constants $\delta_0>0$ such that for every $\delta <\delta_0$,
\begin{align}\label{eq:exp-moment-fisher}
 \sup_{\varepsilon\in(0,1)}\varepsilon \log\mathbb{E}\exp\left\{\frac{\delta}{\varepsilon }\|\nabla \sqrt{\rho^{\varepsilon}}\|_{L^2_{[0,T]}L^2_{\mathbb{T}^2}}^{2}\right\} < \infty. 
  \end{align}
\end{lemma}

\begin{proof}
It follows from the pathwise entropy estimate that for any $\delta > 0$, it holds almost surely that
\begin{align}\label{vious et1}
  \exp\left\{\frac{\delta }{\varepsilon }\|\nabla \sqrt{\rho^{\varepsilon}}\|_{L^2_{[0,T]}L^2_{\mathbb{T}^2}}^{2}\right\}
  \leq&\exp\left\{\frac{\delta}{\varepsilon }\Big(\int_{\mathbb{T}^2}\Psi(\rho_0(x))\mathrm{d}x+\varepsilon N_KT+1\Big)\right\}\notag\\
  &\cdot\exp\left\{\frac{2\delta}{\sqrt{\varepsilon}}\int_0^t\int_{\mathbb{T}^2}\nabla\sqrt{\rho^{\varepsilon}}\cdot\xi_K\right\}.
\end{align}
Since the quadratic variation of the martingale term can by It\^o isometry: almost surely, for every $t\in[0,T]$, 
\begin{align*}
   \left\langle\frac{2\delta}{\sqrt{\varepsilon}}\int_{0}^{\cdot}\int_{\mathbb{T}^2}
\nabla\sqrt{\rho^{\varepsilon}}\cdot \xi_K\right\rangle(t) \leq \frac{4\delta^2}{\varepsilon}\int_{0}^{t}
   \|\nabla\sqrt{\rho^{\varepsilon}}\|_{L^2(\mathbb{T}^2)}^2\mathrm{d}s.
\end{align*}
Then multiplying both sides of \eqref{vious et1} by $$\exp\left(
-\frac{1}{2}\left\langle\frac{2\delta}{\sqrt{\varepsilon}}\int_{0}^{\cdot}\int_{\mathbb{T}^2}
\nabla\sqrt{\rho^{\varepsilon}}\cdot \xi_K\right\rangle(t)\right),$$ taking expectations and derive, as 
$$
\exp\left(
\frac{2\delta}{\sqrt{\varepsilon}}\int_{0}^{t}\int_{\mathbb{T}^2}
\nabla\sqrt{\rho^{\varepsilon}}\cdot \xi_K-\frac{1}{2}\left\langle\frac{2\delta}{\sqrt{\varepsilon}}\int_{0}^{\cdot}\int_{\mathbb{T}^2}
\nabla\sqrt{\rho^{\varepsilon}}\cdot \xi_K\right\rangle(t)\right)
$$
is a supermartingale, that 
\begin{align}\label{Et sqrt}
\mathbb{E}\exp\left\{\frac{\delta-2\delta^2}{\varepsilon }\|\nabla \sqrt{\rho^{\varepsilon}}\|_{L^2([0,T];L^2(\mathbb{T}^2))}^{2}\right\}\leq \mathbb{E}\exp\left\{\frac{\delta }{\varepsilon }\left(\int_{\mathbb{T}^2}\Psi(\rho_0(x))\mathrm{d}x+\varepsilon N_KT+1\right)\right\}.
\end{align}
The desired bound \eqref{eq:exp-moment-fisher} then follows from \eqref{Et sqrt} when $\delta\in(0,1/2)$ and that $(\varepsilon, K(\varepsilon))$ satisfy the scaling $\varepsilon N_{K(\varepsilon)}\rightarrow 0$.
\end{proof}
It is a direct consequence of \eqref{exponential-entropy} and the Markov inequality that  
\begin{align}
\lim_{M\rightarrow \infty}\limsup_{\varepsilon\rightarrow 0^+} \left\{\varepsilon \log\mathbb{P}\left(\|\nabla\sqrt{\rho^{\varepsilon}}\|_{L^2_{[0,T]}L^2_{\mathbb{T}^2}}^2>M\right)\right\}=-\infty. 	
\end{align}

\subsection{Proof of Proposition \ref{exponential-tight}}We recall the following classical estimates of martingales due to Flandoli--Gatarek \cite[Lemma 2.1]{FG95} for readers convenience. 
\begin{lemma}Given $p \ge 2$, $\alpha < \tfrac{1}{2}$ and separable Hilbert spaces $H$ and $K$, we have, for any progressively measurable process 
$f \in L^p(\Omega \times [0,T]; \mathcal{L}_2(K,H))$, that
\begin{align}\label{eq:FlandoliGatarek}
\mathbb{E}\left\|I(f)\right\|_{W^{\alpha,p}_{[0,T]}H}^p 
\leq C^p p^{\frac{p}{2}}  \mathbb{E} \left( \int_0^T \|f(t)\|_{\mathcal{L}_2(K,H)}^p \mathrm{d}t\right),
\end{align}
where $\mathcal{L}_2(K,H)$ denotes the Hilbert--Schmidt space, and $I(f):=\int^{\cdot}_0f\mathrm{d}W$ denotes the It\^o integral of $f$ with respect to the cylindrical Wiener process $W$ on $K$.
\end{lemma}
\begin{remark}
We remark that \cite[Lemma 2.1]{FG95} does not explicitly specify the dependence of the constant on $p$. However, by carefully examining its proof, one observes that this dependence arises from the application of the Burkholder-Davis-Gundy inequality. In view of \cite{Ren2008BDG}, the corresponding constant can be bounded by $C^p p^{\frac{p}{2}}$. Although this bound may not be optimal, it is sufficient for the exponential estimates derived below.  
\end{remark}

\begin{lemma}\label{lem:exp-time-regularity}
Under the same setting as Proposition \ref{exponential-tight}, so in particular,
\begin{align*}
\varepsilon N_K\|s'_\eta\|_{L^\infty}^2\lesssim 1,
\end{align*}
for $\beta>3$ fixed and $\delta>0$ sufficiently small,
\begin{align}\label{eq:exp-time-regularity}
\limsup_{\varepsilon \to 0^{+}} \varepsilon\log \left\{\mathbb{E} \left[ \exp \left( \frac{\delta}{\varepsilon} \|\rho^{\varepsilon}\|_{W_{[0,T]}^{\frac{1}{3},1}H_{\mathbb{T}^2}^{-\beta}} \right) \right]\right\} < \infty.
\end{align}
\end{lemma}

\begin{proof}
Given a weak solution $\rho^{\varepsilon}$ to \eqref{cspde}, we claim that 
\begin{align}\label{timereg-AC}
\left\|\rho^{\varepsilon}\right\|_{W_{[0,T]}^{\frac{1}{3},1} H^{-\beta}_{\mathbb{T}^2}} & \lesssim \| \rho^{\varepsilon}(0)\|_{L_{\mathbb{T}^2}^1}
+  \|\rho^{\varepsilon}(0)\|_{L^1_{\mathbb{T}^2}} \|\nabla\sqrt{\rho^{\varepsilon}}\|_{L_{[0,T]}^2L^2_{\mathbb{T}^2}}^{2} +\|\rho^{\varepsilon}(0)\|_{L^1_{\mathbb{T}^2}}^2\notag\\ 
& + \varepsilon N_K\|s'_{\eta}\|_{L^{\infty}}^2\|\rho^{\varepsilon}(0)\|_{L^1_{\mathbb{T}^2}}^{\frac{1}{2}}\|\nabla\sqrt{\rho^{\varepsilon}}\|_{L_{[0,T]}^2L^2_{\mathbb{T}^2}} + \left\|\sqrt{\varepsilon}\int_0^{\cdot}\nabla\cdot\left(s_{\eta}(\rho^{\varepsilon})\xi_{K}\right)\right\|_{W_{[0,T]}^{\frac{1}{3},1} H^{-\beta}_{\mathbb{T}^2}}.
\end{align}
where the implicit constant is independent of $\varepsilon$. Indeed, as \eqref{cspde} holds for $\rho^{\varepsilon}$ in the sense of distribution, it suffices to bound the terms separately by 
\begin{align}\label{timereg-AC1}
\left\|\int_0^{\cdot} \Delta \rho^{\varepsilon}\right\|_{W_{[0,T]}^{\frac{1}{3},1} H_{\mathbb{T}^2}^{-\beta}} \lesssim\left\|\int_0^{\cdot} \Delta \rho^{\varepsilon}\right\|_{\mathcal{C}_{[0,T]}^{1} H_{\mathbb{T}^2}^{-\beta}} \lesssim \| \rho^{\varepsilon}\|_{L_{[0,T]}^{\infty} H_{\mathbb{T}^2}^{2-\beta}} \lesssim \| \rho^{\varepsilon}\|_{L_{[0,T]}^{\infty} L_{\mathbb{T}^2}^1} \leq \| \rho^{\varepsilon}(0)\|_{L_{\mathbb{T}^2}^1},
\end{align}
since $\beta >3$, and \eqref{cspde} preserves $L^1$ norm;
\begin{align}\label{timereg-AC2}
\left\|\int_{0}^{\cdot} \nabla \cdot \left(\rho^{\varepsilon} \mathcal{V}[\rho^{\varepsilon}]\right)\right\|_{W_{[0,T]}^{\frac{1}{3},1} H_{\mathbb{T}^2}^{-\beta}} \lesssim  \|\rho^{\varepsilon}(0)\|_{L^1_{\mathbb{T}^2}}^{\frac{1}{3}}\|\nabla\sqrt{\rho^{\varepsilon}}\|_{L_{[0,T]}^2L^2_{\mathbb{T}^2}}^{2}+\|\rho^{\varepsilon}(0)\|_{L^1_{\mathbb{T}^2}}^2,
\end{align}
since 
\begin{align*}
\left\|\int_{0}^{\cdot} \nabla \cdot \left(\rho^{\varepsilon} \mathcal{V}[\rho^{\varepsilon}]\right)\right\|_{W^{\frac{1}{3},1}_{[0,T]}H_{\mathbb{T}^2}^{-\beta}} \leq \left\|\int_{0}^{\cdot} \rho^{\varepsilon} \mathcal{V}[\rho^{\varepsilon}]\right\|_{W^{1,1}_{[0,T]}L_{\mathbb{T}^2}^1} \leq  \left\|\rho^{\varepsilon} \mathcal{V}[\rho^{\varepsilon}]\right\|_{L^{1}_{[0,T]}L_{\mathbb{T}^2}^1},
\end{align*}
and we know that $\|\rho^{\varepsilon} \mathcal{V}[\rho^{\varepsilon}]\|_{L^{1}_{[0,T]}L_{\mathbb{T}^2}^1} \lesssim  \|\rho^{\varepsilon}(0)\|_{L^1_{\mathbb{T}^2}}\|\nabla\sqrt{\rho^{\varepsilon}}\|_{L_{[0,T]}^2L^2_{\mathbb{T}^2}}^{2}+\|\rho^{\varepsilon}(0)\|_{L^1_{\mathbb{T}^2}}^2$ by \eqref{interpolation}. For the correction term, we have similarly
\begin{align}\label{timereg-AC3}
\left\|\int_0^{\cdot}\nabla\cdot \left([s'_{\eta}(\rho^{\varepsilon})]^2\nabla\rho^{\varepsilon}\right)\right\|_{W_{[0,T]}^{\frac{1}{3},1} H_{\mathbb{T}^2}^{-\beta}} \lesssim&\left\|\int_0^{\cdot}\nabla\cdot \left([s'_{\eta}(\rho^{\varepsilon})]^2\nabla\rho^{\varepsilon}\right)\right\|_{W_{[0,T]}^{1,1} H_{\mathbb{T}^2}^{-\beta}} \lesssim \|[s'_{\eta}(\rho^{\varepsilon})]^2\nabla\rho^{\varepsilon}\|_{L_{[0,T]}^1L^1_{\mathbb{T}^2}} \nonumber \\ & \lesssim \|s'_{\eta}\|_{L^{\infty}}^2\|\rho^{\varepsilon}(0)\|_{L^1_{\mathbb{T}^2}}^{\frac{1}{2}}\|\nabla\sqrt{\rho^{\varepsilon}}\|_{L_{[0,T]}^2L^2_{\mathbb{T}^2}}.
\end{align}
Combining the estimates~\eqref{timereg-AC1}--\eqref{timereg-AC3}, we arrive at the first inequality in \eqref{timereg-AC}, while the second inequality follows directly from Young's inequality. For the martingale part, we claim that for any $p \in 2\mathbb{N}$ and $\alpha \in (0, \frac{1}{2})$ that 
\begin{align}\label{eq:timereg-MG}
\mathbb{E}\left\| \int_0^{\cdot} \nabla \cdot \left[s_{\eta}(\rho^{\varepsilon})  \cdot \xi_{K}\right]\right\|_{W_{[0,T]}^{\alpha,p}H_{\mathbb{T}^2}^{-\beta}}^p \leq  C^pp^{\frac{p}{2}} \|\rho^{\varepsilon}(0)\|^{\frac{p}{2}}_{L^1_{\mathbb{T}^2}},
\end{align}
where the generic constant $C$ is independent of $p$. 

By the definition of the noise \eqref{noise-def}, we first rewrite the stochastic integral as
\begin{align*}
\int_0^{\cdot} \nabla \cdot \big[ s_{\eta}(\rho^{\varepsilon}) \cdot \xi_{K} \big]
&= \int_0^{\cdot} \sum_{|j| \leq K} \nabla \cdot \big[ s_{\eta}(\rho^{\varepsilon}) e_j \big] B_j \\
&= \int_0^{\cdot} \sum_{i=1}^2 \sum_{|j| \leq K} \partial_{x_i} \big[ s_{\eta}(\rho^{\varepsilon}) e_j \big] B_j^i,
\end{align*}
where we denote $B_j = (B_j^1, B_j^2)$, $j \in \mathbb{N}$, as two-dimensional vector-valued Brownian motions. With this notation, the components $\sum_{j \in \mathbb{N}} e_j B_j^1$ and $\sum_{j \in \mathbb{N}} e_j B_j^2$ are both $L^2_{\mathbb{T}^2}$-cylindrical Wiener processes.

We define the integrand by
$$
f_i^{\varepsilon} = \partial_{x_i} \big[ s_{\eta}(\rho^{\varepsilon}) P_K \cdot \big], \qquad \text{for } i = 1,2,
$$
and introduce the notation
$$
I(f_i^{\varepsilon}) := \int_0^{\cdot} \sum_{|j| \leq K} \partial_{x_i} \big[ s_{\eta}(\rho^{\varepsilon}) e_j \big] B_j^i, \qquad \text{for } i = 1,2.
$$
According to \eqref{eq:FlandoliGatarek}, we deduce that
\begin{align}
\mathbb{E}\left\| \sum_{i=1}^2 I(f_i^{\varepsilon}) \right\|_{\dot{W}_{[0,T]}^{\alpha,p} H^{-\beta}_{\mathbb{T}^2}}^p
&\leq C^p p^{\frac{p}{2}} \sum_{i=1}^2 \mathbb{E} \left( \int_0^T \| f_i^{\varepsilon}(t) \|_{\mathcal{L}_2(L^2_{\mathbb{T}^2}, H^{-\beta}_{\mathbb{T}^2})}^p \, \mathrm{d}t \right).
\end{align}
Since we choose $\beta > 3$, then for $i=1,2$,  
\begin{align*}
\|f^{\varepsilon}_i(t)\|^2_{\mathcal{L}_2(L_{\mathbb{T}^2}^2,H_{\mathbb{T}^2}^{-\beta})} \lesssim \sum_{k \in \mathbb{Z}^2} \frac{1}{\langle k \rangle^{2\beta}}  \left \| \partial_{x_i} \left[s_{\eta}(\rho^{\varepsilon}) e_k\right] \right\|_{H_{\mathbb{T}^2}^{-\beta}}^2 \lesssim \|s_{\eta}(\rho^{\varepsilon}) \|^2_{L_{\mathbb{T}^2}^2} \leq \|\rho^{\varepsilon}(0)\|_{L^1_{\mathbb{T}^2}}.
\end{align*}
Consequently, we have by \eqref{eq:timereg-MG} and Fatou's lemma, 
\begin{align}\label{eq:exp-moment-MG}
\mathbb{E}\exp\left(\frac{\delta}{\sqrt{\varepsilon}} \left\|\sum_{i=1}^2I(f^{\varepsilon}_i)\right\|_{W_{[0,T]}^{\alpha,1}H_{\mathbb{T}^2}^{-\beta}}\right) \lesssim \mathbb{E}  \left[\sum_{p \in 2\mathbb{N}}\frac{1}{p!} \left(\frac{\delta}{\sqrt{\varepsilon}}\right)^p \left\|  \sum_{i=1}^2I(f^{\varepsilon}_i)\right\|_{W_{[0,T]}^{\alpha,1}H_{\mathbb{T}^2}^{-\beta}}^p\right] \lesssim \exp\left(\frac{C\delta^2}{\varepsilon}\right),
\end{align}
where the constant $C$ depends on $\|\rho^{\varepsilon}(0)\|_{L^1_{\mathbb{T}^2}}$. We used in the first inequality the same elementary argument as in \cite[(46)]{JW18} and in the last inequality that 
\begin{align*}
\sum_{p \in 2\mathbb{N}} \frac{1}{p!} (C\delta)^pp^{\frac{p}{2}} \varepsilon^{-\frac{p}{2}} & \sim \sum_{p \in 2\mathbb{N}} \frac{e^p}{p^{p+\frac{1}{2}}} (C\delta)^pp^{\frac{p}{2}} \varepsilon^{-\frac{p}{2}} = \sum_{p \in \mathbb{N}} \frac{e^{2p}}{(2p)^{2p+\frac{1}{2}}} (C\delta)^{2p}(2p)^{p} \varepsilon^{-p} \\& \sim \sum_{p \in \mathbb{N}} \frac{(Ce\delta)^{2p}}{p^{p+\frac{1}{2}}} (2\varepsilon)^{-p} \sim \sum_{p \in \mathbb{N}} \frac{(C^2e\delta^2)^{p}}{p!} (2\varepsilon)^{-p}   \lesssim \exp \left(\frac{C^2 e\delta^2}{2\varepsilon}\right),
\end{align*}
due to Stirling's formula, and we abuse the notation in \eqref{eq:exp-moment-MG} so that $C$ stands for a different constant from the one above. We now combine \eqref{timereg-AC} and \eqref{eq:exp-moment-MG} and derive by Young's inequality that 
\begin{align}
\mathbb{E} \left[ \exp \left( \frac{\delta}{\varepsilon} \|\rho^{\varepsilon}\|_{W_{[0,T]}^{\frac{1}{3},1}H_{\mathbb{T}^2}^{-\beta}} \right) \right] \lesssim & \mathbb{E} \left[ \exp \left(\frac{C\delta}{\varepsilon} \left[ 1 + \|\nabla\sqrt{\rho^{\varepsilon}}\|_{L_{[0,T]}^2L^2_{\mathbb{T}^2}}^{\frac{4}{3}} + \varepsilon N_K \|s'_{\eta}\|_{L^{\infty}}^2 \|\nabla\sqrt{\rho^{\varepsilon}}\|_{L_{[0,T]}^2L^2_{\mathbb{T}^2}}\right] \right) \right] \notag \\  & + \exp\left(\frac{C\delta^2}{\varepsilon}\right),
\end{align}
Taking $\delta>0$ sufficiently small and using Young's inequality, \eqref{eq:exp-moment-fisher}, and the assumption $\varepsilon N_K \|s'_{\eta}\|_{L^{\infty}}^2 \lesssim 1$, we obtain \eqref{eq:exp-time-regularity}.
\end{proof}

The proof of the desired Proposition is now straightforward.
\begin{proof}[Proof of Proposition \ref{exponential-tight}]
It follows directly from Lemma \ref{exponential-entropy}, Lemma \ref{lem:exp-time-regularity} and Markov inequality that $\rho^{\varepsilon}$ is exponential tight with respect the following sequence $\{\mathcal{K}_M\}$ in $L^1_{[0,T]}L^1_{\mathbb{T}^2}$:
\begin{align}
\mathcal{K}_M := \left\{\rho \in L_{[0,T]}^{\infty}L^{1}_{\mathbb{T}^2}: \left\|\rho\right\|_{L_{[0,T]}^{2}W^{1,1}_{\mathbb{T}^2}} \vee \left\|\rho\right\|_{W_{[0,T]}^{\frac{1}{3},1}H^{-3}_{\mathbb{T}^2}} \leq M\right\},
\end{align}
and it follows directly from Aubin--Lions--Simon \cite{Sim87} that $\mathcal{K}_M$ is compact in $L^1_{[0,T]}L^1_{\mathbb{T}^2}$ for fixed $M$.
\end{proof}

\section{Lower bound for large deviations}\label{sec:LB-LDP}
In this section, we prove the restricted lower bound of the large deviations. The proof relies on the method of relative entropy, for which we set up first an abstract Lemma. Let $E$ be a Polish space, and denote by $\mathcal{P}(E)$ the set of all probability measures on $E$. For any $\nu, \mu \in \mathcal{P}(E)$, the relative entropy $\mathcal{H}(\nu \mid \mu)$ is defined by
\begin{align*}
\mathcal{H}\left (\nu\mid\mu\right):=\int_E \log \left(\frac{\mathrm{d} \nu}{\mathrm{d} \mu}\right) \mathrm{d} \nu .
\end{align*}
\begin{lemma}[\cite{M10}, Lemma 7]\label{entropymethod}
Let $E$ be a Polish space, $I: E\rightarrow[0,+\infty]$ be a positive functional and $\{\mu^{\varepsilon}\}_{\varepsilon>0} $ be a family of probability measures on $E$. Suppose that for every $\rho \in E$, there exists a family $\{\nu^{\varepsilon,\rho}\}_{\varepsilon>0}$ such that $\nu^{\varepsilon,\rho}\rightharpoonup\delta_{\rho}$ weakly, and
\begin{align}\label{eq:relative-entropy-bound}
\limsup_{\varepsilon\rightarrow0}\varepsilon \mathcal{H}\left(\nu^{\varepsilon,\rho}|\mu^{\varepsilon}\right)\leq I(\rho).
\end{align}
Then the family $\{\mu^{\varepsilon}\}_{\varepsilon>0}$ satisfies the large deviation lower bound with speed $\varepsilon^{-1}$ and rate function $I$. 
\end{lemma}

Recall that the rate function $\mathcal{I}$ is defined by \eqref{I0-intro}. For any $\rho\in L^1_{[0, T]}L^1_{\mathbb{T}^2}$ such that $\mathcal{I}(\rho)=+\infty$, we can take
\begin{align*}
\nu^{\varepsilon, \rho} \equiv \delta_{\rho}, \ \text{for all} \ \varepsilon > 0, 
\end{align*}
and the conditions in Lemma \ref{entropymethod} were satisfied trivially; thus we only focus on the case of $\mathcal{I}(\rho)<+\infty$. We also notice that, for any $\rho \in \mathcal{C}_0$ with $\mathcal{I}(\rho)<+\infty$, where $\mathcal{C}_0$ is defined in \eqref{S-intro}, $\rho$ can be represented as a weak solution to the skeleton equation. 

\begin{lemma}\label{ske-lem}
Let $\mathcal{V}[\cdot]$ be defined by \eqref{ker-V}, and $\rho$ be a nonnegative function satisfying $\mathcal{I}(\rho)<+\infty$, then there
exists a function $\Psi^{\rho}$, such that $\sqrt{\rho}\nabla\Psi^{\rho} \in L^2_{[0,T]}L^2_{\mathbb{T}^2}$, and $\rho$ is a weak solution of 
\begin{equation}\label{eq-riesz}
\partial_t\rho+\nabla\cdot(\rho \mathcal{V}[\rho])=\Delta\rho-\nabla\cdot(\rho\nabla\Psi^{\rho}),
\end{equation}
and the rate function $\mathcal{I}$ defined in \eqref{I0-intro} has an equivalent representation
\begin{align}\label{RF low}
\mathcal{I}(\rho)=\frac{1}{2}\|\sqrt{\rho}\nabla\Psi^{\rho}\|_{L_{[0,T]}^2L^2_{\mathbb{T}^2}}^2.
\end{align}
If we further assume that $\rho \in L^\infty_{[0,T]}L^\infty_{\mathbb{T}^2}$,  then $\rho$ is also a renormalized kinetic solution of the skeleton equation \eqref{ske} with control $g=\sqrt{\rho}\nabla\Psi^{\rho}$. 
\end{lemma}
\begin{proof}
The existence of $\Psi^{\rho}$ and the identity of the rate function \eqref{RF low} can be shown by a Riesz representation approach, this proof is similar to \cite[Lemma 6.2]{GHW23} and \cite[Lemma 36]{FG23}, thus we omit it. In order to show that $\rho$ is also a renormalized kinetic solution, with the help of \cite[Theorem 3.5]{WZ24}, it is sufficient to verify the condition \cite[(3.22)]{WZ24}. Since we impose regularity $\rho\in L^\infty_{[0,T]}L^\infty_{\mathbb{T}^2}$, by the integrability of $V$ and convolutional Young's inequality, we have that $\rho|\mathcal{V}[\rho]|^2\in L^1_{[0,T]}L^1_{\mathbb{T}^2}$ and $\mathcal{V}[\rho]\in L^2_{[0,T]}L^2_{\mathbb{T}^2}$. Therefore \cite[(3.22)]{WZ24} holds, and $\rho$ is a renormalized kinetic solution of the skeleton equation \eqref{ske} with control $g=\sqrt{\rho}\nabla\Psi^{\rho}$. This completes the proof. 
	
\end{proof}

We fix a probabilistically weak solution $\rho^{\eta}$ of \eqref{SPDE-1} on some probability space $(\Omega,\mathcal{F}, (\mathcal{F}_t)_{t \in [0,T]}, \mathbb{P})$ together with the adapted space--time white noise $\tilde{\xi}$. Let $\mu_{\varepsilon}$ be the law of $\rho^{\eta}$. In the following, we consider the equation
\begin{equation}\label{SCPDE}
  \partial_{t} \rho^{\flat,\varepsilon}=\Delta \rho^{\flat,\varepsilon} - \nabla\cdot \left(\rho^{\flat,\varepsilon}\mathcal{V}[\rho^{\flat,\varepsilon}]\right)
-\sqrt{\varepsilon}\nabla\cdot \left( s_{\eta}(\rho^{\flat,\varepsilon})\circ \xi_{K} \right)-\nabla\cdot \left[ s_{\eta}(\rho^{\flat,\varepsilon}) P_{K}\left(\sqrt{\rho}\nabla\Psi^{\rho}\right)\right],
\end{equation}
where $P_{K}$ is the Fourier projection onto the first $K$\textsuperscript{th} modes. %We now explain the reason for introducing \eqref{SCPDE} as follows:  Therefore, it suffices to prove that $\nu^{\varepsilon, \rho}$, the family of laws of the weak solutions of \eqref{SCPDE}, is well--defined and concentrates at the dirac measure $\delta_{\rho}$. The existence of $\nu^{\varepsilon, \rho}$ for $\varepsilon > 0$ is addressed in the following proposition:

\begin{proposition}\label{cspde WPE}
Let the interaction $\mathcal{V}[\cdot]$ be defined by \eqref{ker-V} and the initial condition $\rho_0\in\mathrm{Ent}(\mathbb{T}^2)$. For fixed $\rho \in \mathcal{C}_0$ in Lemma \ref{entropymethod}, let $g=\sqrt{\rho}\nabla\Psi^{\rho}$. Then there exists a nonnegative probabilistically weak renormalized kinetic solution $\rho^{\flat,\varepsilon}$ of \eqref{SCPDE} with initial data $\rho^{\flat}_0$ and control $g$, in the sense that for every $t\in[0,T]$ and $\psi \in \mathrm{C}_{c}^{\infty}\left(\mathbb{T}^2\times(0,\infty)\right)$, it holds almost surely that for almost every $t\in[0,T]$, 
\begin{align}\label{MC kenitic solution}
&\int_{\mathbb{R}}\int_{\mathbb{T}^2} \chi^{\flat,\varepsilon}(x, \zeta, \cdot) \psi(x, \zeta)\Big|_0^t = -\int_{0}^{t} \int_{\mathbb{R}} \int_{\mathbb{T}^2}\nabla\rho^{\flat,\varepsilon}\cdot\nabla_x\psi(x,\rho^{\flat,\varepsilon}) -\int_{0}^{t} \int_{\mathbb{R}} \int_{\mathbb{T}^2} \partial_{\zeta} \psi(x, \zeta) \mathrm{d}m^{\flat,\varepsilon}\nonumber\\
& - \int_{0}^{t} \int_{\mathbb{T}^2} \nabla\cdot (\rho^{\flat,\varepsilon} \mathcal{V}[\rho^{\flat,\varepsilon}])\psi(x, \rho^{\flat,\varepsilon})
+ \int_{0}^{t} \int_{\mathbb{T}^2} s_{\eta}(\rho^{\flat,\varepsilon})P_{K}g \cdot \nabla_x\psi(x, \rho^{\flat,\varepsilon})\nonumber\\
& + \int_{0}^{t} \int_{\mathbb{T}^2} s_{\eta}(\rho^{\flat,\varepsilon})P_{K}g \cdot \nabla\rho^{\flat,\varepsilon} \partial_\zeta\psi(x, \rho^{\flat,\varepsilon})+\sqrt{\varepsilon}\int_{0}^{t} \int_{\mathbb{T}^2} s_{\eta}(\rho^{\flat,\varepsilon}) \xi_{K}\cdot \nabla \psi(x, \rho^{\flat,\varepsilon})\nonumber\\
&-\frac{\varepsilon F_{1,K}}{2}\int^t_0\int_{\mathbb{T}^2}s_{\eta}'(\rho^{\flat,\varepsilon})^2\nabla\rho^{\flat,\varepsilon}\cdot(\nabla\psi)(x,\rho^{\flat,\varepsilon})+\frac{\varepsilon N_K}{2}\int_{0}^{t} \int_{\mathbb{T}^2} s_{\eta}(\rho^{\flat,\varepsilon})^2 \partial_{\zeta} \psi(x, \rho^{\flat,\varepsilon}),
\end{align}
with preservation of mass $\|\rho^{\flat,\varepsilon}\|_{L^{1}_{\mathbb{T}^2}}\equiv \|\rho^{\flat}_0\|_{L^{1}_{\mathbb{T}^2}}$ and the pathwise entropy inequality 
\begin{align}\label{L1-rho-flat}
\int_0^t\int_{\mathbb{T}^2}
|\nabla\sqrt{\rho^{\flat,\varepsilon}}|^2\mathrm{d}x\mathrm{d}s
\lesssim&\int_{\mathbb{T}^2}\Psi(\rho^{\flat}_0)\mathrm{d}x+\sqrt{\varepsilon}\int_0^t\int_{\mathbb{T}^2}
2\nabla\sqrt{\rho^{\flat,\varepsilon}}\cdot
\xi_{K}+\int_0^t\int_{\mathbb{T}^2}
2\nabla\sqrt{\rho^{\flat,\varepsilon}}\cdot
P_Kg
+\varepsilon N_KT+1,
\end{align}
almost surely, for all $t\in[0,T]$, and further the nonnegative kinetic measure $m^{\flat,\varepsilon}$ satisfies
\begin{align}\label{rho-flat-kinetic-measure}
4\delta_{0}(\zeta-\rho^{\flat,\varepsilon})\zeta |\nabla \sqrt{\rho^{\flat,\varepsilon}}|^{2}\le m^{\flat,\varepsilon},\quad \liminf_{M\to \infty} \mathbb{E}\big[m^{\flat,\varepsilon}(\mathbb{T}^2\times [0,T]\times [M,M+1])\big]=0,
\end{align}
where $\zeta$ stands for the velocity variable. Moreover, the law $\nu^{\varepsilon,\rho}$ of $\rho^{\flat,\varepsilon}$ satisfies 
\begin{align*}
\varepsilon\mathcal{H}\left(\nu^{\varepsilon,\rho}\mid
\mu^{\varepsilon}\right)
\leq \mathcal{I}(\rho). 
\end{align*}

\end{proposition}
\begin{proof}
For any fixed $\varepsilon > 0$, recall that we fix a probabilistically weak solution $\rho^{\eta}$ of \eqref{SPDE-1} on some probability space $(\Omega,\mathcal{F}, (\mathcal{F}_t)_{t \in [0,T]}, \mathbb{P})$ together with the adapted space--time white noise $\tilde{\xi}$ and consider the martingales
\begin{equation}\label{martingale}
M^{\rho}(t) := -\frac{1}{\sqrt{\varepsilon}}\int_0^t\int_{\mathbb{T}^2} \sqrt{\rho}\nabla\Psi^{\rho} \cdot \tilde{\xi}_{K}, \quad \mathcal{E}(M^{\rho}) := \exp\left(M^{\rho}-\frac{1}{2}\langle M^{\rho}\rangle\right).
\end{equation}
We define the new probability measure
\begin{equation}\label{Qmeasure}
\mathrm{d}\mathbb{Q}^{\rho} := \mathcal{E}(M^{\rho})\mathrm{d}\mathbb{P}, 
\end{equation}
and $\nu^{\varepsilon,\rho}$ and $\mu^{\varepsilon}$ as the law of the weak solution $\rho^{\eta}$ of \eqref{SPDE-1} under $\mathbb{Q}^{\rho}$ and $\mathbb{P}$, respectively. We notice that $\rho^{\eta}$ is exactly a weak solution to \eqref{SCPDE} with $\xi=\tilde{\xi}+P_K(\sqrt{\rho}\nabla\Psi^{\rho})$ under $\mathbb{Q}^{\rho}$, as we use $P^2_K = P_K$, and by Girsanov theorem, $\xi=\tilde{\xi}+P_K(\sqrt{\rho}\nabla\Psi^{\rho})$ is a white noise under $\mathbb{Q}^{\rho}$. Furthermore, the desired bound \eqref{eq:relative-entropy-bound} on the relative entropy follows, as
\begin{align*}
\varepsilon \mathcal{H}\left(\nu^{\varepsilon,\rho}\mid
\mu^{\varepsilon}\right) \leq \varepsilon \mathcal{H}\left(\mathbb{Q}^{\rho}
\mid\mathbb{P}\right) = \varepsilon \int_{\Omega} 
\left(M^{\rho}_T-\langle M^{\rho}\rangle_T+\frac{1}{2}\langle M^{\rho}\rangle_T\right)\mathrm{d}\mathbb{Q}^{\rho} = \frac{\varepsilon}{2} \int_{\Omega} 
\langle M^{\rho}\rangle_T\mathrm{d}\mathbb{Q}^{\rho},
\end{align*}
where the first inequality is due to the information processing inequality and last equality is due to Girsanov theorem, which implies that $M^{\rho} -\langle M^{\rho}\rangle$ is a martingale under $\mathbb{Q}^{\rho}$. It follows that 
\begin{align}\label{entropy-check}
\varepsilon\mathcal{H}\left(\nu^{\varepsilon,\rho}\mid
\mu^{\varepsilon}\right)
\leq \frac{1}{2}
\int_0^T\int_{\mathbb{T}^2} \rho \left|\nabla\Psi^{\rho}\right|^2 = \mathcal{I}(\rho)
\end{align}
according to \eqref{RF low}.	 This completes the proof. 
\end{proof}

We will majorly consider controls $g = \sqrt{\rho} \nabla \Psi^{\rho}$. As the measures $\nu^{\varepsilon, \rho}$ are now constructed, we proceed to prove that those measures concentrate at $\delta_{\rho}$ by using the compactness--uniqueness argument, namely we first show that the family $\{\nu^{\varepsilon, \rho}\}_{\varepsilon > 0}$ is tight on $L^1_{[0,T]}L^1_{\mathbb{T}^2}$ equipped with the strong topology in Lemma \ref{tight} below, and then we prove that any limit point of a convergent subsequence is a solution to the skeleton equation \eqref{eq-riesz} in Lemma \ref{C1}. We now specify the notion of solution of the (slightly more general) skeleton equations as follows: Let
\begin{align}\label{control-1}
\partial_t \rho=\Delta\rho-\nabla\cdot(\rho \mathcal{V}[\rho])-\nabla \cdot\left(\sqrt{\rho} g\right),
\end{align}
with $g \in L^2_{[0,T]}L^2_{\mathbb{T}^2}(\mathbb{R}^2)$. For a given nonnegative solution $\rho$ of (\ref{control-1}), recall that the kinetic function of $\rho$ is defined by $\chi(x,\xi,t)=\mathbf{1}_{\{0<\xi<\rho(x, t)\}}$.

\begin{definition}\label{ske-kinetic}
Let $\mathcal{V}[\cdot]$ be defined by \eqref{ker-V} such that \eqref{eq:small-KS} holds, and let $\rho_0 \in \operatorname{Ent}\left(\mathbb{T}^2\right)$ and $g\in L^2_{[0,T]}L^2(\mathbb{T}^2;\mathbb{R}^2)$. A  nonnegative function $\rho\in L^\infty_{[0,T]}L^1_{\mathbb{T}^2}$ is  a renormalized kinetic solution of (\ref{control-1}) with initial  $\rho_0$ and control $g$, if $\rho$ satisfies
\begin{align}\label{eq:skeleton-l1-fisher}
\|\rho(\cdot,t)\|_{L^{1}\left(\mathbb{T}^2\right)} = \left\|\rho_0\right\|_{L^{1}\left(\mathbb{T}^2\right)}, \text{ a.e. }  \quad \nabla\sqrt{\rho}\in L^2_{[0,T]}L^2_{\mathbb{T}^2}(\mathbb{R}^2),
\end{align}
and furthermore, there exists a nonnegative kinetic measure $m$ such that
\begin{align}\label{KM R}
4\delta_{0}(\xi-\rho)\xi|\nabla \sqrt{\rho}|^{2}\le m,
\end{align}
on $\mathbb{T}^2\times(0,\infty)\times[0,T]$, and
\begin{align}\label{KM VI}
\liminf_{M\to \infty}\big[m(\mathbb{T}^2\times [0,T]\times [M,M+1])\big]=0.
\end{align}
For every $\psi\in\mathrm{C}_{c}^{\infty}\left(\mathbb{T}^2\times(0,\infty)\right)$ and a.e. $t \in [0,T]$, it holds
\begin{align}\label{cspde KSE}
\int_{\mathbb{R}}\int_{\mathbb{T}^2} \chi(x, \xi, \cdot) \psi(x, \xi)\Big|_0^t
 & =\int_{0}^{t} \int_{\mathbb{R}} \int_{\mathbb{T}^2}\chi(x, \xi, t)\Delta_x\psi(x,\xi) - \int_{0}^{t} \int_{\mathbb{R}} \int_{\mathbb{T}^2} \partial_{\xi} \psi(x, \xi) \mathrm{d}m\nonumber\\
&-\int_{0}^{t} \int_{\mathbb{T}^2} \nabla\cdot(\rho \mathcal{V}[\rho])\psi(x, \rho)
+\int_{0}^{t} \int_{\mathbb{T}^2}\sqrt{\rho}g\cdot\nabla_x\psi(x, \rho)\nonumber\\
&+2\int_{0}^{t} \int_{\mathbb{T}^2}\rho\nabla\sqrt{\rho}\cdot g\partial_\xi\psi(x, \rho).
\end{align}
\end{definition}

In the following, we state the tightness result, it can be obtained with the help of the pathwise entropy inequality \eqref{L1-rho-flat} and a time-regularity estimate similar to Lemma \ref{spde E T2}, thus we omit the proof. 
\begin{lemma}\label{tight}
For every $\varepsilon>0$ and under the same condition as Proposition \ref{cspde WPE}, we let $\rho^{\flat,\varepsilon}$  be the  probabilistically weak renormalized kinetic solution for \eqref{SCPDE} with initial data $\rho_0$ and control $g$ constructed therein. Then under the scaling relation 
$$
K(\varepsilon)\to\infty,\qquad \varepsilon N_{K(\varepsilon)}\rightarrow 0,
$$
as $\varepsilon\to0$, the family of laws of the sequence $\{\rho^{\flat,\varepsilon}\}_{\varepsilon>0}$ is tight on $L^{1}_{[0, T]}L^1_{\mathbb{T}^2}$ with respect to the strong topology.
\end{lemma}

With the tightness of the measures $\nu^{\varepsilon,\rho}$ obtained above, we intend to use the Skorohod representation to construct a large enough probability space, on which we construct the renormalized kinetic solutions to \eqref{SCPDE} and send $\varepsilon \to 0$.  We now demonstrate that the convergence of those solutions to a solution of the skeleton equation \eqref{eq-riesz} with respect to the strong topology on $L^{1}_{[0, T] } L^{1}_{\mathbb{T}^2}$ in the following Proposition. 

\begin{proposition}\label{C1}Let the interaction $\mathcal{V}[\cdot]$ be defined by \eqref{ker-V}, and let $K=K(\varepsilon)$ and $\eta=\eta(\varepsilon)$ satisfy $K(\varepsilon)\to\infty$, $\varepsilon N_{K(\varepsilon)}\to0$, and $\eta(\varepsilon)\to0$.
Given any initial condition $\rho_{0}$ be of finite entropy, Let $(\rho^{\flat,\varepsilon},\xi)_{\varepsilon > 0}$ be defined on a probability space $(\Omega, \mathcal{F}, (\mathcal{F}_t)_{t \in [0,T]},\mathbb{P})$, such that $\rho^{\flat,\varepsilon}$ is a probabilistically weak renormalized solution to \eqref{SCPDE}. We further assume that laws of the sequence $\{\rho^{\flat,\varepsilon}\}_{\varepsilon>0}$ is tight on $L^{1}_{[0, T]}L^1_{\mathbb{T}^2}$ with respect to the strong topology. Then we have convergence in laws, 
\begin{align}
\rho^{\flat,\varepsilon} \rightarrow \rho^{\flat},
\end{align}
strongly in $L^1_{[0,T]}L^1_{\mathbb{T}^2}$, as $\varepsilon \to 0$, where $\rho^{\flat}$ is a kinetic solution of the skeleton equation \eqref{eq-riesz} with initial data $\rho_0$.  
\end{proposition}
\begin{proof}
Following the same argument as in \cite[Theorem 5.25]{FG24} and with the aid of Jakubowski--Skorokhod representation theorem \cite{Jak97}, it follows from Lemma \ref{tight} and Proposition \ref{entropy estimate} that there exists a filtered probability space on which for each $\varepsilon > 0$, $(\rho^{\flat, \varepsilon}, m^{\flat,\varepsilon}, \xi^{\varepsilon})$ solves \eqref{MC kenitic solution}, and further
\begin{align}\label{CM rho L1}
\rho^{\flat,\varepsilon}\rightarrow \rho^{\flat}\text { strongly in } L^1_{[0, T]}L^1_{\mathbb{T}^2} \ \text{and point-wise}, \text{a.s.}
\end{align}
\begin{align}\label{CM nabla rho L2H1}
\nabla\sqrt{\rho^{\flat,\varepsilon}} \rightarrow \nabla \sqrt{\rho^{\flat}} \text{ weakly in } L^2_{[0, T]} L^2_{\mathbb{T}^2}, \text{a.s.} \text{ and in } L^2(\Omega, L^2_{[0, T]} L^2_{\mathbb{T}^2}),
\end{align}
and in distributional sense, 
\begin{align}\label{CM kinetic measure}
m^{\flat,\varepsilon}\geq \delta_0(\rho^{\flat,\varepsilon}-\zeta)|\nabla\rho^{\flat,\varepsilon}|^2, \text{a.s.}.	
\end{align}

It follows directly from the preservation of mass for $\rho^{\flat, \varepsilon}$, \eqref{CM rho L1} and \eqref{CM nabla rho L2H1} that \eqref{eq:skeleton-l1-fisher} holds. To construct the kinetic measure in the limit, we define for every $\psi\in\mathrm{C}_{c}^{\infty}\left(\mathbb{T}^2\times(0,\infty)\right)$ the map 
\begin{align*}
H_{\psi}:L^{1}\left([0, T] \times \mathbb{T}^2\right)\times \mathrm{C}\left([0, \infty) ; \mathbb{R}^2\right)^{\mathbb{Z}^{2}}\rightarrow\mathbb{R}
\end{align*}
as
\begin{align}
H_{\psi}(\rho^{\flat,\varepsilon},\xi^{\varepsilon}) & := - \int_{\mathbb{R}}\int_{\mathbb{T}^2} \chi^{\flat,\varepsilon}(x, \zeta, \cdot) \psi(x, \zeta)\Big|_0^t - \int_{0}^{t} \int_{\mathbb{R}} \int_{\mathbb{T}^2}\nabla \rho^{\flat,\varepsilon}\cdot\nabla_x\psi(x,\rho^{\flat,\varepsilon})
\nonumber\\
& - \int_{0}^{t} \int_{\mathbb{T}^2} \nabla \cdot (\rho^{\flat,\varepsilon} \mathcal{V}[\rho^{\flat,\varepsilon}])\psi(x, \rho^{\flat,\varepsilon})
+ \int_{0}^{t} \int_{\mathbb{T}^2} s_{\eta}(\rho^{\flat,\varepsilon})P_{K}g \cdot \nabla_x\psi(x, \rho^{\flat,\varepsilon})\nonumber\\
& + \int_{0}^{t} \int_{\mathbb{T}^2} s_{\eta}(\rho^{\flat,\varepsilon})P_{K}g \cdot \nabla\rho^{\flat,\varepsilon} \partial_\zeta\psi(x, \rho^{\flat,\varepsilon})+\sqrt{\varepsilon}\int_{0}^{t} \int_{\mathbb{T}^2} s_{\eta}(\rho^{\flat,\varepsilon}) \xi_{K}\cdot \nabla \psi(x, \rho^{\flat,\varepsilon})\nonumber\\
&-\frac{\varepsilon F_{1,K}}{2}\int^t_0\int_{\mathbb{T}^2}s_{\eta}'(\rho^{\flat,\varepsilon})^2\nabla\rho^{\flat,\varepsilon}\cdot(\nabla\psi)(x,\rho^{\flat,\varepsilon})+\frac{\varepsilon N_K}{2}\int_{0}^{t} \int_{\mathbb{T}^2} s_{\eta}(\rho^{\flat,\varepsilon})^2 \partial_{\zeta} \psi(x, \rho^{\flat,\varepsilon})
%+\int_{0}^{t} \int_{\mathbb{T}^2}\sqrt{\rho^{\flat,\varepsilon}}P_{K}g\cdot\nabla_x\psi(x, \rho^{\flat,\varepsilon})\nonumber\\
%&+2\int_{0}^{t} \int_{\mathbb{T}^2} \rho^{\flat,\varepsilon}\nabla \sqrt{\rho^{\flat,\varepsilon}}\cdot P_{K}g \partial_\zeta\psi(x, \rho^{\flat,\varepsilon})+\sqrt{\varepsilon}\int_{0}^{t} \int_{\mathbb{T}^2} \sqrt{\rho^{\flat,\varepsilon}} \xi_{K}\cdot \nabla \psi(x,\rho^{\flat,\varepsilon})\nonumber\\
%&-\frac{\varepsilon F_{1,K}}{8}\int^t_0\int_{\mathbb{T}^2}\frac{1}{\rho^{\flat,\varepsilon}}\nabla\rho^{\flat,\varepsilon}(\nabla\psi)(x,\rho^{\flat,\varepsilon})+\frac{\varepsilon N_K}{2}\int_{0}^{t} \int_{\mathbb{T}^2} \rho^{\flat,\varepsilon} \partial_{\zeta} \psi(x, \rho^{\flat,\varepsilon}),
\end{align}
so that $m^{\flat,\varepsilon}([0,t];\partial_{\zeta}\psi) = H_{\psi}(\rho^{\flat,\varepsilon},\xi^{\varepsilon})$. Using the same argument as in \cite[Theorem 6.2]{WZ24} and as $\varepsilon N_K\lesssim 1$, for every $r\in\mathbb{N}$ we have
\begin{align}\label{MC L1 KE'}	
\sup_{\varepsilon \in (0,1)}\mathbb{E}\left(m^{\flat,\varepsilon}([0,T]\times\mathbb{T}^2\times[0,r])\right)^2 \lesssim C(r), 
\end{align}
namely, $m^{\flat,\varepsilon}$ is uniformly bounded in $L^2(\Omega; \mathcal{M}_r)$, where $\Omega$ is the probability space constructed at the beginning of the proof and $\mathcal{M}_r$ denote the space of bounded Borel measures over $\mathbb{T}^2\times [0,T]\times [0,r]$ (with norm given by the total variation of measures). 

By the Banach--Alaoglu theorem and a diagonal processes, there exists a Radon measure $m^{\flat}$ on $\mathbb{T}^2\times [0,T]\times [0,\infty)$ such that $m^{\flat,\varepsilon}\rightharpoonup m^{\flat}$ -- weak${}^{\ast}$ in $L^2(\Omega; \mathcal{M}_r)$ for every $r\in \mathbb{N}$. Moreover, combining \eqref{CM rho L1}, \eqref{CM nabla rho L2H1} and \eqref{CM kinetic measure}, thanks to the lower semi-continuity of $L^2(\Omega;L^2_{[0,T]}L^2_{\mathbb{T}^2})$-norm, we have
\begin{equation}\label{MC KQ1}
  m^{\flat} \geq 4\delta_0(\zeta - \rho^{\flat})\zeta|\nabla\sqrt{\rho^{\flat}}|^2,\quad \text{a.s.}. 
\end{equation}
We now aim to show that the limit $(\rho^{\flat},m^{\flat})$ of $\{(\rho^{\flat,\varepsilon},m^{\flat,\varepsilon})\}_{\varepsilon > 0}$ is a renormalized kinetic solution of the skeleton equation \eqref{control-1} in the sense of Definition \ref{ske-kinetic}, as for any $\psi \in \mathrm{C}_{c}^{\infty}\left(\mathbb{T}^2 \times(0, \infty)\right)$, we have \eqref{MC kenitic solution} holds for $(\rho^{\flat,\varepsilon},\beta^{\flat,\varepsilon},m^{\flat,\varepsilon})$ for each fixed $\varepsilon > 0$. It follows from the definition of $\chi^{\flat,\varepsilon}(\cdot, \zeta) = \mathbf{1}_{\{0 < \zeta < \rho^{\flat,\varepsilon}\}}$ and \eqref{CM rho L1} that
\begin{align*}
\chi^{\flat,\varepsilon} \to \chi^{\flat}, \ \text{a.s. for a.e. }(t,x,\zeta)\in [0,T]\times\mathbb{T}^2\times(0,+\infty), 
\end{align*}
so we apply dominated convergence theorem and derive for almost every $t \in [0,T]$ that
\begin{align}\label{MC K P}
\lim_{\varepsilon\to 0} \mathbb{E}\left|\int_{\mathbb{R}}\int_{\mathbb{T}^2}\chi^{\flat,\varepsilon}(x,\zeta,t)\psi(x,\zeta)
-\int_{\mathbb{R}}\int_{\mathbb{T}^2} \chi^{\flat}(x,\zeta,t)\psi(x,\zeta)\right| = 0;
\end{align}
while by splitting the difference, we have
\begin{align}\label{b lap}
\mathbb{E}\left|\int_{0}^{t} \int_{\mathbb{T}^2} \nabla \rho^{\flat} \cdot\nabla_x \psi(x, \rho^{\flat})-\nabla \rho^{\flat,\varepsilon} \cdot\nabla_x\psi (x, \rho^{\flat,\varepsilon})\right|
\lesssim \mathrm{I}_1+\mathrm{I}_2,
\end{align}
with
\begin{align*}
\mathrm{I}_1 & = \mathbb{E}\int_{0}^{t} \int_{\mathbb{T}^2}\left|\left(\nabla \sqrt{\rho^{\flat}}-\nabla \sqrt{\rho^{\flat,\varepsilon}}\right) \cdot \sqrt{\rho^{\flat}}\nabla_x\psi  (x, \rho^{\flat})\right| \to 0, \\
\mathrm{I}_2 & = \mathbb{E}\int_{0}^{t} \int_{\mathbb{T}^2}\left|\nabla \sqrt{\rho^{\flat,\varepsilon}} \cdot\left(\sqrt{\rho^{\flat}}\nabla_x\psi  (x, \rho^{\flat})-\sqrt{\rho^{\flat,\varepsilon}} \nabla_x\psi  (x, \rho^{\flat,\varepsilon})\right)\right| \to 0.
\end{align*}
The convergence of $\mathrm{I}_1$ follows directly from \eqref{CM nabla rho L2H1} and 
$\sqrt{\rho^{\flat}}\nabla_x\psi  (x, \rho^{\flat}) \in L^{\infty}([0,T] \times \mathbb{T}^2)$, and for the term $\mathrm{I}_2$, by H\"older's inequality and the uniform entropy estimate Proposition \ref{entropy estimate}, 
\begin{align*}
\mathrm{I}_2 \lesssim \Big(\mathbb{E}\int_{0}^{t} \int_{\mathbb{T}^2}\left|\left(\sqrt{\rho^{\flat}}\nabla_x\psi  (x, \rho^{\flat})-\sqrt{\rho^{\flat, \varepsilon}} \nabla_x\psi  (x, \rho^{\flat,\varepsilon})\right)\right|^2\Big)^{1/2} \sup_{\varepsilon > 0} \Big(\mathbb{E}\int_{0}^{T} \int_{\mathbb{T}^2}\left|\nabla \sqrt{\rho^{\flat,\varepsilon}} \right|^2\Big)^{1/2},
\end{align*}
which converges to zero by \eqref{CM rho L1} and the dominated convergence theorem,
as the smoothness and compact support of $\psi$ leads to the bound
\begin{align*}
\left|\left(\sqrt{\rho^{\flat}}\nabla_x\psi  (x, \rho^{\flat})-\sqrt{\rho^{\flat, \varepsilon}} \nabla_x\psi  (x, \rho^{\flat,\varepsilon})\right)\right| \lesssim 1,
\end{align*}
which is uniform in $\varepsilon > 0$. For the kernel term, we have
\begin{align}\label{b ker}
\left|\int_{0}^{t} \int_{\mathbb{T}^2} \psi (x, \rho^{\flat,\varepsilon}) \nabla \cdot\left(\rho^{\flat,\varepsilon} \mathcal{V}[\rho^{\flat,\varepsilon}]\right) - \psi (x, \rho^{\flat}) \nabla \cdot\left(\rho^{\flat} \mathcal{V}[\rho^{\flat}]\right)\right| \lesssim \mathrm{II}_{1} + \mathrm{II}_{2},
\end{align}
with
\begin{align*}
\mathbb{E}\mathrm{II}_{1} & := \mathbb{E}\left|\int_{0}^{t} \int_{\mathbb{T}^2} \psi (x, \rho^{\flat,\varepsilon}) \nabla \rho^{\flat,\varepsilon} \cdot \mathcal{V}[\rho^{\flat,\varepsilon}] - \psi (x, \rho^{\flat}) \nabla \rho^{\flat} \cdot \mathcal{V}[\rho^{\flat}]\right| \to 0,\\
\mathrm{II}_{2} & := \left|\int_{0}^{t} \int_{\mathbb{T}^2} \psi (x, \rho^{\flat,\varepsilon}) \left(|\rho^{\flat,\varepsilon}|^2-\rho^{\flat,\varepsilon}\|\rho_0\|_{L^1_{\mathbb{T}^2}}\right) - \psi (x, \rho^{\flat}) \left(|\rho^{\flat}|^2-\rho^{\flat}\|\rho_0\|_{L^1_{\mathbb{T}^2}}\right)\right| \to 0, \ \text{a.s.}
\end{align*}
as $\varepsilon \to 0^+$. The convergence of $\mathrm{II}_{2}$ follows directly from the dominated convergence and \eqref{CM rho L1} thanks to the smoothness and compact support of $\psi$, and we may rewrite the integral in $\mathrm{II}_{1}$ by integration by part as 
\begin{align*}
\int_{\mathbb{T}^2} \psi (x, \rho^{\flat}) \nabla \rho^{\flat} \cdot \mathcal{V}[\rho^{\flat}] = \kappa_1\int_{\mathbb{T}^2} \int_0^{\rho^{\flat}}\psi (x,\zeta)\mathrm{d}\zeta \rho^{\flat}.
\end{align*}
The convergence of $\mathrm{II}_{1}$ then follows from 
\begin{align*}
\Big|\int_0^t\int_{\mathbb{T}^2}  \int_0^{\rho^{\flat,\varepsilon}}\psi (x,\zeta)\mathrm{d}\zeta (\rho^{\flat,\varepsilon} - \rho^{\flat} )\Big| \lesssim \| \rho^{\flat,\varepsilon} - \rho^{\flat} \|_{L^1_{[0,T]}L^1_{\mathbb{T}^2}} \to 0, 
\end{align*}
as $\xi \to  \int_0^{\xi}\psi (x,\zeta)\mathrm{d}\zeta$ is bounded; and since $\int_0^{\rho^{\flat,\varepsilon}}\psi (x,\zeta)\mathrm{d}\zeta \to \int_0^{\rho^{\flat}}\psi (x,\zeta)\mathrm{d}\zeta$ pointwise,
\begin{align*}
\int_0^t\int_{\mathbb{T}^2} \Big| \Big(\int_0^{\rho^{\flat,\varepsilon}}\psi (x,\zeta)\mathrm{d}\zeta - \int_0^{\rho^{\flat,\varepsilon}}\psi (x,\zeta)\mathrm{d}\zeta \Big) \rho^{\flat} \Big| \to 0,
\end{align*}
due to dominated convergence. For the first control term, we have $\nabla_x\psi(x, \rho^{\flat,\varepsilon}) s_{\eta}(\rho^{\flat,\varepsilon})\to\nabla_x\psi(x, \rho^{\flat}) \sqrt{\rho^{\flat}}$ point-wise, since $s_{\eta}$ approximates $\sqrt{\cdot}$ uniformly on compact sets of $(0,\infty)$, $\nabla_x\psi$ is compactly supported and $\rho^{\flat,\varepsilon} \to \rho^{\flat}$ point-wise according to \eqref{CM rho L1}, and therefore
\begin{align}\label{eq:conv-grho-11}
& \ \Big|\int_{0}^{t} \int_{\mathbb{T}^2} P_{K}\big(\sqrt{\rho} \nabla \Psi^{\rho}\big) \cdot \left[\nabla_x\psi(x, \rho^{\flat,\varepsilon}) s_{\eta}(\rho^{\flat,\varepsilon}) - \nabla_x\psi(x, \rho^{\flat}) \sqrt{\rho^{\flat}}\right]\Big| \nonumber\\
\leq & \ \|\sqrt{\rho} \nabla \Psi^{\rho}\|_{L^2_{[0,T]}L^2_{\mathbb{T}^2}}\left\|\nabla_x\psi(x, \rho^{\flat,\varepsilon}) s_{\eta}(\rho^{\flat,\varepsilon}) - \nabla_x\psi(x, \rho^{\flat}) \sqrt{\rho^{\flat}}\right\|_{L^2_{[0,T]}L^2_{\mathbb{T}^2}} \to 0, 
\end{align}
where the limit is due to dominated convergence and the boundedness of $(x, \zeta) \mapsto \nabla_x\psi(x, \zeta) s_{\eta}(\zeta)$ uniformly in $\eta \geq 0$. Meanwhile, 
\begin{align}\label{eq:conv-grho-12}
\Big|\int_{0}^{t} \int_{\mathbb{T}^2} (1 - P_{K})\big(\sqrt{\rho} \nabla \Psi^{\rho}\big) \cdot \nabla_x\psi(x, \rho^{\flat}) \sqrt{\rho^{\flat}}\Big| \lesssim \|(1 - P_{K})\big(\sqrt{\rho} \nabla \Psi^{\rho}\big)\|_{L^2_{[0,T]}L^2_{\mathbb{T}^2}} \to 0,
\end{align}
as $\sqrt{\rho} \nabla \Psi^{\rho} \in L^2([0,T] \times \mathbb{T}^2)$. Similarly, we have the bound for the second term with control that
\begin{align}\label{eq:conv-grho-21}
\lim_{\varepsilon \to 0} \int_{0}^{t} \int_{\mathbb{T}^2} s_{\eta}(\rho^{\flat,\varepsilon})\nabla\rho^{\flat,\varepsilon} \cdot (1 - P_{K}) \big(\sqrt{\rho} \nabla \Psi^{\rho}\big)\partial_\zeta\psi(x, \rho^{\flat,\varepsilon}) = 0,
\end{align}
as $s_{\eta}(\rho^{\flat,\varepsilon})\nabla\rho^{\flat,\varepsilon} \partial_\zeta\psi(x, \rho^{\flat,\varepsilon})$ is bounded in $L^2([0,T] \times \mathbb{T}^2)$ uniformly in $\eta \geq 0$ and $\varepsilon > 0$ according to \eqref{CM nabla rho L2H1}; also directly by \eqref{CM nabla rho L2H1}
\begin{align}\label{eq:conv-grho-22}
\left|\int_{0}^{t} \int_{\mathbb{T}^2} \rho^{\flat} \left(\nabla\sqrt{\rho^{\flat,\varepsilon}} -\nabla\sqrt{\rho^{\flat}}\right)\cdot \nabla \Psi^{\rho} \sqrt{\rho}\partial_\zeta\psi(x, \rho^{\flat})\right| \to 0.
\end{align}
Furthermore, using dominated convergence, we have
\begin{align}
& \ \left|\int_{0}^{t} \int_{\mathbb{T}^2} \left(\rho^{\flat}\partial_\zeta\psi(x, \rho^{\flat})-\rho^{\flat,\varepsilon}\partial_\zeta\psi(x, \rho^{\flat,\varepsilon})\right) \nabla\sqrt{\rho^{\flat,\varepsilon}}\cdot  \nabla \Psi^{\rho} \sqrt{\rho} \right| \nonumber\\
\lesssim & \ \left\| \left(\rho^{\flat}\partial_\zeta\psi(x, \rho^{\flat})-\rho^{\flat,\varepsilon}\partial_\zeta\psi(x, \rho^{\flat,\varepsilon})\right) \nabla \Psi^{\rho} \sqrt{\rho} \right\|_{L^2_{[0,T]}L^2_{\mathbb{T}^2}} \to 0, 
\end{align}
since $\rho^{\flat}\partial_\zeta\psi(x, \rho^{\flat}) \to \rho^{\flat,\varepsilon}\partial_\zeta\psi(x, \rho^{\flat,\varepsilon})$ pointwise and $(x, \zeta) \mapsto \zeta \partial_\zeta\psi(x, \zeta)$ is bounded. In combination of \eqref{eq:conv-grho-11}--\eqref{eq:conv-grho-22}, we have proved that
\begin{equation}\label{eq:conv-grho}
  \lim_{\varepsilon\rightarrow 0} \int_{0}^{t} \int_{\mathbb{T}^2} s_{\eta}(\rho^{\flat,\varepsilon})\nabla\rho^{\flat,\varepsilon} \cdot P_{K}\big(\sqrt{\rho} \nabla \Psi^{\rho}\big)\partial_\zeta\psi(x, \rho^{\flat,\varepsilon}) = \int_{0}^{t} \int_{\mathbb{T}^2} \sqrt{\rho^{\flat}}\nabla\rho^{\flat} \cdot  P_{K} \big(\sqrt{\rho} \nabla \Psi^{\rho}\big)\partial_\zeta\psi(x, \rho^{\flat}),
\end{equation}
for all $t \in [0,T]$ a.s. For the martingale term, we apply the Burkholder-Davis-Gundy Inequality and derive
\begin{align}\label{MC M P}
\mathbb{E}\left[\sup_{t\in[0,T]}\left|\sqrt{\varepsilon}\int_0^t
\int_{\mathbb{T}^2}\nabla \left[\psi(x,\rho^{\flat,\varepsilon})\right] \cdot \xi_{K} \right|^2 \right] & \lesssim \varepsilon \mathbb{E} \int_0^T\int_{\mathbb{T}^2} |\nabla_x\psi(x,\rho^{\flat,\varepsilon})|^2 + \left|\partial_{\zeta}\psi(x,\rho^{\flat,\varepsilon})\nabla{\sqrt{\rho^{\flat,\varepsilon}}}\right|^2\rho^{\flat,\varepsilon}  \nonumber\\
& \lesssim \varepsilon \left(1 +  \|\nabla{\sqrt{\rho^{\flat,\varepsilon}}}\|^2_{L^2_{[0,T]}L^2_{\mathbb{T}^2}}\right) \to 0;
\end{align}
while the scaling $\varepsilon N_K\rightarrow 0$ ensures that \begin{align}\label{MC V P}
\lim_{\varepsilon\rightarrow 0}\frac{\varepsilon N_K}{2}\int_{0}^{t} \int_{\mathbb{T}^2} \rho^{\flat,\varepsilon} \partial_{\zeta} \psi(x, \rho^{\flat,\varepsilon}) = 0.
\end{align}
Finally, for every $M\geq1$, by taking a sequence of smooth approximations of $I_{\{[M,M+1]\}}$ as the test function, following the approach in \cite[Theorem 6.2]{WZ24}, we derive that a.s.
\begin{align}\label{MC KQ2}
\liminf _{M \rightarrow \infty}  m^{\flat}\left([0, T] \times \mathbb{T}^2 \times[M, M+1]\right)=0. 
\end{align}
Then it follows from Fatou's lemma that
\begin{align*}
  &\liminf _{M \rightarrow \infty} \mathbb{E} m^{\flat}\left([0, T] \times \mathbb{T}^2 \times[M, M+1]\right)=0.
\end{align*}
Therefore, in combination with \eqref{MC KQ1} and \eqref{MC KQ2},  it follows that the measure $m^{\flat}$ satisfies the conditions in Definition \ref{ske-kinetic}. This completes the proof. 
\end{proof}

The concentration of measures is then established, once we prove that the given $\rho$ is the unique renormalized kinetic solution to \eqref{eq-riesz} in the sense of Definition \ref{ske-kinetic} with control $g = \sqrt{\rho} \nabla \Psi^{\rho}$, since then for any subsequence $\{\nu^{\varepsilon, \rho}\}_{\varepsilon > 0}$, we have found a subsubsequence that converges in law to $ \rho^{\flat}$ in $L^1_{[0,T]}L^1_{\mathbb{T}^2}$ with the strong topology. We claim that the following Proposition holds, the proof of which is postponed to the next section \ref{sec-ws!}. 
\begin{proposition}\label{L1uniq}
Given interaction $\mathcal{V}[\cdot]$ as in \eqref{ker-V} under the condition \eqref{eq:small-KS}, we let $\rho^1_0,\rho^2_0 \in \mathrm{Ent}(\mathbb{T}^2)$. If $\rho^1$ and $\rho^2$ are renormalized kinetic solutions of  \eqref{control-1} in the sense of Definition \ref{ske-kinetic} with control $g$ and initial condition initial  $\rho^1_0$ and $\rho^2_0$ respectively, and further 
\begin{align}\label{eq:strong-rho2}
\rho^2 \in \mathcal{C}_0,
\end{align}
then we have
\begin{equation}\label{qq-31-1}
\|\rho^1-\rho^2\|_{L^{\infty}_{[0,T]}L^1_{\mathbb{T}^2}}\lesssim \|\rho_0^1-\rho_0^2\|_{L^1_{\mathbb{T}^2}}, 
\end{equation}
for some implicit constant which depends on $\rho^2$.
\end{proposition}

With the aid of Lemma \ref{ske-lem}, any given $\rho \in \mathcal{C}_0$ is both a weak solution and a renormalized kinetic solution to \eqref{control-1} with control $g = \sqrt{\rho}\nabla\Psi^{\rho}$. The weak--strong uniqueness result then implies that the limit $\rho^{\flat}$ in Proposition \ref{C1} coincides with $\rho$. The large deviations lower bound restricted to $\mathcal{C}_0$ then follows directly from Lemma \ref{entropymethod}.
\begin{proof}[Proof of \eqref{RS E} in Theorem \ref{thm:LDP}]
It follows from Lemma \ref{C1} that 
\begin{equation}\label{RL MC}
\nu^{\varepsilon;\rho}\rightharpoonup\delta_{\rho},
\end{equation}
weakly in $\mathcal{P}(L^1([0,T] \times \mathbb{T}^2))$ as $\varepsilon\rightarrow0$, and \eqref{entropy-check} verifies the condition \eqref{eq:relative-entropy-bound} in Lemma \ref{entropymethod}, from which \eqref{RS E} follows.
\end{proof}

\section{Weak-strong uniqueness}\label{sec-ws!}
A function $\rho$ is called a strong renormalized kinetic solution of Eq. (\ref{control-1}) with initial  $\rho_0$ and control $g$, if $\rho$ is a renormalized kinetic solution of it in the sense of Definition \ref{ske-kinetic} with initial  $\rho_0$ and control $g$, and further 
\begin{align*}
\rho\in L^\infty_{[0,T]}L^\infty_{\mathbb{T}^2}\cap L^\infty_{[0,T]}W^{1,p'}_{\mathbb{T}^2}, 
\end{align*}
for an arbitrary $p'>2$. We recall from \eqref{S-intro} the definition of the restriction set as 
\begin{align*}
\mathcal{C}_0 = \bigcup_{p'>2}\left\{\rho\in L_{[0,T]}^{\infty}L^1_{\mathbb{T}^2} : \rho\in L_{[0,T]}^{\infty}L^{\infty}_{\mathbb{T}^2} \cap L_{[0,T]}^{\infty}W_{\mathbb{T}^2}^{1,p'}\right\},
\end{align*}
and therefore a strong renormalized kinetic solution is just a renormalized kinetic solution that lives in the restriction space $\mathcal{C}_0$. 

In the following, we establish a decay estimate for the kinetic measure at small values, which will be used in the proof of weak--strong uniqueness. Since the result can be proved in a similar manner to \cite[Theorem 4.6]{FG24}, we omit the proof. 
\begin{lemma}\label{KM V0 remark}
Let $\rho$ be a renormalized kinetic solution of \eqref{control-1} in the sense of Definition \ref{ske-kinetic} and let $m$ be  the corresponding kinetic measure. Then we have
\begin{equation}\label{KM V0}
  \lim _{\beta \rightarrow 0}\left[\beta^{-1} m\left(\mathbb{T}^2 \times[0, T] \times[\beta / 2, \beta]\right)\right]=0.
\end{equation}

\end{lemma}

The weak-strong uniqueness of skeleton equation \eqref{control-1} can now be proved.
\begin{proof}[Proof of Proposition \ref{L1uniq}]
We fix $p \in (1,2)$ and let $p' > 1$ denote the H\"older conjugate of $p$. Let $\chi^i$ and $m^i$ be the renormalized kinetic function and the kinetic measure of $\rho^i$ respectively, for $i = 1,2$. Let $\kappa^{\mathrm{s}}$ and $\kappa^{\mathrm{v}}$ be nonnegative smooth kernels on $\mathbb{T}^2$ and $\mathbb{R}^2$ respectively. For every $ \nu,\delta\in(0,1)$, we further define 
\begin{align*} 
\kappa_{\mathrm{s}}^{\nu}(x,y) := \frac{1}{\nu^{2}}\kappa^{\mathrm{s}}\left(\frac{\mathrm{d}_{\mathbb{T}^2}(x,y)}{\nu}\right), \quad \kappa_{\mathrm{v}}^{\delta}(\xi,\eta) := \frac{1}{\delta}\kappa^{\mathrm{v}}\left(\frac{\xi - \eta}{\delta}\right),
\end{align*}
and
\begin{align*}
\kappa^{\nu,\delta}(x,y,\xi,\eta) := \kappa_{\mathrm{s}}^{\nu}(x,y) \kappa_{\mathrm{v}}^{\delta}(\xi,\eta), \quad \chi^{i,\nu,\delta}(y,\eta) := \int_{\mathbb{R}}\int_{\mathbb{T}^2}\chi^i(x,\xi)\kappa^{\nu,\delta}(x,y,\xi,\eta)\mathrm{d}x\mathrm{d}\xi,
\end{align*}
for $i = 1,2$ and $\nu,\delta\in(0,1)$.
For every $M>0$, let the cutoff function $\zeta^M:\mathbb{R}\rightarrow[0,1]$ be defined such that 
\begin{align*}
\zeta^M(\xi) = 0, \ \text{if} \ \xi \leq \frac{1}{M} \ \text{or} \ \xi > M+1; \quad \zeta^M(\xi)  = 1, \ \text{if} \ \frac{2}{M} \leq \xi \leq M, 
\end{align*}
and linear interpolation in between. Then almost surely for almost every $t \in (0,T]$, by the dominated convergence theorem, we have that almost surely, 
\begin{align}\label{qq-28} 
\int_{\mathbb{R}}\int_{\mathbb{T}^2}|\chi_t^1-\chi_t^2|^2\zeta^M(\eta)\mathrm{d}y\mathrm{d}\eta = \lim_{\nu,\delta\rightarrow0}\int_{\mathbb{R}}\int_{\mathbb{T}^2}\left|\chi_t^{1,\nu,\delta}-\chi_t^{2,\nu,\delta}\right|^2\zeta^M(\eta)\mathrm{d}y\mathrm{d}\eta = \lim_{\nu,\delta\rightarrow0}\mathbb{I}_t^{\nu,\delta,M},
\end{align}
where 
\begin{equation}\label{rr-11}
\mathbb{I}_t^{\nu,\delta,M}
= \int_{\mathbb{R}}\int_{\mathbb{T}^2}\Big(\chi_t^{1,\nu,\delta}
+\chi_t^{2,\nu,\delta}-2\chi_t^{1,\nu,\delta}\chi_t^{2,\nu,\delta}\Big)\zeta^M(\eta)\mathrm{d}y\mathrm{d}\eta,
\end{equation}
for $M\in(0,\infty)$, $\delta\in(0,1/M)$ and $\nu\in(0,1)$. It follows from Definition \ref{ske-kinetic} of renormalized kinetic solution that, it holds as distributions on $\mathbb{T}^2\times \mathbb{R}\times[0,T]$ for $i = 1,2$,
\begin{align}\notag
 \partial_t \chi^{i,\nu,\delta}_t(y,\eta) &= \int_{\mathbb{R}}\int_{\mathbb{T}^2}\chi_t^i\Delta_x\kappa^{\nu,\delta}(x,y,\xi,\eta)\mathrm{d}x\mathrm{d}\xi
  \notag
  -\int_{\mathbb{R}}\int_{\mathbb{T}^2}m^i_t\partial_{\xi}\kappa^{\nu,\delta}(x,y,\xi,\eta)\mathrm{d}x\mathrm{d}\xi\\
  \notag
  &+2\int_{\mathbb{T}^2}\rho^ig(x,t)\cdot\nabla_x\sqrt{\rho^i}\partial_{\xi}\kappa^{\nu,\delta}(x,y,\rho^i,\eta)\mathrm{d}x\mathrm{d}\xi\\
  &+\int_{\mathbb{T}^2}\sqrt{\rho^i}g(x,t)\cdot\nabla_x\kappa^{\nu,\delta}(x,y,\rho^i,\eta)\mathrm{d}x\mathrm{d}\xi -\int_{\mathbb{T}^2} \nabla_x\cdot\left(\rho^i\mathcal{V}[\rho^i]\right) \kappa^{\nu,\delta}(x,y,\rho^i,\eta)\mathrm{d}x.  \label{kineticsmooth}
\end{align}
We use variables $(x, \xi) \in \mathbb{T}^2 \times \mathbb{R}$ for the kinetic function $\chi^1$, and variables $\left(x^{\prime}, \xi^{\prime}\right) \in \mathbb{T}^2 \times \mathbb{R}$ for the kinetic function $\chi^2$ for clarity, and let
\begin{align*}
\bar{\kappa}_{t, i}^{\nu, \delta}(\cdot, y, \eta)=\kappa^{\nu, \delta}\left(\cdot, y, \rho^i(\cdot, t), \eta\right),
\end{align*}
for $i = 1,2$. Proceeding as \cite[Theorem 8]{FG23},  we can decompose the temporal derivative of \eqref{rr-11} into the five terms
\begin{equation}\label{full}
  \partial_t\mathbb{I}^{\nu,\delta,M}_t=\partial_t\mathbb{I}^{\nu,\delta,M}_{t,\mathrm{par}}
  +\partial_t\mathbb{I}^{\nu,\delta,M}_{t,\mathrm{hyp}}+\partial_t\mathbb{I}^{\nu,\delta,M}_{t,\mathrm{con}}
  +\partial_t\mathbb{I}^{\nu,\delta,M}_{t,\mathrm{vel}}+\partial_t\mathbb{I}^{\nu,\delta,M}_{t,\mathrm{ker}},
\end{equation}
where the first four terms are the same as in \cite[(25)]{FG23}, thus we omit their detailed expressions to avoid repeating them. It is sufficient to deal with the kernel term
\begin{align}\label{t-1}
\partial_t\mathbb{I}^{\varepsilon,\delta,M}_{t,ker}=&\int_{\mathbb{R}}\int_{(\mathbb{T}^2)^2} \nabla_x\cdot ( \rho^1 \mathcal{V}[\rho^1] ) \bar{\kappa}^{\varepsilon,\delta}_{1,t}(2\chi_t^{2,\varepsilon,\delta}-1)\zeta^M(\eta)\mathrm{d}x\mathrm{d}y\mathrm{d}\eta\nonumber\\
 &+\int_{\mathbb{R}}\int_{(\mathbb{T}^2)^2} \nabla_{x'}\cdot (\rho^2 \mathcal{V}[\rho^2]) \bar{\kappa}^{\varepsilon,\delta}_{2,t}(2\chi_t^{1,\varepsilon,\delta}-1)\zeta^M(\eta)\mathrm{d}x'\mathrm{d}y\mathrm{d}\eta.
\end{align}
Recall that for every  $t\in[0,T]$, $\|\rho^1(t)\|_{L^1_{\mathbb{T}^2}}=\|\rho^1_0\|_{L^1_{\mathbb{T}^2}}$ and the weak derivative $\nabla\rho^1=2\sqrt{\rho^1}\nabla\sqrt{\rho^1}$ holds for almost every $(x,t)\in\mathbb{T}^2\times[0,T]$. For every $t\in[0,T]$, $M\in(0,\infty)$ and $\delta\in(0,\frac{1}{M})$, by the definitions of $\kappa^{\delta}$, $\zeta^M$, and following a similar estimate as \eqref{eq:interaction is L1}, we deduce that $\nabla_x\cdot ( \rho^1 \mathcal{V}[\rho^1]) {\kappa}^{\delta}(\rho^1(x),\eta)\zeta^M(\eta)$ is $L^1_{[0,T]\times\mathbb{R}\times\mathbb{T}^2}$-integrable. Similarly,  $\nabla_{x'}\cdot ( \rho^2 \mathcal{V}[\rho^2]) {\kappa}^{\delta}(\rho^2(x'),\eta)\zeta^M(\eta)$ is also $L^1_{[0,T]\times\mathbb{R}\times\mathbb{T}^2}$-integrable.
By  the
definition of $\kappa^{\varepsilon,\delta}$, the boundedness of the kinetic functions and the dominated convergence theorem,  taking $\varepsilon\rightarrow 0$ in \eqref{t-1}, it gives
\begin{align}\label{WS kernel delta}
\lim_{\varepsilon\rightarrow 0}\mathbb{I}^{\varepsilon,\delta,M}_{t,ker}
=
 &\int^t_0\int_{\mathbb{R}}\int_{\mathbb{T}^2}(2\chi^{2,\delta}_t(y)-1)  \nabla_y\cdot (\rho^1 \mathcal{V}[\rho^1] ) \bar{\kappa}^{\delta}_{1,t}\zeta^M(\eta)\mathrm{d}y\mathrm{d}\eta \mathrm{d}s\\ \notag
&+\int^t_0\int_{\mathbb{R}}\int_{\mathbb{T}^2}(2\chi^{1,\delta}_t(y)-1)  \nabla_y\cdot(\rho^2  \mathcal{V}[\rho^2] ) \bar{\kappa}^{\delta}_{2,t}\zeta^M(\eta)\mathrm{d}y\mathrm{d}\eta \mathrm{d}s,
\end{align}
where
$
\chi_t^{i,\delta}(y,\eta)=\int_{\mathbb{R}}\chi^i_t(y,\xi)\kappa^{\delta}(\xi,\eta)\mathrm{d}\xi$ and $\bar{\kappa}^{\delta}_{i,t}=\kappa^{\delta}(\rho^i(y,t),\eta)$, for $i\in\{1,2\}$.

For every  $M\in(0,\infty)$ and $\delta\in(0,\frac{1}{M})$, the definition of $\chi_t^{2,\delta}$, $\bar{\kappa}^{\delta}_{1,t}$ and $\zeta_M$ imply that there exists a constant $c>0$ such that
\begin{align}\label{rho-1}
  &\Big|\int^t_0\int_{\mathbb{R}}\int_{\mathbb{T}^2}(2\chi^{2,\delta}_t-1)  \nabla_y\cdot(\rho^1  \mathcal{V}[\rho^1] ) \bar{\kappa}^{\delta}_{1,t}(\zeta^M(\eta)-\zeta^M(\rho^1))\mathrm{d}y\mathrm{d}\eta \mathrm{d}s\Big|\\
  \notag
  \leq&\int^t_0\int_{\mathbb{R}}\int_{\mathbb{T}^2}\Big| \nabla_y\cdot(\rho^1 \mathcal{V}[\rho^1]) \bar{\kappa}^{\delta}_{1,t}(\zeta^M(\eta)-\zeta^M(\rho^1))\Big|\mathrm{d}y\mathrm{d}\eta \mathrm{d}s\\
  \notag
 \leq &c\delta \int^t_0 \int_{\mathbb{T}^2}|\nabla_y\cdot(\rho^1 \mathcal{V}[\rho^1] )| (MI_{\{\frac{1}{M}-\delta<\rho^1<\frac{2}{M}+\delta\}}+I_{\{M-\delta<\rho^1<M+1+\delta\}} ) \mathrm{d}y \mathrm{d}s.
\end{align}
For every fixed $M\in(0,\infty)$, following a computation similar to \eqref{eq:interaction is L1}, we get for all $\delta\in(0,1/M)$, there exists  a constant $C$  independent of $\delta$ such that,
\begin{equation}\label{WS V q1}
  \int^t_0 \int_{\mathbb{T}^2}|\nabla_y\cdot(\rho^1 \mathcal{V}[\rho^1] )| (MI_{\{\frac{1}{M}-\delta<\rho^1<\frac{2}{M}+\delta\}}+I_{\{M-\delta<\rho^1<M+1+\delta\}} ) \mathrm{d}y \mathrm{d}s<C.
\end{equation}
Therefore, taking $\delta\rightarrow 0$ in \eqref{rho-1}, we obtain
\begin{align}\label{t-2}
  \lim_{\delta\rightarrow 0}\left|\int^t_0 \int_{\mathbb{R}}\int_{\mathbb{T}^2}\bar{\kappa}^{\delta}_{1,s}(2\chi^{2,\delta}_s-1)\nabla_y\cdot(\rho^1\mathcal{V}[\rho^1]) (\zeta^M(\eta)-\zeta^M(\rho^1))\mathrm{d}y\mathrm{d}\eta \mathrm{d}s\right|=0.
  \end{align}
  Similarly, for the second term of \eqref{WS kernel delta},
  \begin{align}\label{t-3}
  \lim_{\delta\rightarrow 0}\left|\int^t_0 \int_{\mathbb{R}}\int_{\mathbb{T}^2}
  \bar{\kappa}^{\delta}_{2,s}(2\chi^{1,\delta}_s-1)\nabla_y\cdot(\rho^2\mathcal{V}[\rho^2]) (\zeta^M(\eta)-\zeta^M(\rho^2))\mathrm{d}y\mathrm{d}\eta \mathrm{d}s\right|=0.
  \end{align}

Furthermore, refer to \cite[(4.22)]{FG24}, when $2\delta<\rho^2(y,s)$, we have that 
\begin{align*}
\int_{\mathbb{R}^2}\kappa^{\delta}(\xi-\eta)\kappa^{\delta}(\eta-\xi')\chi^{2}_s(y,\xi')\mathrm{d}\eta
\mathrm{d}\xi'=
\left\{
  \begin{array}{ll}
   0, & {\rm{if}}\ \xi\leq -2\delta\ {\rm{or}}\ \xi\geq \rho^2(y,s)+2\delta, \\
    1/2, & {\rm{if}}\ \xi=0\ {\rm{or}}\ \xi\geq \rho^2(y,s), \\
    1, & {\rm{if}}\ 2\delta<\xi< \rho^2(y,s)-2\delta.
  \end{array}
\right.
\end{align*}
With the aid of the fact that $\zeta^M(0)=0$, we obtain that for almost every $(t,x)\in[0,T]\times\mathbb{T}^2$, 
\begin{align}\label{t-4}
  \lim_{\delta\rightarrow 0}\left(\int_{\mathbb{R}}\bar{\kappa}^{\delta}_{1,s}(2\chi^{2,\delta}_s-1)\mathrm{d}\eta\right)\zeta_M(\rho^1)
  =(I_{\rho^1=\rho^2}+2I_{\rho^1<\rho^2}-1)\zeta^M(\rho^1).
\end{align}
For every fixed $M\in(0,\infty)$ and $\delta\in(0,1/M)$,  the definition of $\chi_t^{2,\delta}$, $\bar{\kappa}^{\delta}_{1,t}$ and $\zeta_M$ imply that
\begin{align}
&\left|\left(\int_{\mathbb{R}}\bar{\kappa}^{\delta}_{1,s}(2\chi^{2,\delta}_s-1)\mathrm{d}\eta\right)\nabla_y\cdot
(\rho^1\mathcal{V}[\rho^1]) \zeta_M(\rho^1)\right|\nonumber\\
\leq& |\nabla_y\cdot(\rho^1 \mathcal{V}[\rho^1] )| (MI_{\{\frac{1}{M}<\rho^1<\frac{2}{M}\}}+I_{\{M<\rho^1<M+1\}} ),
\end{align}
one may find the above inequality is familiar with \eqref{WS V q1}. Thus, similar estimates to \eqref{WS V q1} show that $\nabla_y\cdot(\rho^1\mathcal{V}[\rho^1]) \zeta_M(\rho^1)$ is $L^1_{[0,T]\times \mathbb{T}^2}$ integral. Combining (\ref{t-2}), (\ref{t-3}), (\ref{t-4}),  after passing to $\delta\rightarrow 0$, by the dominated convergence theorem,  we have
\begin{equation}\label{WS K M}
  \lim_{\delta\rightarrow 0}( \lim_{\varepsilon\rightarrow 0}\mathbb{I}^{\varepsilon,\delta,M}_{t,ker})=K_{11}+K_{12}+K_{21}+K_{22},
\end{equation}
 where
\begin{align*}
  K_{11}+K_{12}
  =&\int^t_0\int_{\mathbb{T}^2}(I_{\rho^1=\rho^2}+2I_{ \rho^1<\rho^2}-1)\zeta^M(\rho^1)
   \nabla\rho^1\cdot(\mathcal{V}[\rho^1])\mathrm{d}y\mathrm{d}s\\
   &+\int^t_0\int_{\mathbb{T}^2}(I_{\rho^1=\rho^2}+2I_{ \rho^1<\rho^2}-1)\zeta^M(\rho^1)
   \rho^1\nabla\cdot \mathcal{V}[\rho^1]\mathrm{d}y\mathrm{d}s\\
    =&\int^t_0\int_{\mathbb{T}^2}(I_{\rho^1=\rho^2}+2I_{ \rho^1<\rho^2}-1)\zeta^M(\rho^1)
   \nabla\rho^1\cdot\mathcal{V}[\rho^1]\mathrm{d}y\mathrm{d}s\\
   &+\int^t_0\int_{\mathbb{T}^2}(I_{\rho^1=\rho^2}+2I_{ \rho^1<\rho^2}-1)\zeta^M(\rho^1)
   \kappa_1\left(|\rho^1|^2-\rho^1\|\rho_0\|_{L^1_{\mathbb{T}^2}}\right)\mathrm{d}y\mathrm{d}s
\end{align*}
and
\begin{align*}
  K_{21}+K_{22}
   = &\int^t_0\int_{\mathbb{T}^2}(I_{\rho^2=\rho^1}+2I_{ \rho^2<\rho^1}-1)\zeta^M(\rho^2)
   \nabla\rho^2\cdot(\mathcal{V}[\rho^2])\mathrm{d}y\mathrm{d}s\\
    &+\int^t_0\int_{\mathbb{T}^2}(I_{\rho^2=\rho^1}+2I_{ \rho^2<\rho^1}-1)\zeta^M(\rho^2)
    \rho^2\nabla\cdot \mathcal{V}[\rho^2]\mathrm{d}y\mathrm{d}s\\
    = &\int^t_0\int_{\mathbb{T}^2}(I_{\rho^2=\rho^1}+2I_{ \rho^2<\rho^1}-1)\zeta^M(\rho^2)
   \nabla\rho^2\cdot\mathcal{V}[\rho^2]\mathrm{d}y\mathrm{d}s\\
    &+\int^t_0\int_{\mathbb{T}^2}(I_{\rho^2=\rho^1}+2I_{ \rho^2<\rho^1}-1)\zeta^M(\rho^2)
    \kappa_1\left(|\rho^2|^2-\rho^2\|\rho_0\|_{L^1_{\mathbb{T}^2}}\right)\mathrm{d}y\mathrm{d}s.
\end{align*}

For every $i\in\{1,2\}$ and $t\in[0,T]$, combined with the result $\|\rho^i(t)\|_{L^1_{\mathbb{T}^2}}=\|\rho^i_0\|_{L^1_{\mathbb{T}^2}}$,
 it then follows from \eqref{L2-ES} that 
 $\rho^i$ is $ L^2_{[0,T]}L^2_{\mathbb{T}^2}$-integrable.
Thanks to the dominated convergence theorem, we deduce that 
	\begin{align*} 
	\lim_{M\to\infty}\zeta^{M}(\rho^i)\left(|\rho^i|^2-\rho^i\|\rho_0\|_{L^1_{\mathbb{T}^2}}\right)=|\rho^i|^2-\rho^i\|\rho_0\|_{L^1_{\mathbb{T}^2}},\text{ strongly in }L^{1}_{\mathbb{T}^2\times[0,T]}.
	\end{align*}
Therefore, taking $M\rightarrow \infty$ in $K_{12}+K_{22}$, by the dominated convergence theorem and the fact $\sgn(\rho^2-\rho^1)=I_{\rho^1=\rho^2}+2I_{\rho^1<\rho^2}-1$, we obtain
\begin{align}\label{G1 2}
  \lim_{M\rightarrow\infty}(K_{12}+K_{22})
   =&\kappa_1\int^t_0\int_{\mathbb{T}^2}\sgn(\rho^2-\rho^1)
   (|\rho^1|^2-
   |\rho^2|^2)\mathrm{d}y\mathrm{d}s\notag\\
   &-\kappa_1\int^t_0\int_{\mathbb{T}^2}\sgn(\rho^2-\rho^1)\|\rho_0\|_{L^1_{\mathbb{T}^2}}(\rho^1-\rho^2)\notag\\
   \leq&\kappa_1\int^t_0\int_{\mathbb{T}^2}\sgn(\rho^2-\rho^1)
   (\rho^1-\rho^2)(\rho^1+\rho^2)\mathrm{d}y\mathrm{d}s\notag\\
   &+C(\rho_0,\kappa_1)\int^t_0\int_{\mathbb{T}^2}|\rho^1-\rho^2|.
   \end{align}
   We will revisit the first term in \eqref{G1 2} later on. 
   
 In the following, we estimate $K_{11}+K_{21}$ of \eqref{WS K M}. Let $J$ be defined by $$J=\int^t_0\int_{\mathbb{T}^2}(I_{\rho^1=\rho^2}+2I_{ \rho^1<\rho^2}-1)\zeta^M(\rho^2)
   \nabla\rho^2\cdot\mathcal{V}[\rho^1]\mathrm{d}y\mathrm{d}s,$$
then $K_{11}+K_{21}$ can be divided into the following two terms
\begin{equation}\label{WS K11-J}
K_{11}-J=\int_{0}^{t} \int_{\mathbb{T}^2} (I_{\rho^1=\rho^2}+2I_{ \rho^1<\rho^2}-1)\left[\zeta^{M}\left(\rho^{1}\right) \nabla \rho^{1}-\zeta^{M}\left(\rho^{2}\right) \nabla \rho^{2}\right]\cdot\mathcal{V}[\rho^{1}] \mathrm{d}y \mathrm{d} s,
\end{equation}
and
\begin{equation}\label{WS K21+J}
     K_{21}+J =\int^t_0\int_{\mathbb{T}^2}(I_{\rho^2=\rho^1}+2I_{ \rho^2<\rho^1}-1)\zeta^M(\rho^2)
   \nabla\rho^2\cdot \mathcal{V}[\rho^1-\rho^2]\mathrm{d}y\mathrm{d}s.
   \end{equation}

With the help of the assumption condition that $\rho^2\in\mathcal{C}_0$, we know that $\rho^2\in L^\infty_{[0,T]}W^{1,p'}_{\mathbb{T}^2}$, for some $p'>2$.  For every $i\in\{1,2\}$, it then follows from H\"older's inequality and the convolutional Young inequality that
\begin{align*}
  &\int^t_0\int_{\mathbb{T}^2}|\nabla\rho^2\cdot\mathcal{V}[\rho^i]|\mathrm{d}y\mathrm{d}t\leq \int^t_0\|\nabla\rho^2\|_{L^{p'}_{\mathbb{T}^2}(\mathbb{R}^2)}\|\mathcal{V}[\rho^i]\|_{L^p_{\mathbb{T}^2}(\mathbb{R}^2)}\mathrm{d}s\\
  \leq& \|\rho^2\|_{L^\infty_{[0,T]}W^{1,p'}_{\mathbb{T}^2}}\|V \|_{L^p_{\mathbb{T}^2}(\mathbb{R}^2)}\|\rho^i_0\|_{L^1_{\mathbb{T}^2}}T<\infty,
\end{align*}
where $\frac{1}{p}+\frac{1}{p'}=1$. Taking $M\rightarrow \infty$, by the dominated convergence theorem and the definition of $\zeta^M$, for every $i\in\{1,2\}$, we have
\begin{align*}
  &\lim_{M\rightarrow\infty}\int^T_0\int_{\mathbb{T}^2}|\zeta^{M}(\rho^i)\nabla\rho^2\cdot\mathcal{V}[\rho^i] -\nabla\rho^2\cdot\mathcal{V}[\rho^i]|\mathrm{d}y\mathrm{d}s\\
\leq &\lim_{M\rightarrow\infty}\int^T_0\int_{\mathbb{T}^2}(I_{\{|\rho^i|\leq \frac{1}{M}\}}+I_{\{|\rho^2|\geq M+1\}})|\nabla\rho^2\cdot\mathcal{V}[\rho^i] |\mathrm{d}y\mathrm{d}s= 0,
\end{align*}
which implies
	\begin{align*} \lim_{M\to\infty}\zeta^{M}(\rho^2)\nabla\rho^2\cdot\mathcal{V}[\rho^i]=\nabla\rho^2\cdot\mathcal{V}[\rho^i], \text{ strongly in }L^{1}\left(\mathbb{T}^{2}\times[0,T]\right).
	\end{align*}
Moreover,  taking $M\rightarrow \infty$ in \eqref{WS K21+J},  by the dominated convergence theorem and the fact $\sgn(\rho^2-\rho^1)=I_{\rho^1=\rho^2}+2I_{\rho^1<\rho^2}-1$, we obtain, for every $t\in[0,T]$,
\begin{equation}\label{G3}
 \lim_{M\rightarrow\infty}(K_{21}+J) =\int^t_0\int_{\mathbb{T}^2}\sgn(\rho^2-\rho^1)\nabla\rho^2
  \cdot \mathcal{V}[\rho^1-\rho^2].
\end{equation}
For every $t\in[0,T]$, by H\"{o}lder's inequality and the convolutional Young inequality, we get
\begin{align}\label{G3-1}
&\int^t_0\int_{\mathbb{T}^2}\sgn(\rho^2-\rho^1)\nabla\rho^2
   \mathcal{V}[\rho^1-\rho^2]\mathrm{d}y\mathrm{d}s
\leq\|\rho^2\|_{L^\infty_{[0,T]}W^{1,p'}_{\mathbb{T}^2}}
\int^t_0\|\mathcal{V}[\rho^1-\rho^2]\|_{L^p_{\mathbb{T}^2}}\nonumber\\
\leq&\|\rho^2\|_{L^\infty_{[0,T]}W^{1,p'}_{\mathbb{T}^2}}\|V\|_{L^p_{\mathbb{T}^2}(\mathbb{R}^2)}\int^t_0\|\rho^1-\rho^2\|_{L^1_{\mathbb{T}^2}}\mathrm{d}s.
\end{align}

It remains to estimate the term $K_{11}-J$. For every $i\in\{1,2\}$ and $M\in(0,\infty)$, following a similar estimate to \eqref{WS V q1}, we obtain that $\zeta^M(\rho^i)
   \nabla\rho^iV\ast \rho^1$ is $L^1_{[0,T]\times\mathbb{T}^2}$ integral. For every $\delta\in(0,1)$, let $\sgn^{\delta}=(\sgn\ast\kappa^{\delta})$,
 it then follows from the dominated convergence theorem and  the fact $\sgn(\rho^2-\rho^1)=I_{\rho^1=\rho^2}+2I_{\rho^1<\rho^2}-1$ that,
\begin{align}\label{K11 J}
  K_{11}-J=& \int_{0}^{t} \int_{\mathbb{T}^2} \operatorname{sgn}\left(\rho^{2}-\rho^{1}\right)\left[\zeta^{M}\left(\rho^{1}\right) \nabla \rho^{1}-\zeta^{M}\left(\rho^{2}\right) \nabla \rho^{2}\right]\cdot \mathcal{V}[\rho^{1}] \mathrm{d}y \mathrm{d}s \nonumber\\
=& \lim _{\delta \rightarrow 0} \int_{0}^{t} \int_{\mathbb{T}^2} \operatorname{sgn}^{\delta}\left(\rho^{2}-\rho^{1}\right)\left[\zeta^{M}\left(\rho^{1}\right) \nabla \rho^{1}-\zeta^{M}\left(\rho^{2}\right) \nabla \rho^{2}\right]\cdot \mathcal{V}[\rho^{1}] \mathrm{d}y \mathrm{d}s \nonumber\\
=& \lim _{\delta \rightarrow 0} \int_{0}^{t} \int_{\mathbb{T}^2} \operatorname{sgn}^{\delta}\left(\rho^{2}-\rho^{1}\right) \nabla\left\{\int^{\rho^{1}}_0 \zeta^{M}(\xi) \mathrm{d}\xi-\int^{\rho^{2}}_0 \zeta^{M}(\xi) \mathrm{d}\xi\right\}\cdot\mathcal{V}[\rho^{1}] \mathrm{d}y \mathrm{d}s \nonumber\\
=&-\lim _{\delta \rightarrow 0} \int_{0}^{t} \int_{\mathbb{T}^{2}}\left\{\nabla \operatorname{sgn}^{\delta}\left(\rho^{2}-\rho^{1}\right)\left[\int_{\rho^{2}}^{\rho^{1}} \zeta^{M}(\xi) \mathrm{d}\xi\right]\cdot \mathcal{V}[\rho^{1}]\right\} \mathrm{d}y \mathrm{d}s \nonumber\\
&-\lim _{\delta \rightarrow 0} \kappa_1\int_{0}^{t} \int_{\mathbb{T}^{2}}\left\{\operatorname{sgn}^{\delta}\left(\rho^{2}-\rho^{1}\right)\left[\int_{\rho^{2}}^{\rho^{1}} \zeta^{M}(\xi) \mathrm{d}\xi\right] \rho^{1}\right\} \mathrm{d}y \mathrm{d}s\notag\\
&+\lim_{\delta \rightarrow 0} \kappa_1\int_{0}^{t} \int_{\mathbb{T}^{2}}\left\{\operatorname{sgn}^{\delta}\left(\rho^{2}-\rho^{1}\right)\left[\int_{\rho^{2}}^{\rho^{1}} \zeta^{M}(\xi) \mathrm{d}\xi\right] \|\rho_0\|_{L^1_{\mathbb{T}^2}}\right\} \mathrm{d}y \mathrm{d}s
\end{align}
For every $i\in\{1,2\}$,  \eqref{L2-ES} implies that $\rho^i \rho^1$ is $L^1_{[0,T]\times\mathbb{T}^2}$ integrable. Therefore,
taking $\delta\rightarrow 0$ and $M\rightarrow \infty$, by the dominated convergence theorem and the definition of $\zeta^M$, we have that almost surely, 
\begin{align}\label{K11 J 1}
  -\lim_{M\rightarrow \infty}\lim_{\delta\rightarrow 0}\kappa_1\int^t_0\int_{\mathbb{T}^2}\sgn^{\delta}(\rho^2-\rho^1)\int^{\rho^1}_{\rho^2}\zeta^M(\xi)\mathrm{d}\xi
  \rho^1\mathrm{d}y\mathrm{d}s=-\kappa_1\int^t_0\int_{\mathbb{T}^2}\sgn(\rho^2-\rho^1)(\rho^1-\rho^2)
  \rho^1\mathrm{d}y\mathrm{d}s.
\end{align}
Combining \eqref{K11 J 1} with the first term in \eqref{G1 2}, we derive 
\begin{align*}
	&-\kappa_1\int^t_0\int_{\mathbb{T}^2}\sgn(\rho^2-\rho^1)(\rho^1-\rho^2)
  \rho^1\mathrm{d}y\mathrm{d}s+\kappa_1\int^t_0\int_{\mathbb{T}^2}\sgn(\rho^2-\rho^1)
   (\rho^1-\rho^2)(\rho^1+\rho^2)\mathrm{d}y\mathrm{d}s\\
   =&\kappa_1\int^t_0\int_{\mathbb{T}^2}\sgn(\rho^2-\rho^1)
   (\rho^1-\rho^2)\rho^2\mathrm{d}y\mathrm{d}s\\
   \leq&|\kappa_1|\|\rho^2\|_{L^\infty_{[0,T]}W^{1,p'}_{\mathbb{T}^2}}\int^t_0\int_{\mathbb{T}^2}
   |\rho^1-\rho^2|\mathrm{d}y\mathrm{d}s. 
\end{align*}
Furthermore, we have 
\begin{align*}
&\lim_{\delta \rightarrow 0} \kappa_1\int_{0}^{t} \int_{\mathbb{T}^{2}}\left\{\operatorname{sgn}^{\delta}\left(\rho^{2}-\rho^{1}\right)\left[\int_{\rho^{2}}^{\rho^{1}} \zeta^{M}(\xi) \mathrm{d}\xi\right] \|\rho_0\|_{L^1_{\mathbb{T}^2}}\right\} \mathrm{d}y \mathrm{d}s\\
\leq&C(\rho_0)\int^t_0\int_{\mathbb{T}^2}
   |\rho^1-\rho^2|\mathrm{d}y\mathrm{d}s. 
\end{align*}

Finally, we deal with the first term of the last equality in  \eqref{K11 J}. For every $i\in\{1,2\}$, \cite[Lemma 3.3]{WWZ22} shows that the weak derivative $\nabla\rho^i=2\sqrt{\rho^{i}}\nabla\sqrt{\rho^{i}}$ holds for almost every $(x,t)\in\mathbb{T}^2\times[0,T]$, it gives that
\begin{align}\label{K11 J 2}
 \int_{0}^{t} \int_{\mathbb{T}^2}\left\{\nabla \operatorname{sgn}^{\delta}\left(\rho^{2}-\rho^{1}\right)\left(\int_{\rho^{2}}^{\rho^{1}} \zeta^{M}(\xi) \mathrm{d}\xi\right) \mathcal{V}[\rho^{1}]\right\} \mathrm{d}y \mathrm{d}s
   =K_3+K_4,
\end{align}
where
\begin{equation}\label{WS K3}
  K_3=4\int^t_0\int_{\mathbb{T}^2}\kappa^{\delta}(\rho^2-\rho^1)
   \sqrt{\rho^2}(\nabla\sqrt{\rho^2}-
   \nabla\sqrt{\rho^1})
   \Big(\int^{\rho^2}_{\rho^1}\zeta^M(\xi)\mathrm{d}\xi\Big)
   \mathcal{V}[\rho^1]\mathrm{d}y\mathrm{d}s,
\end{equation}
and
\begin{equation}\label{WS K4}
  K_4=4\int^t_0\int_{\mathbb{T}^2}\kappa^{\delta}(\rho^2-\rho^1)
   (\sqrt{\rho^2}-\sqrt{\rho^1}
   )\nabla\sqrt{\rho^1}
   \Big(\int^{\rho^2}_{\rho^1}\zeta^M(\xi)\mathrm{d}\xi\Big)
   \mathcal{V}[\rho^1]\mathrm{d}y\mathrm{d}s.
\end{equation}

With the help of the assumption condition that $\rho^2\in\mathcal{C}_0$, we know that $\rho^2\in L^\infty_{[0,T]}L^\infty_{\mathbb{T}^2}$. For every $M\in(0,\infty)$ and $\delta\in(0,1)$, it then follows from the definition of $\kappa^{\delta}$ and $\zeta^M$, and  Young's inequality that
  \begin{align*}
  K_3\leq&4\|\rho^2\|_{L^\infty([0,T];L^\infty(\mathbb{T}^2))}^{\frac{1}{2}}
   \int^t_0\int_{\mathbb{T}^2}\kappa^{\delta}(\rho^2-\rho^1)
   |\nabla\sqrt{\rho^2}-\nabla\sqrt{\rho^1}||\rho^2-{\rho^1}|
   |\mathcal{V}[\rho^1]|\mathrm{d}y\mathrm{d}s\\
   \leq&4\|\rho^2\|_{L^\infty([0,T];L^\infty(\mathbb{T}^2))}^{\frac{1}{2}}\int^t_0\int_{\mathbb{T}^2}
  I_{\{0<|\rho^2-\rho^1|\leq\delta\}} |\nabla\sqrt{\rho^2}||V\ast\rho^1|+|\nabla\sqrt{\rho^1}||\mathcal{V}[\rho^1]|\mathrm{d}y\mathrm{d}s\\
  \leq & 2 \|\rho^2\|_{L^\infty([0,T];L^\infty(\mathbb{T}^2))}^{\frac{1}{2}}\int^t_0\int_{I_{\{0<|\rho^2-\rho^1|\leq\delta\}}}
   2|\mathcal{V}[\rho^1]|^2+ |\nabla\sqrt{\rho^2}|^2+|\nabla\sqrt{\rho^1}|^2\mathrm{d}y\mathrm{d}s.
\end{align*}
Since $\mathcal{V}[\rho^1]$ is  $L^2_{[0,T]}L^2_{\mathbb{T}^2}(\mathbb{R}^2)$-integrable, combining the result that, for every $i\in\{1,2\}$,  $\sqrt{\rho^i}\in L^2_{[0,T]}H^1_{\mathbb{T}^2}$, taking $\delta\rightarrow 0$ in \eqref{K3}, it gives that
\begin{equation}\label{K3}
  \lim_{\delta\rightarrow 0}K_3=0. 
\end{equation}
Similarly, for every $M\in(0,\infty)$ and $\delta\in(0,1)$, by the definition of $\zeta^M$ and $\kappa^{\delta}$, the $1/2$-H\"older continuous property of $\sqrt{\cdot}$ function and Young's inequality, we get there exists a constant $c$ such that
\begin{align*}
  K_4=&2\int^t_0\int_{\mathbb{T}^2}\kappa^{\delta}(\rho^2-\rho^1)
   (\sqrt{\rho^2}-\sqrt{\rho^1}
   )\nabla\sqrt{\rho^1}
   \Big(\int^{\rho^2}_{\rho^1}\zeta^M(\xi)\mathrm{d}\xi\Big)
   \mathcal{V}[\rho^1]\mathrm{d}y\mathrm{d}s\\
   \leq&2\int^t_0\int_{\mathbb{T}^2}\kappa^{\delta}(\rho^2-\rho^1)
   |\sqrt{\rho^2}-\sqrt{\rho^1}
   ||\nabla\sqrt{\rho^1}|
   |\rho^2-\rho^1|
   |\mathcal{V}[\rho^1]|\mathrm{d}y\mathrm{d}s\\
   \leq&2c\delta^{\frac{1}{2}}\int^t_0\int_{\mathbb{T}^2}I_{\{0<|\rho^2-\rho^1|\leq\delta\}}
   |\nabla\sqrt{\rho^1}||\mathcal{V}[\rho^1]|\mathrm{d}y\mathrm{d}s\\
   \leq&c\delta^{\frac{1}{2}}\int^t_0\int_{I_{\{0<|\rho^2-\rho^1|\leq\delta\}}}
   |\nabla\sqrt{\rho^1}|^2+|\mathcal{V}[\rho^1]|^2\mathrm{d}y\mathrm{d}s,
\end{align*}
taking $\delta\rightarrow 0$ in \eqref{K3}, it gives that
\begin{equation}\label{K4}
  \lim_{\delta\rightarrow 0}K_4=0.
\end{equation}

In combination  \eqref{WS K M}, \eqref{G1 2}, \eqref{G3-1},  \eqref{K11 J}, \eqref{K11 J 1}, \eqref{K11 J 2}, \eqref{K3} and \eqref{K4} prove that, for every $t\in[0,T]$,
\begin{align}\label{V-2}
  \lim_{M\rightarrow\infty}\lim_{\delta\rightarrow 0}( \lim_{\varepsilon\rightarrow 0}\mathbb{I}^{\delta,M}_{t,ker})
  \leq &C\left(\|V\|_{L^p_{\mathbb{T}^2}(\mathbb{R}^2)},\|\rho^2\|_{L^\infty_{[0,T]}W^{1,p'}_{\mathbb{T}^2}}\right)\cdot\int^t_0\int_{\mathbb{T}^2}|\rho^1-\rho^2|\mathrm{d}y\mathrm{d}s.
\end{align}

\textbf{Conclusion.}
It follows from properties of kinetic functions that
\begin{equation}
\int_{\mathbb{R}}|\chi^1-\chi^2|^2\mathrm{d}\xi = |\rho^1(x,t)-\rho^2(x,t)|.
\end{equation}
In combination with \eqref{qq-28}, \eqref{rr-11}, we find that for every $t\in[0,T]$,

\begin{align}\label{r-13}
  \|\rho^1(t)-\rho^2(t)\|_{L^1_{\mathbb{T}^2}}
\leq&\|\rho^1_0-\rho^2_0\|_{L^1_{\mathbb{T}^2}}+
\|\rho^2\|_{L^\infty_{[0,T]}W^{1,p'}_{\mathbb{T}^2}}\|V\|_{L^p_{\mathbb{T}^2}(\mathbb{R}^2)}\int_{0}^{t}\|\rho^1-\rho^2\|_{L^{1}_{\mathbb{T}^2}}\mathrm{d}s.
\end{align}
Applying Gronwall's inequality, it gives that
\begin{align*}
\|\rho^1(t)-\rho^2(t)\|_{L^1_{\mathbb{T}^2}}\leq\|\rho^1_0-\rho^2_0\|_{L^1_{\mathbb{T}^2}}  \exp\Big\{C\left(\|V\|_{L^p_{\mathbb{T}^2}(\mathbb{R}^2)},\|\rho^2\|_{L^\infty_{[0,T]}W^{1,p'}_{\mathbb{T}^2}}\right)\Big\}. 
\end{align*}
This completes the proof of \eqref{qq-31-1}.  
\end{proof}

\section{Upper bound for large deviations}\label{sec:LDP-UB} In this section, the upper bound of large deviations is obtained by the exponential martingale approach. We recall from \eqref{I0-intro} and \eqref{Efin-intro} that, if we define $\Lambda:\mathcal{E}_{\mathrm{fin}}\times C^{\infty}([0,T]\times\mathbb{T}^2)\rightarrow \mathbb{R}$ be defined by
\begin{equation}\label{Lambda}
\Lambda(\rho,\varphi)=F_{\rho}(\varphi)-\frac{1}{2} \left\langle \sqrt{\rho}\nabla \varphi,\sqrt{\rho}\nabla\varphi\right\rangle_{L_{[0,T]}^2L^2_{\mathbb{T}^2}},
\end{equation}
where $F_{\rho}(\varphi)$ is defined by 
\begin{align}\label{F}
F_{\rho}(\varphi) := \langle \rho,\varphi \rangle_{L^2_{\mathbb{T}^2}}\Big|^T_0 & - \langle \rho,(\partial_t + \Delta)\varphi\rangle_{L_{[0,T]}^2L^2_{\mathbb{T}^2}} - \langle\rho \mathcal{V}[\rho],\nabla\varphi\rangle_{L_{[0,T]}^2L^2_{\mathbb{T}^2}},
\end{align}
then the rate function $\mathcal{I}$ has the following equivalent representation
\begin{equation}
\mathcal{I}(\rho)=\sup_{\varphi\in C^{\infty}([0,T]\times\mathbb{T}^2)}\Lambda(\rho,\varphi).
\end{equation}

\begin{lemma}[L.S.C. of Rate Function]\label{Iup L-S-C}
Let $\mathcal{V}[\cdot]$ be defined by \eqref{ker-V} such that \eqref{eq:small-KS} holds. For each fixed $\varphi \in C^{\infty}([0,T]\times\mathbb{T}^2)$, $\Lambda(\cdot,\varphi)$ as defined in \eqref{Lambda} is continuous on $L^1([0,T] \times \mathbb{T}^2)$, and consequently, the rate function $\mathcal{I}$ is lower semi-continuous  with respect to the strong $L^1$ topology, namely, for every $(\rho_n)_{n\geq1}$ and $\rho$ in $\mathcal{E}_{\mathrm{fin}}$, 
  \begin{equation}\label{IUP LSC}
\lim_{n \to \infty}\rho_n = \rho \ \text{in} \ L^1_{[0,T]}L^1_{\mathbb{T}^2} \quad \Rightarrow \quad \liminf_{n \to \infty}\mathcal{I}(\rho_n)\geq \mathcal{I}(\rho).
  \end{equation}
\end{lemma}
\begin{proof}
    If $\liminf_{\rho_n\rightarrow \rho}\mathcal{I}(\rho_n)=+\infty$, then \eqref{IUP LSC} is trivial. Therefore, we only consider the case that
\begin{align*}
\liminf_{n \to \infty}\mathcal{I}(\rho_n)<\infty, 
\end{align*}
which implies that there exists a subsequence $\left\{\rho_{k}\right\}_{k\in \mathbb{N}} \subset \left\{\rho_{n}\right\}_{n\in \mathbb{N}}$ such that it converges to $\rho$ in  $ L^1_{[0,T]}L^1_{\mathbb{T}^2}$ and
\begin{align}\label{limit-subseq-rhonk}
   \lim_{k \rightarrow\infty}\mathcal{I}(\rho_{k})=\liminf_{n \rightarrow \infty}\mathcal{I}(\rho_n). 
\end{align}
According to Lemma \ref{ske-lem} and \eqref{limit-subseq-rhonk}, we find a $\Psi^{k}$ for each $k \in \mathbb{N}$ such that $\rho_{k}$ is a weak solution to the skeleton equation
\begin{align*}
\partial_t\rho_{k}=\Delta\rho_{k}-\nabla\cdot(\rho_{k} \mathcal{V}[\rho_{k}])-\nabla\cdot(\rho_{k}\nabla\Psi^{k}), \quad \text{with} \ \sup_{k \in \mathbb{N}}\mathcal{I}(\rho_{k})=\sup_{k \in \mathbb{N}}\frac{1}{2}\|\sqrt{\rho_{k}}\nabla\Psi^{k}\|_{L^2_{[0,T]}L^2_{\mathbb{T}^2}}^2 < \infty.
\end{align*}
With the same argument as used in Proposition \ref{entropy estimate}, we derive the following entropy estimate
\begin{align}\label{SK EE}
\int_{\mathbb{T}^2}\Psi(\rho_k)\mathrm{d}x \Big|_0^T+2\int_0^T\int_{\mathbb{T}^2}|\nabla\sqrt{\rho_k}|^2{d}x{d}t
\lesssim \frac{1}{2}\|\sqrt{\rho_{k}}\nabla\Psi^{k}\|_{L^2_{[0,T]}L^2_{\mathbb{T}^2}}^2,
\end{align}
which implies that the $L^2$ norms of the Fisher information is bounded uniformly in $k$. It follows that $\{\rho_k\}_{k \in \mathbb{N}}$ in bounded in $L^2([0,T] \times \mathbb{T}^2)$, and therefore by further taking a subsequence that we still label by $k \in \mathbb{Z}$ with  a slight abuse of notation, $\rho_k \rightharpoonup \rho$ weakly in $L^2([0,T] \times \mathbb{T}^2)$, and in particular $\rho \in L^2([0,T] \times \mathbb{T}^2)$. Besides, we have according to the Gagliardo-Nirenberg inequality that 
\begin{align}\label{eq:Lp'-solution}
\|\rho_k\|_{L^{p'}_{\mathbb{T}^2}} = \|\sqrt{\rho_k}\|^2_{L_{\mathbb{T}^2}^{2p'}} \lesssim  \| \nabla \sqrt{\rho_k}\|^{2/p}_{L_{\mathbb{T}^2}^{2}} \|\sqrt{\rho_k}\|^{2/p'}_{L_{\mathbb{T}^2}^{2}}+C(\rho_0) = \| \nabla \sqrt{\rho_k}\|^{2/p}_{L_{\mathbb{T}^2}^{2}} \|\sqrt{\rho_k}\|^{1/p'}_{L_{\mathbb{T}^2}^{1}}+C(\rho_0),
\end{align}
where $p \in (1,2)$ and $1/p + 1/p' = 1$. For every $\varphi\in C^{\infty}([0,T]\times\mathbb{T}^2)$, in order to show the continuity of $\Lambda(\cdot,\varphi)$, it is sufficient to prove that
\begin{align}
 \lim_{k \to \infty}\left\langle (\rho_{k} - \rho) \mathcal{V}[\rho],\nabla\varphi \right\rangle = 0, \label{eq:conv-interaction1}\\
  \lim_{k \to \infty}\left\langle  \rho_{k} \mathcal{V}[\rho_{k} - \rho],\nabla\varphi \right\rangle = 0.\label{eq:conv-interaction2}
\end{align}
As $\mathcal{V}[\rho] \in L^2_{[0,T] \times \mathbb{T}^2}$, \eqref{eq:conv-interaction1} directly follows from weak convergence of $\rho_k$, and \eqref{eq:conv-interaction2} is due to \eqref{eq:Lp'-solution} and the $L^1$ conservation, since
\begin{align*}
\left\|\rho_{k} \mathcal{V}[\rho_{k} - \rho]\right\|_{L^1([0,T] \times \mathbb{T}^2)}\leq &\int_{0}^{T}\|\rho_{k}\|_{L^{p'}_{\mathbb{T}^2}} \left\|\nabla \mathcal{G}\ast (\rho_{k}-\rho) \right\|_{L^p_{\mathbb{T}^2}}\mathrm{d}s\\
& \lesssim \|\rho_{k}\|_{L^\infty_{[0,T]}L^1_{\mathbb{T}^2}}^{\frac{1}{p'}}
\int_{0}^{T}\left(\|\nabla\sqrt{\rho_{k}}\|_{L^2_{\mathbb{T}^2}}^{\frac{2}{p}}+C(\rho_0)\right)
  \|\rho_{k}-\rho \|_{L^1_{\mathbb{T}^2}}\mathrm{d}s,
\end{align*}
where we use \eqref{eq:Lp'-solution} and $\nabla \mathcal{G} \in L^p$ for any $p < 2$ in the last inequality, and 
\begin{align*}
\int_{0}^{T}\left(\|\nabla\sqrt{\rho_{k}}\|_{L^2_{\mathbb{T}^2}}^{\frac{2}{p}}+C(\rho_0)\right)
 \|\rho_{k}-\rho \|_{L^1_{\mathbb{T}^2}}\mathrm{d}s  \leq& \left(\|\nabla\sqrt{\rho_{k}}\|_{L^2_{[0,T]}L^2_{\mathbb{T}^2}}^{\frac{2}{p}}+C(\rho_0,T)\right)
\left(\int_{0}^{T}\|\rho_{k} - \rho\|_{L^1_{\mathbb{T}^2}}^{\frac{p}{p-1}}\mathrm{d}s\right)^{\frac{p-1}{p}}\\
 \lesssim &
\|\rho_{k}-\rho \|_{L^1_{[0,T]}L^1_{\mathbb{T}^2}}^{\frac{p-1}{p}},
\end{align*}
for any $p \in (1,2)$, where we use the $L^1$ conservation of $\{\rho_k\}$ and $\rho$ in the last step.  Consequently, $\Lambda(\rho,\varphi)$ is a continuous function of $\rho$ with respect to the strong topology on $L^1_{[0,T] \times \mathbb{T}^2}$ and therefore, $\mathcal{I}$ is lower semi--continuous. 
\end{proof}

The following upper bound of large deviations will be proved by using an exponential martingale method and the exponential tightness Lemma \ref{exponential-tight}. 

\begin{proof}[Proof of \eqref{LDP LB} in Theorem \ref{thm:LDP}]
    For every $t\in[0,T]$, $\varphi\in C^{\infty}([0,T]\times\mathbb{T}^2)$, we define $\mathcal{N}^{\varphi}:[0,T]\times \mathcal{E}_{\mathrm{fin}}$ as
\begin{align*}
\mathcal{N}_t^{\varphi}(\rho) = \langle \rho,\varphi \rangle_{L^2_{\mathbb{T}^2}} \Big|_0^t-\left\langle \rho,(\partial_t + \Delta)\varphi\right\rangle_{L^2_{[0,T]}L^2_{\mathbb{T}^2}} - \left\langle\rho \mathcal{V}[\rho],\nabla\varphi \right\rangle_{L^2_{[0,T]}L^2_{\mathbb{T}^2}}.
\end{align*}
According to the definition of $\mathcal{N}^{\varphi}$ and $\mu^{\varepsilon}$, we obtain that
 $\mathcal{N}_{\cdot}^{\varphi}(\rho^{\varepsilon})$ is a $\mu^{\varepsilon}$--martingale with quadratic variation
\begin{align*}
\langle\mathcal{N}^{\varphi}\rangle(t,\rho^{\varepsilon})
\leq\varepsilon \int_0^t\|\sqrt{\rho^{\varepsilon}}\nabla\varphi\|_{L^2_{\mathbb{T}^2}}^2\mathrm{d}s,
\end{align*}
and we further denote $\mathcal{Q}^{\varphi}$ as the exponential martingale with respect to $\mathcal{N}^{\varphi}$, i.e. 
\begin{equation*}
\mathcal{Q}_t^{\varphi}(\rho^{\varepsilon})
:=\exp\Big\{\mathcal{N}_t^{\varphi}(\rho^{\varepsilon})
-\frac{1}{2}\langle\mathcal{N}^{\varphi}\rangle(t,\rho^{\varepsilon})\Big\},
\end{equation*}
and then for every compact set $F$ of $L^1([0,T] \times \mathbb{T}^2)$, we can then bound its probability from above by
\begin{align*}
\mu^{\varepsilon}(F) \leq \sup_{\rho\in F} \left[Q_T^{\varphi}(\rho)\right]^{-1}\int_{F}Q_T^{\varphi}(\rho)\mu^{\varepsilon}(\mathrm{d}\rho)\leq \exp\Big\{-\inf_{\rho\in F}\Big(\mathcal{N}_t^{\varphi}(\rho)-\frac{\varepsilon }{2}\int_0^t\|\sqrt{\rho}\nabla\varphi\|_{L^2_{\mathbb{T}^2}}^2\mathrm{d}s\Big)\Big\},
\end{align*}
where we use the fact that $\mathcal{Q}^{\varphi}(\rho^{\varepsilon})$ is a non-negative $\mu^{\varepsilon}$--martingale in the last inequality, which can be rewritten as
\begin{align*}
\varepsilon \log\mu^{\varepsilon}(F) \leq -\inf_{\rho\in F}\Big(\mathcal{N}_t^{\varepsilon \varphi}(\rho)-\frac{1}{2}\int_0^t\|\sqrt{\rho}\nabla(\varepsilon\varphi)\|_{L^2_{\mathbb{T}^2}}^2\mathrm{d}s\Big)
\end{align*}
Consequently, taking infimum with respect to $\varphi\in C^{\infty}([0,T]\times\mathbb{T}^2)$ and applying the min--max Lemma in \cite[Lemma 3.2]{KL}, we get from the construction of $\mathcal{N}^{\varphi}$ that
\begin{align*}
\limsup_{\varepsilon\rightarrow 0}\varepsilon \log\mu^{\varepsilon}(F)
\leq - \inf_{\rho\in F}\sup_{\varphi\in C^{\infty}([0,T]\times\mathbb{T}^2)}\Big\{\langle \rho,\varphi \rangle_{L^2_{\mathbb{T}^2}} \Big|_0^t &- \left\langle \rho,(\partial_t + \Delta)\varphi\right\rangle_{L^2_{[0,T]}L^2_{\mathbb{T}^2}} - \left\langle\rho \mathcal{V}[\rho],\nabla\varphi \right\rangle_{L^2_{[0,T]}L^2_{\mathbb{T}^2}}\\
& - \frac{1}{2}\int_0^t\|\sqrt{\rho}\nabla\varphi\|_{L^2_{\mathbb{T}^2}}^2\mathrm{d}s\Big\},
\end{align*}
Recalling the definition of the rate function from \eqref{I0-intro}, we conclude that for every compact set $F\subset L^1_{[0,T]\times \mathbb{T}^2}$, 
\begin{align*}
&\limsup_{\varepsilon\rightarrow0}\varepsilon \log\mu^{\varepsilon}(F)\leq-\inf_{\rho\in F}\mathcal{I}(\rho). 
\end{align*}
The upper bound of LDP \eqref{LDP LB} then holds for every closed set of $L^1_{[0,T]\times \mathbb{T}^2}$ due to the exponential tightness of $\mu^{\varepsilon}$ established in Proposition \ref{exponential-tight}. 
\end{proof}

\noindent{\bf  Acknowledgements}\quad The third author acknowledges the support by the US Army Research Office, grant W911NF2310230.

\bibliographystyle{alphaurl}
\bibliography{Ji-Sun-Wu.bib}

\newcommand{\etalchar}[1]{$^{#1}$}
\begin{thebibliography}{BDSG{\etalchar{+}}15}

\bibitem[BDM11]{BDM11}
Amarjit Budhiraja, Paul Dupuis, and Vasileios Maroulas.
\newblock Variational representations for continuous time processes.
\newblock {\em Ann. Inst. Henri Poincar\'{e} Probab. Stat.}, 47(3):725--747,
  2011.
\newblock \href {https://doi.org/10.1214/10-AIHP382}
  {\path{doi:10.1214/10-AIHP382}}.

\bibitem[BDSG{\etalchar{+}}15]{BDGJL}
Lorenzo Bertini, Alberto De~Sole, Davide Gabrielli, Giovanni Jona-Lasinio, and
  Claudio Landim.
\newblock Macroscopic fluctuation theory.
\newblock {\em Rev. Modern Phys.}, 87(2):593--636, 2015.
\newblock \href {https://doi.org/10.1103/RevModPhys.87.593}
  {\path{doi:10.1103/RevModPhys.87.593}}.

\bibitem[BFM16]{BrzezniakFlandoliMaurelli2016Euler}
Zdzis{\l}aw Brze{\'z}niak, Franco Flandoli, and Mario Maurelli.
\newblock Existence and uniqueness for stochastic 2d euler flows with bounded
  vorticity.
\newblock {\em Archive for Rational Mechanics and Analysis}, 221(1):107--142,
  2016.
\newblock \href {https://doi.org/10.1007/s00205-015-0957-8}
  {\path{doi:10.1007/s00205-015-0957-8}}.

\bibitem[BJW19]{BJW19}
Didier Bresch, Pierre-Emmanuel Jabin, and Zhenfu Wang.
\newblock On mean-field limits and quantitative estimates with a large class of
  singular kernels: application to the {P}atlak-{K}eller-{S}egel model.
\newblock {\em C. R. Math. Acad. Sci. Paris}, 357(9):708--720, 2019.
\newblock \href {https://doi.org/10.1016/j.crma.2019.09.007}
  {\path{doi:10.1016/j.crma.2019.09.007}}.

\bibitem[CD16]{CD16}
Jurandir Ceccon and Carlos~E. Dur{\'a}n.
\newblock Sharp constants in {R}iemannian {$L^p$}-{G}agliardo-{N}irenberg
  inequalities.
\newblock {\em J. Math. Anal. Appl.}, 433(1):260--281, 2016.
\newblock \href {https://doi.org/10.1016/j.jmaa.2015.07.023}
  {\path{doi:10.1016/j.jmaa.2015.07.023}}.

\bibitem[CF23]{CF23arma}
Federico Cornalba and Julian~L. Fischer.
\newblock The dean--kawasaki equation and the structure of density fluctuations
  in systems of diffusing particles.
\newblock {\em Archive for Rational Mechanics and Analysis}, 247(5):76, 2023.
\newblock \href {https://doi.org/10.1007/s00205-023-01903-7}
  {\path{doi:10.1007/s00205-023-01903-7}}.

\bibitem[CF25]{CF23}
Andrea Clini and Benjamin Fehrman.
\newblock A central limit theorem for nonlinear conservative {SPDE}s.
\newblock {\em Stochastics and Partial Differential Equations: Analysis and
  Computations}, 13:1407--1450, 2025.
\newblock \href {https://doi.org/10.1007/s40072-025-00359-y}
  {\path{doi:10.1007/s40072-025-00359-y}}.

\bibitem[CFIR26]{CFIR26}
Federico Cornalba, Julian Fischer, Jonas Ingmanns, and Claudia Raithel.
\newblock Density fluctuations in weakly interacting particle systems via the
  dean--kawasaki equation.
\newblock {\em The Annals of Probability}, 54(1):155--215, 2026.
\newblock \href {https://doi.org/10.1214/25-AOP1763}
  {\path{doi:10.1214/25-AOP1763}}.

\bibitem[DE97]{DE97}
Paul Dupuis and Richard~S. Ellis.
\newblock {\em A weak convergence approach to the theory of large deviations}.
\newblock Wiley Series in Probability and Statistics: Probability and
  Statistics. John Wiley \& Sons, Inc., New York, 1997.
\newblock A Wiley-Interscience Publication.
\newblock \href {https://doi.org/10.1002/9781118165904}
  {\path{doi:10.1002/9781118165904}}.

\bibitem[Dea96]{D96}
David~S. Dean.
\newblock Langevin equation for the density of a system of interacting
  {L}angevin processes.
\newblock {\em J. Phys. A}, 29(24):L613--L617, 1996.
\newblock \href {https://doi.org/10.1088/0305-4470/29/24/001}
  {\path{doi:10.1088/0305-4470/29/24/001}}.

\bibitem[DFG20]{DFG20}
Nicolas Dirr, Benjamin~J. Fehrman, and Benjamin Gess.
\newblock Conservative stochastic pde and fluctuations of the symmetric simple
  exclusion process.
\newblock {\em arXiv: Probability}, 2020.

\bibitem[DFG26]{DFG}
Nicolas Dirr, Benjamin Fehrman, and Benjamin Gess.
\newblock Conservative stochastic {PDE} and fluctuations of the symmetric
  simple exclusion process.
\newblock {\em Communications in Mathematical Physics}, 407:Paper No. 74, 2026.
\newblock \href {https://doi.org/10.1007/s00220-026-05587-4}
  {\path{doi:10.1007/s00220-026-05587-4}}.

\bibitem[DG20]{DG20}
Konstantinos Dareiotis and Benjamin Gess.
\newblock Nonlinear diffusion equations with nonlinear gradient noise.
\newblock {\em Electron. J. Probab.}, 25:Paper No. 35, 43, 2020.
\newblock \href {https://doi.org/10.1214/20-ejp436}
  {\path{doi:10.1214/20-ejp436}}.

\bibitem[DJP25]{DJP25}
Ana Djurdjevac, Xiaohao Ji, and Nicolas Perkowski.
\newblock Weak error of dean--kawasaki equation with smooth mean-field
  interactions, 2025.
\newblock \href {http://arxiv.org/abs/2502.20929} {\path{arXiv:2502.20929}}.

\bibitem[DKP24]{DKP24}
Ana Djurdjevac, Helena Kremp, and Nicolas Perkowski.
\newblock Weak error analysis for a nonlinear spde approximation of the
  dean--kawasaki equation.
\newblock {\em Stochastics and Partial Differential Equations: Analysis and
  Computations}, 12:2330--2355, 2024.
\newblock \href {https://doi.org/10.1007/s40072-024-00324-1}
  {\path{doi:10.1007/s40072-024-00324-1}}.

\bibitem[Due16]{D16}
Mitia Duerinckx.
\newblock Mean-field limits for some {R}iesz interaction gradient flows.
\newblock {\em SIAM J. Math. Anal.}, 48(3):2269--2300, 2016.
\newblock \href {https://doi.org/10.1137/15M1042620}
  {\path{doi:10.1137/15M1042620}}.

\bibitem[Feh25]{fehrman2025stochastic}
Benjamin Fehrman.
\newblock Stochastic pdes with correlated, non-stationary stratonovich noise of
  dean--kawasaki type, 2025.
\newblock \href {http://arxiv.org/abs/2504.18370} {\path{arXiv:2504.18370}}.

\bibitem[FG95]{FG95}
Franco Flandoli and Dariusz Gatarek.
\newblock Martingale and stationary solutions for stochastic navier-stokes
  equations.
\newblock {\em Probab. Theory Related Fields}, 102(3):367--391, 1995.
\newblock \href {https://doi.org/10.1007/BF01192467}
  {\path{doi:10.1007/BF01192467}}.

\bibitem[FG16]{FG16}
Peter~K. Friz and Benjamin Gess.
\newblock Stochastic scalar conservation laws driven by rough paths.
\newblock {\em Ann. Inst. H. Poincar\'{e} C Anal. Non Lin\'{e}aire},
  33(4):933--963, 2016.
\newblock \href {https://doi.org/10.1016/j.anihpc.2015.01.009}
  {\path{doi:10.1016/j.anihpc.2015.01.009}}.

\bibitem[FG19]{FG19}
Benjamin Fehrman and Benjamin Gess.
\newblock Well-posedness of nonlinear diffusion equations with nonlinear,
  conservative noise.
\newblock {\em Arch. Ration. Mech. Anal.}, 233(1):249--322, 2019.
\newblock \href {https://doi.org/10.1007/s00205-019-01357-w}
  {\path{doi:10.1007/s00205-019-01357-w}}.

\bibitem[FG23]{FG23}
Benjamin Fehrman and Benjamin Gess.
\newblock Non-equilibrium large deviations and parabolic-hyperbolic {PDE} with
  irregular drift.
\newblock {\em Invent. Math.}, 234(2):573--636, 2023.
\newblock \href {https://doi.org/10.1007/s00222-023-01207-3}
  {\path{doi:10.1007/s00222-023-01207-3}}.

\bibitem[FG24]{FG24}
Benjamin Fehrman and Benjamin Gess.
\newblock Well-{P}osedness of the {D}ean--{K}awasaki and the {N}onlinear
  {D}awson--{W}atanabe {E}quation with {C}orrelated {N}oise.
\newblock {\em Arch. Ration. Mech. Anal.}, 248(2):Paper No. 20, 2024.
\newblock \href {https://doi.org/10.1007/s00205-024-01963-3}
  {\path{doi:10.1007/s00205-024-01963-3}}.

\bibitem[FG25]{FG25}
Benjamin~J. Fehrman and Benjamin Gess.
\newblock Conservative stochastic pdes on the whole space.
\newblock {\em Stochastics and Partial Differential Equations: Analysis and
  Computations}, 2025.
\newblock \href {https://doi.org/10.1007/s40072-025-00369-w}
  {\path{doi:10.1007/s40072-025-00369-w}}.

\bibitem[FGG22]{fehrman2022ergodicity}
Benjamin Fehrman, Benjamin Gess, and Rishabh~S. Gvalani.
\newblock Ergodicity and random dynamical systems for conservative {SPDE}s,
  2022.
\newblock \href {http://arxiv.org/abs/2206.14789} {\path{arXiv:2206.14789}}.

\bibitem[GH23]{GH23}
Benjamin Gess and Daniel Heydecker.
\newblock A rescaled zero-range process for the porous medium equation:
  Hydrodynamic limit, large deviations and gradient flow, 2023.
\newblock \href {http://arxiv.org/abs/2303.11289} {\path{arXiv:2303.11289}}.

\bibitem[GHW24]{GHW23}
Benjamin Gess, Daniel Heydecker, and Zhengyan Wu.
\newblock Landau--lifshitz--navier--stokes equations: Large deviations and
  relationship to the energy equality, 2024.
\newblock \href {http://arxiv.org/abs/2311.02223} {\path{arXiv:2311.02223}}.

\bibitem[GS15]{GS15}
Benjamin Gess and Panagiotis~E. Souganidis.
\newblock Scalar conservation laws with multiple rough fluxes.
\newblock {\em Commun. Math. Sci.}, 13(6):1569--1597, 2015.
\newblock \href {https://doi.org/10.4310/CMS.2015.v13.n6.a10}
  {\path{doi:10.4310/CMS.2015.v13.n6.a10}}.

\bibitem[GWZ25]{GWZ24}
Benjamin Gess, Zhengyan Wu, and Rangrang Zhang.
\newblock Higher order fluctuation expansions for nonlinear stochastic heat
  equations in singular limits.
\newblock {\em Stochastic Processes and their Applications}, 193:104847, 2025.
\newblock URL:
  \url{https://www.sciencedirect.com/science/article/pii/S0304414925002911},
  \href {https://doi.org/10.1016/j.spa.2025.104847}
  {\path{doi:10.1016/j.spa.2025.104847}}.

\bibitem[Hey23]{Heydecker23}
Daniel Heydecker.
\newblock Large deviations of kac's conservative particle system and energy
  nonconserving solutions to the boltzmann equation: A counterexample to the
  predicted rate function.
\newblock {\em The Annals of Applied Probability}, 33(3):1758--1826, 2023.
\newblock \href {https://doi.org/10.1214/22-AAP1852}
  {\path{doi:10.1214/22-AAP1852}}.

\bibitem[HWZ25]{HWZ25}
Zimo Hao, Zhengyan Wu, and Johannes Zimmer.
\newblock Kinetic theory with fluctuations: Strong well-posedness of the
  vlasov--fokker--planck--dean--kawasaki system, 2025.
\newblock \href {http://arxiv.org/abs/2511.10194} {\path{arXiv:2511.10194}}.

\bibitem[Jak97]{Jak97}
Adam Jakubowski.
\newblock The almost sure {S}korokhod representation for subsequences in
  nonmetric spaces.
\newblock {\em Teor. Veroyatnost. i Primenen.}, 42(1):209--216, 1997.
\newblock \href {https://doi.org/10.1137/S0040585X97976052}
  {\path{doi:10.1137/S0040585X97976052}}.

\bibitem[JW18]{JW18}
Pierre-Emmanuel Jabin and Zhenfu Wang.
\newblock Quantitative estimates of propagation of chaos for stochastic systems
  with {$W^{-1,\infty}$} kernels.
\newblock {\em Invent. Math.}, 214(1):523--591, 2018.
\newblock \href {https://doi.org/10.1007/s00222-018-0808-y}
  {\path{doi:10.1007/s00222-018-0808-y}}.

\bibitem[Kaw98]{K98}
Kyozi Kawasaki.
\newblock Microscopic analyses of the dynamical density functional equation of
  dense fluids.
\newblock {\em J. Statist. Phys.}, 93(3-4):527--546, 1998.
\newblock \href {https://doi.org/10.1023/B:JOSS.0000033240.66359.6c}
  {\path{doi:10.1023/B:JOSS.0000033240.66359.6c}}.

\bibitem[KL99]{KL}
Claude Kipnis and Claudio Landim.
\newblock {\em Scaling limits of interacting particle systems}, volume 320 of
  {\em Grundlehren der mathematischen Wissenschaften [Fundamental Principles of
  Mathematical Sciences]}.
\newblock Springer-Verlag, Berlin, 1999.
\newblock \href {https://doi.org/10.1007/978-3-662-03752-2}
  {\path{doi:10.1007/978-3-662-03752-2}}.

\bibitem[KLvR19]{KLvR19}
Vitalii Konarovskyi, Tobias Lehmann, and Max-K. von Renesse.
\newblock Dean-{K}awasaki dynamics: ill-posedness vs. triviality.
\newblock {\em Electron. Commun. Probab.}, 24:Paper No. 8, 9, 2019.
\newblock \href {https://doi.org/10.1214/19-ECP208}
  {\path{doi:10.1214/19-ECP208}}.

\bibitem[KLvR20]{KLR20}
Vitalii Konarovskyi, Tobias Lehmann, and Max von Renesse.
\newblock On dean--kawasaki dynamics with smooth drift potential.
\newblock {\em Journal of Statistical Physics}, 178(3):666--681, 2020.
\newblock \href {https://doi.org/10.1007/s10955-019-02449-3}
  {\path{doi:10.1007/s10955-019-02449-3}}.

\bibitem[KvR19]{KvR19}
Vitalii Konarovskyi and Max-K. von Renesse.
\newblock Modified massive {A}rratia flow and {W}asserstein diffusion.
\newblock {\em Comm. Pure Appl. Math.}, 72(4):764--800, 2019.
\newblock \href {https://doi.org/10.1002/cpa.21758}
  {\path{doi:10.1002/cpa.21758}}.

\bibitem[LPS13]{LPS13}
Pierre-Louis Lions, Beno\^{\i}t Perthame, and Panagiotis~E. Souganidis.
\newblock Stochastic averaging lemmas for kinetic equations.
\newblock In {\em S\'{e}minaire {L}aurent {S}chwartz---\'{E}quations aux
  d\'{e}riv\'{e}es partielles et applications. {A}nn\'{e}e 2011--2012},
  S\'{e}min. \'{E}qu. D\'{e}riv. Partielles, pages Exp. No. XXVI, 17. \'{E}cole
  Polytech., Palaiseau, 2013.

\bibitem[Mar10]{M10}
Mauro Mariani.
\newblock Large deviations principles for stochastic scalar conservation laws.
\newblock {\em Probab. Theory Related Fields}, 147(3-4):607--648, 2010.
\newblock \href {https://doi.org/10.1007/s00440-009-0218-6}
  {\path{doi:10.1007/s00440-009-0218-6}}.

\bibitem[McK67]{M67}
H.~P. McKean, Jr.
\newblock Propagation of chaos for a class of non-linear parabolic equations.
\newblock In {\em Stochastic {D}ifferential {E}quations ({L}ecture {S}eries in
  {D}ifferential {E}quations, {S}ession 7, {C}atholic {U}niv., 1967)}, pages
  41--57. Air Force Office Sci. Res., Arlington, Va., 1967.

\bibitem[MM24]{AA24}
Adrian Martini and Avi Mayorcas.
\newblock An additive-noise approximation to keller--segel--dean--kawasaki
  dynamics: Small-noise results, 2024.
\newblock \href {http://arxiv.org/abs/2410.17022} {\path{arXiv:2410.17022}}.

\bibitem[MM25]{AA25}
Adrian Martini and Avi Mayorcas.
\newblock An additive-noise approximation to keller--segel--dean--kawasaki
  dynamics: Local well-posedness of paracontrolled solutions.
\newblock {\em Stochastics and Partial Differential Equations: Analysis and
  Computations}, 13(2):956--1033, 2025.
\newblock \href {https://doi.org/10.1007/s40072-024-00343-y}
  {\path{doi:10.1007/s40072-024-00343-y}}.

\bibitem[MvRZ25]{MRZ25}
Fenna M{\"u}ller, Max von Renesse, and Johannes Zimmer.
\newblock Well-posedness for dean--kawasaki models of vlasov--fokker--planck
  type.
\newblock {\em Proceedings of the Royal Society A: Mathematical, Physical and
  Engineering Sciences}, 481(2289):20250089, 2025.
\newblock \href {https://doi.org/10.1098/rspa.2025.0089}
  {\path{doi:10.1098/rspa.2025.0089}}.

\bibitem[Pop25]{Shyam25}
Shyam Popat.
\newblock Well-posedness of the generalised {D}ean-{K}awasaki equation with
  correlated noise on bounded domains.
\newblock {\em Stochastic Process. Appl.}, 179:Paper No. 104503, 32, 2025.
\newblock \href {https://doi.org/10.1016/j.spa.2024.104503}
  {\path{doi:10.1016/j.spa.2024.104503}}.

\bibitem[PW25]{PW25}
Shyam Popat and Zhengyan Wu.
\newblock Ergodicity for the dean--kawasaki equation with dirichlet boundary
  conditions: Taming the square-root, 2025.
\newblock \href {http://arxiv.org/abs/2512.12861} {\path{arXiv:2512.12861}}.

\bibitem[Ren08]{Ren2008BDG}
Yao-Feng Ren.
\newblock On the burkholder--davis--gundy inequalities for continuous
  martingales.
\newblock {\em Statistics \& Probability Letters}, 78(17):3034--3039, 2008.
\newblock \href {https://doi.org/10.1016/j.spl.2008.05.024}
  {\path{doi:10.1016/j.spl.2008.05.024}}.

\bibitem[Sch22]{D22}
Lorenzo~Dello Schiavo.
\newblock The dirichlet--ferguson diffusion on the space of probability
  measures over a closed riemannian manifold.
\newblock {\em The Annals of Probability}, 50(2):591--648, 2022.
\newblock \href {https://doi.org/10.1214/21-AOP1541}
  {\path{doi:10.1214/21-AOP1541}}.

\bibitem[Ser20]{S20}
Sylvia Serfaty.
\newblock Mean field limit for {C}oulomb-type flows.
\newblock {\em Duke Math. J.}, 169(15):2887--2935, 2020.
\newblock With an appendix by Mitia Duerinckx and Serfaty.
\newblock \href {https://doi.org/10.1215/00127094-2020-0019}
  {\path{doi:10.1215/00127094-2020-0019}}.

\bibitem[Sim87]{Sim87}
Jacques Simon.
\newblock Compact sets in the space {$L^p(0,T;B)$}.
\newblock {\em Ann. Mat. Pura Appl. (4)}, 146:65--96, 1987.
\newblock \href {https://doi.org/10.1007/BF01762360}
  {\path{doi:10.1007/BF01762360}}.

\bibitem[WW25]{WW25}
Lin Wang and Zhengyan Wu.
\newblock Probabilistic approaches to the energy equality in forced surface
  quasi-geostrophic equations.
\newblock {\em Stochastics and Partial Differential Equations: Analysis and
  Computations}, dec 2025.
\newblock URL: \url{https://link.springer.com/10.1007/s40072-025-00405-9},
  \href {https://doi.org/10.1007/s40072-025-00405-9}
  {\path{doi:10.1007/s40072-025-00405-9}}.

\bibitem[WWZ24]{WWZ22}
Likun Wang, Zhengyan Wu, and Rangrang Zhang.
\newblock Dean--kawasaki equation with singular interactions and applications
  to dynamical ising--kac model, 2024.
\newblock \href {http://arxiv.org/abs/2207.12774} {\path{arXiv:2207.12774}}.

\bibitem[WZ22]{WZ22}
Zhengyan Wu and Rangrang Zhang.
\newblock Central limit theorem and moderate deviation principle for stochastic
  scalar conservation laws.
\newblock {\em J. Math. Anal. Appl.}, 516(1):Paper No. 126445, 26, 2022.
\newblock \href {https://doi.org/10.1016/j.jmaa.2022.126445}
  {\path{doi:10.1016/j.jmaa.2022.126445}}.

\bibitem[WZ24]{WZ24}
Zhengyan Wu and Rangrang Zhang.
\newblock Mckean--vlasov {PDE} with irregular drift and applications to large
  deviations for conservative {SPDE}s, 2024.
\newblock \href {http://arxiv.org/abs/2208.13142} {\path{arXiv:2208.13142}}.

\bibitem[WZZ21]{WZZ21}
Zhenfu Wang, Xianliang Zhao, and Rongchan Zhu.
\newblock Gaussian fluctuations for interacting particle systems with singular
  kernels. arxiv: 2105.13201, 2021.
\newblock \href {http://arxiv.org/abs/2105.13201} {\path{arXiv:2105.13201}}.

\end{thebibliography}

\end{document}